\documentclass{amsart}
\usepackage{mathtext}

\usepackage{amscd}
\usepackage{amsfonts}
\usepackage{amsthm,amssymb,latexsym,amsmath}
\usepackage[all]{xy}
\usepackage[english,russian]{babel}
\usepackage[cp1251]{inputenc}
\usepackage{xcolor}

\usepackage{eucal}
\usepackage{amsxtra}
\usepackage[colorlinks,linkcolor=blue,citecolor=blue,
filecolor=blue,urlcolor=blue,breaklinks=true]{hyperref}

\newcommand{\Coh}{\operatorname{Coh}}

\DeclareMathOperator{\rmext}{ext}
\DeclareMathOperator{\Hom}{Hom}

\DeclareMathOperator{\coker}{coker}
\DeclareMathOperator{\im}{im}

\DeclareMathOperator{\rk}{{rk}}

\DeclareMathOperator{\Sing}{Sing}

\newcommand{\into}{\hookrightarrow}
\newcommand{\onto}{\twoheadrightarrow}
\newcommand{\p}[1]{{\mathbb{P}^{#1}}}

\newlength{\rrrr}

\newcommand{\intoo}[1]{\:
\xymatrix@1{\ar@{^(->}[r]^{#1}&}\:}
\newcommand{\ontoo}[1]{\:
\xymatrix@1{\ar@{->>}[r]^{#1}&}\:}
\def\lHom{\mathop{\mathcal Hom}\nolimits}

\def\GL{\mathop{\mathrm{GL}}\nolimits}
\def\Gr{\mathop{\mathrm{Gr}}\nolimits}
\def\lExt{\mathop{\mathcal Ext}\nolimits}

\def\C{\ensuremath{\mathbb{C}}}
\def\P{\ensuremath{\mathbb{P}}}

\def\R{\ensuremath{\mathbb{R}}}
\def\Z{\ensuremath{\mathbb{Z}}}

\def\AA{\ensuremath{\mathcal A}}

\def\EE{\ensuremath{\mathcal E}}
\def\FF{\ensuremath{\mathcal F}}
\def\GG{\ensuremath{\mathcal G}}
\def\HH{\ensuremath{\mathcal H}}
\def\II{\ensuremath{\mathcal I}}
\def\JJ{\ensuremath{\mathcal J}}
\def\KK{\ensuremath{\mathcal K}}
\def\LL{\ensuremath{\mathcal L}}
\def\MM{\ensuremath{\mathcal M}}

\def\OO{\ensuremath{\mathcal O}}

\def\QQ{\ensuremath{\mathcal Q}}
\def\RR{\ensuremath{\mathcal R}}
\def\SS{\ensuremath{\mathcal S}}
\def\TT{\ensuremath{\mathcal T}}
\def\UU{\ensuremath{\mathcal U}}
\def\VV{\ensuremath{\mathcal V}}
\def\WW{\ensuremath{\mathcal W}}

\def\YY{\ensuremath{\mathcal Y}}

\def\bk{\ensuremath{\mathbf k}}

\newcommand\Ext{\operatorname{Ext}\nolimits}

\newtheorem{theorem}{Theorem}[section]

\newtheorem{proposition}{Proposition}[section]
\newtheorem{lemma}{Lemma}[section]

\newtheorem{remark}[theorem]{Remark}
\newtheorem{definition}[theorem]{{\bf Definition}}

\usepackage{float}

\title[Moduli of Sheaves on Fano Threefolds]{Moduli of Rank Two Semistable Sheaves on Rational Fano Threefolds of the Main Series}

\begin{document}

\renewcommand{\abstractname}{Abstract} \renewcommand{\contentsname}{Contents} \renewcommand{\proofname}{Proof}
\renewcommand{\refname}{References}

\maketitle

\begin{center}
\large Alexander~S.~Tikhomirov, Danil~A.~Vassiliev
\end{center}

\begin{abstract}

In this paper we investigate the moduli spaces of semistable coherent
sheaves of rank two on the projective space $\p3$ and the following
rational Fano manifolds of the main series - the three-dimensional quadric
$X_2$, the intersection of two 4-dimensional quadrics $X_4$ and the Fano manifold 
$X_5$ of degree 5. For the quadric $X_2$, the boundedness of the third Chern class $c_3$ of rank two semistable objects in $\mathrm{D}^b(X_2)$, including sheaves, is proved. An explicit description is given of all the moduli spaces of semistable sheaves of rank
two on $X_2$, including reflexive ones, with a maximal third class
$c_3\ge0$. These spaces turn out to be irreducible smooth
rational manifolds in all cases, except for the following
two: $(c_1,c_2,c_3)=(0,2,2)$ or (0,4,8). Several new infinite series of rational
components of the moduli spaces of semistable sheaves of rank two on $\p3$, $X_2$, $X_4$ and $X_5$ are constructed,
as well as a new infinite series of irrational
components on $X_4$. The boundedness of the class $c_3$ is proved for $c_1=0$ and
any $c_2>0$ for stable reflexive sheaves of general type on
manifolds $X_4$ and $X_5$.
\end{abstract}

\tableofcontents

\section{Introduction}
\vspace{5mm}

In 1980, R. Hartshorne, while investigating in \cite{H} the spectra of stable
reflexive coherent sheaves of rank two on the projective space
$\mathbb{P}^3$, proved the boundedness of the third Chern class $c_3$ of these sheaves
for fixed first and second Chern classes $c_1$ and $c_2$.
The exact estimates he obtained for the class $c_3$ have the form (see \cite[Thm. 8.2]{H})
\begin{equation}\label{estimates}
	c_3\le c^2_2-c_2+2,\ \ \ \text{if}\ \ \ c_1=0;\ \ \ \ \ \ \ \ \ \  
	c_3\le c^2_2\ \ \ \text{if}\ \ \ c_1=-1.
\end{equation}
In the same work, the irreducibility, smoothness and rationality of the moduli spaces of such sheaves
with $c_1=-1$, arbitrary $c_2>0$ and maximum $c_3=c^2_2$ are proved.
In 2018, B. Schmidt in \cite{Sch18}, investigating the properties of tilt stability in the derived category $\mathrm{D}^b(\P^3)$,
proved that the estimates \eqref{estimates} are true for all semistable
sheaves of rank two on $\P^3$, and gave an explicit description of their moduli
space for $-1\le c_1\le0$, $c_2>0$ and maximal $c_3$. As a consequence, he
obtained that these spaces are irreducible smooth
rational projective varieties. It is not difficult to see that
the moduli spaces of reflexive sheaves described by Hartshorne
are open subsets of these manifolds (see Theorem 4 below).
We also note that quite recently in the 2023 work \cite{Sch23}, Schmidt
generalized the above results to the case of sheaves on $\P^3$ of all ranks from 0 to 4. 

In this paper, we study the moduli spaces of semistable
rank two sheaves on rational three-dimensional Fano varieties of the main series.
There are four such varieties --- these are the projective space $X_1=\P^3$,
the three-dimensional quadric $X_2$, the complete intersection $X_4$ of two quadrics in
the space $\P^5$, and the section $X_5$ of the grassmannian $\mathrm{Gr}(2,5)$
embedded by Pl\"ucker in the space $\P^9$ by a linear subspace $\P^6$
(see, e. g., \cite{Isk} or \cite{IP}). Here the subscript $i$ of
the variety $X_i$ is its projective degree. Chern classes of bundles $E$
on $X=X_i$ are defined by the integers $c_1,c_2,c_3$ (see formulas \label{c i} below), and the corresponding moduli spaces (Gieseker-Maruyama moduli schemes) of semistable bundles of rank two on $X$ will be denoted by
$M_X(2;c_1,c_2,c_3)$.

The first direction of research in this paper concerns the question of
boundedness of the third Chern class $c_3$ of semistable rank 2 sheaves
on $X$ with fixed $c_1\in\{-1,0\}$ and $c_2\ge0$ and of getting estimates
for the third Chern class $c_3$. Using the tilt stability technique in
the derived category $\mathrm{D}^b(X)$, we give a complete answer to this
question for the three-dimensional quadric $X_2$ in the following theorem (see the statements
(3.1)-(4.2) in Theorem \ref{Theorem 3.1} of the paper).

\vspace{2mm}
\textbf{Theorem 1.}
\textit{(i) Let $E$ be a semistable sheaf of rank 2 with $c_1=-1$ on the quadric $X_2$. Then $c_2\ge0$ and
	$c_3\le\frac12c_2^2$ if $c_2$ is even, and, respectively, $c_3\le\frac12(c_2^2-1)$ if $c_2$ is odd.\\
	(ii) Let $E$ be a semistable sheaf of rank 2 on $X_2$ with $c_1(E)=0$.
	Then $c_2\ge0$ and $c_3\le\frac 12c_2^2$, if $c_2$ is even, and,
	respectively, $c_3\le\frac 12(c_2^2+1)$, if $c_2$ is odd.\\
	These estimates are exact for all $c_3\ge0$.}\\

The proof of this theorem is based on the study of the relationship between
tilt semistability and Bridgland semistability in $\mathrm{D}
^b(X_2)$. The key here is Schmidt's important technical result
(2014) on the description of a subcategory in $\mathrm{D}^b(X_2)$ generated by
a torsion pair in the sense of Bridgeland - see Proposition \ref{heart}
of this paper.

Unfortunately, no analogues of this result are known to date
for varieties $X_4$ and $X_5$. Therefore, for these varieties it is not possible to
use the same method to obtain exact upper bounds for the class $c_3$ for all
semistable sheaves of rank 2 on $X_4$ and $X_5$. However, using more
traditional technique considering the behavior of stable sheaves at
standard birational transformations $X_4\dashrightarrow X_1$ and
$X_5\dashrightarrow X_2$, we give a partial answer to the question about
boundedness of $c_3$ for a sufficiently wide class of sheaves on $X_4$ and
$X_5$.

Namely, we consider stable reflexive sheaves of rank 2 with $c_1=0$,
called in this paper sheaves of general type. These sheaves $E$ on $X_i,\ i=4.5$, are such that $E|_l\cong\OO_{\P^1}^{\oplus2}$ for a general line
$l\subset X_i$, and lines for which either $E|_l\cong\OO_{\P^1}(a)\oplus
\OO_{\P^1}(-a)$ with $a\ge2$, or $E|_l$ is not locally free,
constitute a subset of dimension $\le0$ in the base of the family of lines on
$X_i$ (see definition \ref{def 4.10} of this paper). In Theorem \ref{Theorem 4.4} we give examples of infinite series of components of moduli spaces of
semistable sheaves in which the general sheaf is a reflexive sheaf of
general type. (Presumably the property of being a sheaf of general type is true
for all stable reflexive sheaves of rank 2 with $c_1=0$, that is,
perhaps an analogue of the Grauert-M\"ulich Theorem holds for them,
which is known to be valid for stable reflexive sheaves of rank 2 on $X_1$.) For
sheaves of general type we prove the following theorem (see Theorems
\ref{Thm 6.4} and \ref{Thm 6.1} of this paper).

\vspace{2mm}
\textbf{Theorem 2.}
\textit{Let $E$ be a stable reflexive sheaf of rank 2 of general type with
	Chern classes $c_1=0$, $c_2>0$, $c_3$ on the variety $X_4$ or $X_5$.
	Then the following inequalities are true for the class $c_3$ of the sheaf $E$.\\
	(i) On $X_4$: $c_3\le c_2^2-c_2+2$.\\
	(ii) On $X_5$: $c_3\le\frac 29c_2^2$ if $c_2$ is even, and,
	respectively, $c_3\le\frac 29c_2^2+\frac 12$, if $c_2$ is odd.}\\
Whether these estimates are sharp is an open question.

\vspace{2mm}
The second direction of research in this paper is the construction of
new infinite series (with growing class $c_2$) of moduli components of semistable sheaves of rank two on the varieties $X_1$, $X_2$,
$X_4$ and $X_5$, including an explicit description of general\footnote{Here and everywhere
below under a general sheaf in a flat irreducible family of sheaves on $X$ with
base $B$ we understand a sheaf corresponding to a point in $B$ belonging to a
dense open subset in $B$.} sheaves in these components. First
two infinite series of such components on $X_1=\P^3$ were constructed
R.~Hartshorne in \cite{SVB} in 1978 for vector bundles
(this is a series of instanton components with $c_1=0$ and a parallel series
components with $c_1=-1$) and in the work \cite{H} in 1980 for reflexive
bundles with $c_1=-1$ and maximal $c_3$, and with $c_1=0$ and maximal
spectrum. In 1978, W.~Barth and K.~Hulek \cite{BH} built another
infinite series of families of stable vector bundles of rank 2 on
$\P^3$, and in 1981 G.~Ellingsrud and S.~A.~Str\o mme proved \cite{ESt},
that these families are open subsets of irreducible components of
moduli spaces. In \cite{V1} (1985) and \cite{V2} (1987)
V.~K.~Vedernikov constructed new infinite series of families of stable
vector bundles of rank 2 on $\P^3$, and in 1984 A.~P.~Rao \cite{R}
built a more general series of families, including the Vedernikov series.
L.~Ein independently described these families in 1988 and
proved that they are open subsets of irreducible
components of moduli spaces. In 2019 A. A. Kytmanov, A. S. Tikhomirov
and S. A. Tikhomirov in the paper \cite{KTT} proved the rationality of significant
part of the components of these series.

A large number of new infinite series of components of moduli spaces of
semistable sheaves of rank 2 on $\P^3$ were found in the work of M.~Jardim,
D.~Markushevich and A.~S.~Tikhomirov \cite{JMT} in 2017 and in the work of
C.~Almeida, M.~Jardim and A.~S.~Tikhomirov \cite{AJT} in 2022, and
general sheaves in these components were described. B.~Schmidt in the above work
\cite{Sch18} gave a complete description of all schemes of stable sheaves of
rank 2 with maximal Chern class $c_3$, proving their projectivity,
irreducibility for any admissible values of the Chern classes $c_1$ and
$c_2$, rationality and in almost all cases smoothness. Note that
the last property of smoothness is a rather unexpected phenomenon for
projective Gieseker-Maruyama moduli schemes of semistable sheaves of rank 2
on three-dimensional varieties. 

As for the varieties $X_2$, $X_4$ and $X_5$, by now
on each of them only one infinite series of moduli components of semistable 
sheaves of rank 2 was found. These are the series of components containing as
open sets the families of instanton bundles. Instantonic
bundles on $X_2$ were defined by L.~Costa and R.~M.~Miro-Roig in
\cite{CM} in 2009, and on $X_4$ and $X_5$ and other Fano varieties by
A.~Kuznetsov \cite{K} in 2012 and D.~Faenzi \cite{F} in 2013. In work
\cite{F} D.~Faenzi proved that families of instanton bundles on
$X_2$, $X_4$ and $X_5$ are indeed open subsets of
irreducible components of moduli spaces, which are reduced at a general point
and have the expected dimension. In recent years, an extensive number of works 
were devoted to the study of instanton series of bundles, a review of which can be
found, for example, in \cite{ACG} and \cite{CJMM}.

In this paper we construct several new infinite series of irreducible
rational components of moduli spaces of semistable sheaves of rank 2
on the varieties $X_1$, $X_2$, $X_4$ and $X_5$. We describe general sheaves in
these components and prove their reflexivity, and also find
dimensions of the constructed components. These results are proven in Theorems
\ref{Theorem 4.1}, \ref{Theorem 4.1S}, \ref{Theorem 4.2}, \ref{Theorem
	4.2S} and \ref{Theorem 4.3}. They are collected in the following theorem.

\vspace{2mm}
\textbf{Theorem 3.}
\textit{Let $X$ be one of the varieties $X_1$, $X_2$, $X_4$, $X_5$, and
	let $\OO_X(1)$ be the ample sheaf on $X$ such that $\mathrm{Pic}(X)=
	\Z[\OO_X(1)]$. Consider a sheaf $E$ of rank 2 on $X$ defined by one of
	nontrivial extensions of the form
\begin{equation}\label{eqn 1}
0\to F_i\to E\to G_j\to0,\ \ \ \ \ 1\le i\le3,\ \ \ 1\le j\le2, 
\end{equation}
where $F_1=\OO_X(-n)^{\oplus2}$, $F_2=F(-n)$, where $F$ is a rank 2 sheaf of one of
the types (I)-(III) described in subsection 4.3, $G_1=\OO_S(m)$, where $S\in|
\OO_X(k)|$, and the sheaves $F_3$ and $G_2$ are defined in the case of the quadric $X=X_2$, namely,
$F_3=\SS(-n)$, where $\SS$ is the spinor bundle on $X_2$ with $\det\SS=\OO_X(1)$,
and $G_2=\II_{\P^1,S}(m)$, where $S\in|\OO_X(1)|$, $\P^1$ is a line on the
surface $S$. Let $M_X(v)$ be the Gieseker-Maruyama moduli scheme of
semistable sheaves on $X$ with Chern character $v=\mathrm{ch}(E)$,
determined from the triple \eqref{eqn 1}, and let
\begin{equation}\label{eqn 2}
	M:=\{[E]\in M_X(v)\ |\ E \ \textit{is a Gieseker-stable
		extension}\ \eqref{eqn 1}\}.
\end{equation}
Then the following statements are true.\\
1) For $X_1,\ X_2,\ X_4,\ X_5$ in the case\ $i=j=1$, $k\ge1$, $n=\lceil\frac k2\rceil$,
$m<-n$,\\
2) for $X_1,\ X_4,\ X_5$ in the case of $i=2$, $j=1$, $k\ge1$, $n=\lfloor\frac k2\rfloor$,
$m<-n$,\\
3) for $X_2$ in each of the cases \\
3.1)\ \ $i=1$, $j=2$, $n=1$, $m\le-1$,\\
3.2)\ \ $i=3$, $j=1$, $k\ge1$, $n=\lfloor\frac k2\rfloor+1$, $m\le-n$,\\
3.3)\ \ $i=3$, $j=2$, $1\le k\le2$, $n=1$, $m\le-1$,\\
the set $M$ is a smooth dense open subset of an irreducible component
$\overline{M}$ of the moduli scheme $M_X(v)$. Moreover, $M$ is a fine moduli space, and reflexive sheaves form a dense open set in $M$. Moreover,
all components $\overline{M}$ from the infinite series 1), 2) and 3.1)-3.3) are rational varieties for each of the varieties $X_l,\ l=1,2,4,5$, except for the series 2) for $X=X_4$, in which each component is irrational. Moreover, in all cases
the dimensions of the components $\overline{M}$ are found as polynomials from $\mathbb{Q}[k,m,n]$
or $\mathbb{Q}[m]$, respectively.}

\vspace{2mm}
A significant part of this paper is devoted to the study of
semistable sheaves of rank 2 with maximal class $c_3$ on the quadric
$X_2$. We show that for $c_1\in\{-1,0\}$ and all values of the class
$c_2$, except for a few small values, every such sheaf is given by an
extension of the form \eqref{eqn 1}, that is, in the notation \eqref{eqn 2} we have
equality $M=\overline{M}$. In this case the construction from the
proof of Theorem 3 allows for a significant refinement, giving complete
description of all moduli spaces of semistable sheaves with maximal
class $c_3$ on $X_2$. In the regeneraling cases of small values of $c_2$ and
maximal $c_3\ge0$ it is also possible to obtain an explicit description of
moduli spaces. (The only case of negative maximal
$c_3$ --- the case $(c_1,c_2,c_{3\mathrm{max}})=(0,1,-1)$ -- is specified in
Remark \ref{Remark 3.2} of the paper.) These results, proven in Theorems
\ref{moduli with c3 max} -- \ref{moduli with c3 max 4} are collected in
the following two theorems.

\vspace{2mm}
\textbf{Theorem 4.}
\textit{Let $X=X_2$ be a quadric, and $M_X(v)$ be the Gieseker-Maruyama moduli scheme of semistable sheaves $E$ of rank 2 on $X$ with Chern classes
	$(c_1,c_2,c_3)$, where $c_1\in\{-1,0\}$, $c_2\ge0$, $c_3=c_{3\mathrm{max}}\ge0$ is maximal for each $c_2$, and
	\begin{equation*}
		v=\ch(E)=(2,c_1H,\frac12(c_1^2-c_2)H^2,\frac12(c_{3\mathrm{max}}+
		\frac23c_1^3-c_1c_2)[\mathrm{pt}]),
	\end{equation*}
	where $H=c_1(\OO_X(1))$. Then the following statements hold.\\
	(1.i) For $c_1=-1$, even $c_2=2p$, $p\ge2$, and $c_{3\mathrm{max}}=
	\frac12c_2^2$ the variety $M_X(v)$ is a grassmannization of 2-dimensional
	quotient spaces of the vector bundle of rank $\frac14(c_2+2)^2$ on the
	space $\P^4$ defined by the first formula \eqref{sheaf A} for
	$n=1$ and $m=-p$. In this case $\dim M_X(v)=\frac12(c_2+2)^2$.\\
	(1.ii) For $c_1=-1$, odd $c_2=2p+1$, $p\ge1$, and $c_{3\mathrm{max}}=
	\frac12(c_2^2-1)$ the variety $M_X(v)$ is the grassmannization of 2-dimensional
	quotient spaces of the vector bundle of rank $\frac14(c_2+1)(c_2+3)$ on the
	grassmannian $\mathbb{G}=\mathrm{Gr}(2,4)$ defined by the second formula
	\eqref{sheaf A} for $m=-p$. In this case $\dim M_X(v)=\frac12(c_2+1)(c_2+3)$.\\
	(1.iii) For $c_1=0$, odd $c_2=2p+1$, $p\ge1$, and $c_{3\mathrm{max}}=
	\frac12(c_2^2+1)$ the variety $M_X(v)$ is a projectivization of the
	vector bundle of rank $\frac12(c_2+1)(c_2+3)$ on the space
	$\P^4$, defined by the formula \eqref{sheaf AS} with $n=1$ and $m=-p$. In this case $\dim M_X(v)=\frac12c_2^2+2c_2+\frac92$.\\
	(1.iv) For $c_1=0$, even $c_2=2p$, $p\ge3$, and $c_{3\mathrm{max}}=
	\frac12c_2^2$ the variety $M_X(v)$ is a projectivization of the vector
	bundles of rank $\frac12c_2^2+2c_2+1$ on the grassmannian $\mathbb{G}$,
	defined by the formula \eqref{sheaf AS I} for $m=1-p$. In this case $\dim
	M_X(v)=\frac12c_2^2+2c_2+4$.\\
	(2) In all the above cases, the scheme $M_X(v)$ is irreducible and is a
	smooth rational projective variety, all sheaves from $M_X(v)$
	are stable, the general sheaf in $M_X(v)$ is reflexive, and $M_X(v)$ is a fine
	moduli space.}

\vspace{2mm}
\textbf{Theorem 5.}
\textit{Under the conditions and notation of Theorem 4, the following statements are true:\\
	(1) For $c_1=-1$, $c_2=1$ and $c_{3\mathrm{max}}=0$, the variety $M_X(v)$ is a point $[\SS(-1)]
	$.\\
	(2) For $c_1=c_2=c_{3\mathrm{max}}=0$, the variety $M_X(v)$ is a point $[\OO_X^{\oplus2}]$.\\
	(3) For $c_1=-1$, $c_2=2$ and $c_{3\mathrm{max}}=2$ we have $M_X(v)\simeq
	G(2,5)$.\\
	(4) For $c_1=0$, $c_2=2$ and $c_{3\mathrm{max}}=2$ the scheme $M_X(v)$ is
	irreducible, has dimension 9 and is not smooth. \\
	(5) For $c_1=0$, $c_2=4$ and $c_{3\mathrm{max}}=8$ the scheme $M_X(v)=
	M_X(2;0,4,8)$ is irreducible and equal to a union of two irreducible subsets $M_1$ and
	$M_2$. These subsets are described as follows.\\
	(5.i) $M_1$ is a smooth rational variety of dimension 20,
	which is the projectivization of a locally free sheaf of rank 17 on
	grassmannian $\mathbb{G}$. $M_1$ is a fine moduli space and all sheaves in $M_X(v)_1$ are stable. Moreover, the scheme $M_X(v)$ is nonsingular along
	$M_1$.\\
	(5.ii) the scheme $M_2$ is irreducible, has dimension 21, and polystable
	sheaves in $M_2$ form a closed subset of dimension 12 in $M_2$, in
	which the scheme $M_X(v)$ is not smooth.}  

\vspace{2mm} 

Let us proceed on to a brief summary of the contents of the paper. 
In section 2 we recall necessary technical means to work in derived 
categories of coherent sheaves on Fano varieties, which are used in
further. In section 3 the key result of the paper is proved --- the 
theorem \ref{Theorem 3.1} on the boundedness of the third Chern 
class $c_3$ of semistable objects, including sheaves, in $\mathrm{D} 
^b(X_2)$ and their complete classification is given for maximal 
values of $c_3\ge0$. In section 4 we construct new infinite series 
of components of moduli spaces of rank two semistable sheaves on the 
varieties $X_1,\ X_2,\ X_4$ and $X_5$. As a special case of these 
series, in section 5 we describe moduli spaces of rank two 
semistable sheaves with the maximal third Chern class on the quadric 
$X_2$. At the end of the same section we prove the disconnectedness 
of moduli space $M_{X_2}(2,0,4,8)$. Section 6 is dedicated to a 
proof of boundedness of the class $c_3$ of semistable rank two 
reflexive sheaves of general type with $c_1=0$ on the varieties 
$X_4$ and $X_5$.

\subsection*{Notation.}
\begin{center}
\begin{tabular}{ r l }
$\bk=\C$ & base field\ \\
$X_1$ & three-dimensional projective space $\P^3$\\
$X_2$ & smooth three-dimensional quadric in $\p4$\\
$X_4$ & smooth intersection of two four-dimensional quadrics in $\P^5$\\
$X_5$ & smooth section of the grassmannian $G(2,5)$ embedded by\\
& Pl\"ucker into the space $\p9$, by a linear subspace $\P^6$\\
$\SS$ & spinor bundle on the quadric $X_2$ with $\det\SS\cong\OO_Q(1)$\\
\end{tabular}
\end{center}		

\begin{center}
\begin{tabular}{rl}
$X$ & one of the varieties $X_1$, $X_2$, $X_4$, $X_5$\\
$H$ & positive generator of the Picard group $\mathrm{Pic}X\simeq\Z$ 
-- \\ 
& class of hyperplane section with embedding $X=X_i\hookrightarrow
\P^{2+i}$,\\
& $i=1,2$, $X_i\hookrightarrow\mathbb{P}^{1+i}$, respectively, 
$i=4,5$ \\
$\mathrm{Coh}(X)$ & category of coherent sheaves on $X$ \\
$[E]$ & $S$-equivalence class of an $H$-Gieseker-semistable sheaf \\
& $E\in\mathrm{Coh}(X)$ \\
$\mathrm D^b(X)$ & bounded derived category of coherent sheaves on
$X$ \\
$\mathcal H^{i}(E)$ & $i$th cohomology group of the complex 
$E\in\mathrm D^b(X)$ \\
$\ch(E)$ & Chern character of the object $E \in\mathrm D^b(X)$ \\
$\ch_{\le m}(E)$ & $(\ch_0(E), \ldots, \ch_m(E))$ \\
$M_X(v)$ & space (scheme) of $H$-Gieseker-semistable \\
& sheaves $E$ on $X$ with $\ch(E)=v$\\
$M_X(r;c_1,c_2,c_3)$ & alternative, more traditional, notation
of the moduli scheme \\
& $M_X(v)$, where $r,c_1,c_2,c_3$ are the rank and Chern classes of 
\\ 
& a sheaf $[E]\in M_X(v)$ \\
$\rmext^i(\FF,\GG)$ & $:=\dim\Ext^i(\FF,\GG)$ for coherent sheaves 
$\FF$ and $\GG$ on $X$\\
\end{tabular}
\end{center}

\vspace{5mm}

\section{Preliminaries}

\vspace{5mm}

This section contains the necessary definitions and results that we 
will use below.

Let $X$ be one of the varieties $X_i,\ i=1,\ 2,\ 4,\ 5$. Cohomology 
ring $H^*(X,\Z)$ is generated by the classes of a hyperplane section 
$H\in H^2(X,\Z)$, a line $L\in H^4(X,\Z)$ (understood as a projective line in the space $\p{2+i}\supset X_i=X$ for $i=1,2$, respectively, $X_i\hookrightarrow\mathbb{P}^{1+i}$ for $i=4,5$) and a point $\{\mathrm{pt}\}\in H^6(X,\Z)$ (for simplicity we will also
denote the class of a point by 1). We have
\begin{equation}
	H^2=iL,\ \ H\cdot L=1,\ \ H^3=i,\ \ \ \textrm{}\ \ \ X=X_i,\ \ \ i=1,2,4,5.
\end{equation}

The \textit{slope} $\mu(E)$ of a coherent sheaf $E\in\Coh(X)$ is defined as $$\mu(E)=\frac{H^2\cdot\ch_1(E)}{H^3\cdot\ch_0(E)},$$ in the case of division by 0 we set $\mu(E)=
+\infty$. A coherent sheaf $E$ is called
$\mu$-\textit{(semi)stable} if for any proper
subsheaf $0\ne F\hookrightarrow E$ the inequality $\mu(F)<\nolinebreak{(\le)}\,\mu(E/F)$ holds.

Let $f, g \in \R[m]$ be polynomials. If $\deg(f) < \deg(g)$, then we set
$f > g$. If $d = \deg(f) = \deg(g)$ and $a$, $b$ are the leading coefficients in $f$,
$g$ respectively, then we put $f < (\leq) g$ if $\frac{f(m)}{a} < (\leq)
\frac{g(m)}{b}$ for all $m \gg 0$.
For an arbitrary sheaf $E\in\Coh(\P^3)$ we define the numbers $a_i(E)$ for $i \in \{0,
1, 2, 3\}$ through the Hilbert polynomial $P(E, m):=\chi(E(m))=a_3(E)m^3+a_2(E) m^2+
a_1(E) m + a_0(E)$. In addition, we set $P_2(E,m):=a_3(E)m^2+a_2(E)m+a_1(E)$.
The sheaf $E \in \Coh(\P^3)$ is called \emph{(Gieseker-)(semi)stable},
respectively, \emph{(Gieseker-)2-(semi)stable}, if for any
its proper subsheaf $0\ne F \into E$ the inequality $P(F, m) < (\leq)P(E/F, m)$ holds
(respectively, the inequality $P_2(F, m) < (\leq) P_2(E/F, m)$ holds).
Stability, 2-stability and $\mu$-stability of a sheaf satisfy the relations:
\begin{equation*}
	\xymatrix{
		\text{$\mu$-stability} \ar@{=>}[r] & \text{$2$-stability} \ar@{=>}[r] & \text{stability} \ar@{=>}[d] \\
		\text{$\mu$-semistability} & \text{$2$-semistability} \ar@{=>}[l] &\text{semistability} \ar@{=>}[l]
	}
\end{equation*}
Let us recall the concept of tilt stability. Let $\beta\in\R$. Let us define \emph{twisted
	Chern character} as $\ch^\beta=e^{-\beta H}\cdot\ch$. Note that for $\beta\in\Z$ and for any $E\in\mathrm D^b(X)$ the equality $\ch^\beta(E)=\ch(E(-\beta))$ holds.
Let us present explicit formulas for the components $\ch_i^\beta=\ch_i^\beta(E)$:
\begin{equation}
	\begin{split}
		& \ch_0^\beta=r,\ \ch_1^\beta=\ch_1-\beta H\ch_0,\ \ch_2^\beta=\ch_2-\beta H\ch_1+
		\frac{\beta^2}{2} H^2\ch_0,\\
		& \ch_3^\beta=\ch_3-\beta H\ch_2+\frac{\beta^2}{2} H^2\ch_1-\frac{\beta^3}{6}
		H^3\ch_0.
	\end{split}
\end{equation}
In particular, for 
\begin{equation}\label{c i}
\begin{split}
& X=X_2\ \ \ \ \textit{and}\ \ \ \ E\in\mathrm D^b(X),\\
& v=\ch(E)=(r,cH,dH^2,e[\mathrm{pt}]),\ \ \ \
r,c\in\Z,\ \ d\in\frac{1}{2}\Z,\ \ e\in\frac{1}{6}\Z,\\
& c(E)=(1,c_1(E),c_2(E),c_3(E))=(1,c_1H,c_2[l],c_3[\mathrm{pt}]),\ \ \ c_1,c_2,c_3\in\Z,
\end{split}
\end{equation} 
we have:
\begin{equation}\label{c vers ch}
\begin{split}
&\ch_0^\beta(E)=r,\ \ \ch_1^\beta(E)=(c-\beta r)H,\ \ \ch_2^\beta(E)=(d-\beta 
c+\frac{\beta^2}{2} r)H^2,\\
& \ch_3^\beta(E)=e-2\beta d+\beta^2c-\frac{\beta^3}{3}r,
\end{split}
\end{equation}
\begin{equation}\label{c vers ch 2}
c_1=c,\ \ \ c_2=c^2-2d,\ \ \ c_3=2e+\frac13c^3-2cd,\\
\end{equation}
\begin{equation}\label{c vers ch 3}
\ch(E)=(r,c_1H,\frac12(c_1^2-c_2)H^2,\frac12(c_3+\frac23c_1^3-c_1c_2)[\mathrm{pt}]).
\end{equation}
Define a \emph{torsion pair}
$$
\mathcal T_\beta=\{E\in\Coh(X)\colon\text{any}\ \text{quotient}\ E\to 
G\ \text{satisfies}\ \mu(G)>\beta\},
$$
$$
\FF_\beta=\{E\in\Coh(X)\colon\text{any}\ \text{subsheaf}\ 0\ne F\to 
E\ \text{satisfies}\ \mu(F)\le\beta\}
$$
and a category $\Coh^\beta(X)$ as the subcategory $\langle\FF_\beta[1],\TT_\beta\rangle$ in $D^b(X)$. For $\alpha\in\R_+$ the \emph{tilt-slope} of an object
$E\in\Coh^\beta(X)$ is defined as
$$
\nu_{\alpha,\beta}(E)=\nu_{\alpha,\beta}(\ch_0(E),\ch_1(E),\ch_2(E))=\frac{H\cdot
\ch_2^\beta(E)-\frac{\alpha^2}{2}H^3\cdot\ch_0^\beta(E)}{H^2\cdot\ch_1^\beta(E)}.
$$
An object $E\in\mathrm \Coh^\beta(X)$ is called to be \textit{tilt-(semi)stable} (or \\
$\nu_{\alpha,\beta}$-\textit{(semi)stable}), if for any subobject
$0\ne F\hookrightarrow E$ we have $\nu_{\alpha,\beta}(F)<(\le)\, 
\nu_{\alpha,\beta}(E/F)$. In case where the last inequality becomes an equality,
we will also say that $E$ \textit{is destabilized by the exact triple} $0\to F\to E\to E/F\to0$.

\begin{proposition}\label{Prop 2.1}
(i) An object $E\in\Coh^{\beta}(X)$ is $\nu_{\alpha,\beta}$-(semi)stable for $\beta<\mu(E)$ 
and $\alpha\gg0$ iff $E$ is a 2-(semi)stable sheaf.\\
(ii) Let $X=X_2$. For any $\nu_{\alpha,\beta}$-semistable object 
$E\in\Coh^\beta(X_2)$ with $\ch(E)=(r,cH,dH^2,e)$ we have the following inequalities:
$$\Delta(E)=\Delta(\ch E):=\frac{(H^2\cdot\ch_1^\beta(E))^2-2(H^3 \cdot
\ch_0^\beta(E)) (H\cdot\ch_2^\beta(E))}{(H^3)^2}=c^2-2rd\ge 0,$$
$$W_{\alpha,\beta}(E):=\alpha^2\Delta(E)+\frac{4(H\cdot\ch_2^\beta(E))^2}{(H^3)^2}
-\frac{6(H^2\cdot\ch_1^\beta(E))\ch_3^\beta(E)}{(H^3)^2}=$$
$$=(\alpha^2+\beta^2)(c^2-2rd)+(3re-2cd)\beta+4d^2-3ce\ge 0.$$
\end{proposition}
\begin{proof}
The statement (i) is proven in \cite[Prop. 14.2]{Bri} for K3 surfaces, but an analogous proof holds in our case. For a proof of the first inequality in (ii) see \cite[Thm. 2.1]{MS18}, and for the second one --- in \cite[Assumption C]{MS18}.
\end{proof}

The set of semistable objects changes with variation of $(\alpha,\beta)$ pairs.
Consider $H\cdot\ch_{\le 2}$ as a mapping to $\Lambda=\mathbb Z^2\oplus\frac{1}{2}\Z$. 
\textit{A numerical wall} with respect to a vector $v\in\Lambda$ is a
non-empty proper subset $W$ of the upper half-plane $\{\alpha,\beta)\ |\
\alpha>0\}$, given by an equation of the form $\nu_{\alpha,\beta}(v)=\nu_{\alpha,\beta}(w)$
for some other vector $w\in\Lambda$. A subset $\UU$ of the numerical wall $W$ is called an \textit{actual wall} (or simply \textit{wall}) if the set of
semistable objects $E\in \Coh^\beta(X)$ with $\ch(E)=v$ changes when passing through $\UU$.

\begin{proposition}[{\cite[Theorem 2.9]{Sch18}}] Let
	$v\in\Lambda$ be a fixed class. Everywhere below, numerical walls are considered
	with respect to class $v$.\\
	(i) Numerical walls are either semicircles with centers on the $\beta$-axis,
	or rays parallel to the $\alpha$-axis. If $v_0\neq 0$, then there is exactly one
	vertical numerical wall with the equation $\beta = v_1/v_0$. If $v_0 = 0$, then
	there are no actual vertical walls. \\
	(ii) The curve $\nu_{\alpha, \beta}(v) = 0$ is a hyperbola that can
	degenerate if $v_0 = 0$ or $\Delta(v) = 0$. Moreover, this hyperbola crosses all
	semicircular walls at their highest points.\\
	(iii) If two numerical walls defined by the classes $w,u \in \Lambda$ intersect,
	then $v$, $w$ and $u$ are linearly dependent. In particular, these two walls coincide.\\
	(iv) If a numerical wall has at least one point at which it is an
	actual wall, then it is an actual wall at every point.\\
	(v) If $v_0\neq0$, then there is a largest semicircular wall on both sides
	from a single vertical wall.\\
	(vi) If there is an actual wall, numerically defined by an exact triple of
	tilt-semistable objects $0\to F\to E\to G\to0$ such that $\ch_{\le 2}(E)= v$,
	then $\Delta(F) + \Delta(G) \le \Delta(E).$
	Moreover, the equality is satisfied if and only if $\ch_{\le 2}(G)=0$.\\
	(vii) If $\Delta(E) = 0$, then $E$ can destabilize only in a single numerical
	vertical wall. In particular, line bundles and their shifts by one are
    tilt-semistable everywhere.
\end{proposition}

Throughout this section we will consider the case when
$$
X=X_2
$$ is a three-dimensional quadric. The following lemma can be verified by direct calculation.
\begin{lemma} Let $E\in\mathrm D^b(X)$ and $\ch(E)=(r,cH,dH^2,e)$.
	The equation $W_{\alpha,\beta}(E)=0$ is equivalent to the relation $\nu_{\alpha,\beta}(r,
	cH,dH^2)=\nu_{\alpha,\beta}(2c,4dH,3eH^2).$
	In particular, the equation $W_{\alpha,\beta}(E)=0$ describes a numerical wall in
	tilt stability.
\end{lemma}

Let us denote the radius of the semicircular wall $W_{\alpha,\beta}(E)=0$ by $\rho_W(E)$, and
its center by $s_W(E)$. Center of the semicircular wall $\nu_{\alpha,\beta}(E)=
\nu_{\alpha, \beta}(F)$ will also be denoted by $s(E,F)$.

Below we will use the following results about tilt-semistable objects.
\begin{lemma}[{\cite[Lemma 7.2.1]{BMT}}]\label{big alpha} If the object $E\in
	\Coh^\beta(X)$ is $\nu_{\alpha,\beta}$-semistable for $\alpha\gg0$, then
	one of the following statements is true:\\
	(1) $\HH^{-1}(E)=0$ and $\HH^0(E)$ is a pure $\mu$-semistable sheaf of dimension
	$\ge0$,\\
	(2) $\HH^{-1}(E)=0$ and $\HH^0(E)$ is a sheaf of dimension $\le 1$,\\
	(3) $\HH^{-1}(E)$ is a $\mu$-semistable torsion-free sheaf and
	$\HH^0(E)$ is a sheaf of dimension $\le 1$.
\end{lemma}
\begin{proposition}[{\cite[Lem. 2.4]{MS18}}]
	\label{radius}
	Suppose that an object $E$ is tilt-semistable and is destabilized by either a subobject
	$F\hookrightarrow E$, or a quotient $E\onto F$ in $\Coh^\beta(X)$, inducing a
	non-empty semicircular wall $W$. Let us further assume that $\ch_0(F)>\ch_0(E)\ge0$.
	Then the radius $\rho_W$ of the wall $W$ satisfies the inequality
	$$
	\rho_W\le\frac{\Delta(E)}{4\ch_0(F)(\ch_0(F)-\ch_0(E))}.
	$$
\end{proposition}
Let us recall the construction of Bridgeland stability conditions on $X$. Let
\begin{equation*}
	\begin{split}
		& \TT'_{\alpha,\beta}=\{E\in\Coh^\beta(X)\ |\ \text{any quotient}\,E\onto
		G\;\text{satisfies}\;\nu_{\alpha,\beta}(G)>0\},\\
		& \mathcal F'_{\alpha,\beta}=\{E\in\Coh^\beta(X)\ |\ \text{any
			subobject}\ 0\ne F\into E\ \text{satisfies}\\
		&\nu_{\alpha,\beta}(F)\le 0\},
	\end{split}
\end{equation*}
and set $\AA^{\alpha,\beta}(X)=\langle \mathcal F'_{\alpha,\beta}[1],\mathcal
T'_{\alpha,\beta}\rangle$. For any $s>0$ we define
$$\lambda_{\alpha,\beta,s}=\frac{\ch^{\beta}_3 - s\alpha^2
	H^2\cdot\ch^{\beta}_1}{H\cdot\ch^{\beta}_2 - \frac{\alpha^2}{2}H^3\cdot
	\ch^{\beta}_0}.$$
An object $E\in\AA^{\alpha, \beta}(X)$ is called $\lambda_{\alpha, \beta,
s}$-(semi)stable if for any nontrivial subobject $F\into E$ we have
$\lambda_{\alpha,\beta,s}(F)<(\le)\lambda_{\alpha,\beta,s}(E)$.

$D^b(X)$ has a full strong exceptional collection $(\OO_X(-1),\SS(-1),$\\
$\OO_X,\OO_X(1))$, where $\SS$ is a spinor bundle on $X$.
\begin{proposition}[{\cite{Sch14}},{\cite[Thm. 6.1(2)]{Sch19}}] \label{heart}
	(i) Let $\alpha<\frac 13,\beta\in [-\frac 12.0],s=\frac 16.$ For any
	$\gamma\in\R$ we define a torsion pair
	\begin{equation*}
		\begin{split}
			& \TT''_{\gamma}=\{E\in\AA^{\alpha,\beta}(X)\ |\ \text{any quotient}\,E\onto G
			\,\text{satisfies}\,\lambda_{\alpha,\beta, s}(G)>\gamma \}, \\
			& \FF''_{\gamma} = \{E \in \mathcal A^{\alpha, \beta}(X)\ |\ \text{any
				subobject}\, 0\ne F\into E\,\text{satisfies}\, \\
			& \lambda_{\alpha,\beta, s}(F)\leq\gamma\}.
		\end{split}
	\end{equation*}
	There is a $\gamma\in\R$ such that
	$$\langle \mathcal T''_{\gamma},\mathcal F''_{\gamma}[1]\rangle =\mathfrak
	C:=\langle\OO_X(-1)[3],\SS(-1)[2],\OO_X[1],\OO_X(1)\rangle.$$
	(ii) Let $v$ be the Chern character of an object from $\mathrm D^b(X),$ and
	$\alpha_0>0, \beta_0\in\mathbb{R},$ and $s>0$
	such that $\nu_{\alpha_0,\beta_0}(v)=0,\ H^2\cdot
	v_1^{\beta_0}>0,$ and $\Delta(v)\ge 0$. Let us assume that all $\nu_{\alpha_0,
		\beta_0}$-semistable objects of class $v$ are $\nu_{\alpha_0,\beta_0
	}$-stable. Then there is a neighborhood $U$ of the point $(\alpha_0,\beta_0)$
	such that for all $(\alpha,\beta)\in U$ with $\nu_{\alpha,\beta}(v)
	>0$, an object $E\in\Coh^\beta(X)$ with $\ch(E)=v$ is $\nu_{\alpha,\beta}
	$-semistable if and only if it is $\lambda_{\alpha,\beta,s}
	$-semistable.
\end{proposition}

\begin{remark}\label{e le d(d+1)}
From the proof of \cite[Prop. 3.2]{MS18} it follows that for a 
tilt-semistable object $E\in\Coh^\beta(X)$ with $\ch(E)=(1,0,-dH^2,e 
)$ the inequality $e\le d(d+1)$ holds.
\end{remark}

In what follows we will use the equalities
\begin{equation}\label{v(S)}
	\ch(\OO_X(n))=(1,nH,\frac {n^2}2H^2,\frac{n^3}3[\mathrm{pt}]),\ \ \ \
	\ch(\SS(-1))=(2,-H,0,\frac16[\mathrm{pt}]),
\end{equation}
the exact triple of sheaves on $X$ (\cite{Sch14})
\begin{equation}\label{spin tr}
	0\to\SS(-1)\to V\otimes\OO_X\xrightarrow{\varepsilon}\SS\to 0,\ \ \ \ \
	\ \ \ \ \ V\cong\bk^4,
\end{equation}
and the following properties of the spinor bundle $\SS$ on $X$. In the triple \eqref{spin tr} the epimorphism $\varepsilon$ gives an isomorphism in sections
$h^0(\varepsilon):\ V\xrightarrow{\cong}H^0(\SS)$. Let $B$ be the base of the
family of lines on $X$. Everywhere below we will
identify $B$ with $\P^3$ through the isomorphism $f:\ \P^3=\P(V)
\xrightarrow{\cong}B,\ \bk v\mapsto l$, where line $l$ is the scheme of
zeros of the section $h^0(\varepsilon)(v)$. Let $\P^4=\langle X\rangle$ be
the projective hull of the quadric $X$, $\check{\P}^4=|\OO_X(1)|$ and
let $\mathbb{S}\subset\check{\P}^4\times X$ be the universal family of
two-dimensional quadrics (hyperplane sections) on $X$. Let $\mathbb{G}:=Gr(
2,V)$ be the grassmannian of lines in $\P^3$, and let $l_x$ be the line in $\P^3$
corresponding to a point $x\in\mathbb{G}$. Then $f(l_x)$ is a series
of lines on the two-dimensional quadric $S_x\in\check{\P}^4$. It is not difficult to
see that $\mu:\ \mathbb{G}\to\check{\P}^4,\ x\mapsto S_x$ is a double
covering ramified in the dual quadric $\check{X}\subset
\check{\P}^4$. The covering $\mu:\ \mathbb{G}\to\check{\P}^4$ defines
on $X$ a family of two-dimensional quadrics $\widetilde{\mathbb{S}}:=\mathbb{G} \times_{\check{\P}^4}\mathbb{S}\subset\mathbb{G}\times X$ with the base $\mathbb{G}$.
The tautological subbundle $g:\KK\to V\otimes\OO_{\mathbb{G}}$ of
rank two on $\mathbb{G}$ and the epimorphism $\varepsilon$ in \eqref{spin tr}
determine the composition
$e:\ \KK\boxtimes\OO_X\xrightarrow{g\boxtimes\mathrm{id}_X}V\otimes\OO
_{\mathbb{G}}\boxtimes\OO_X\xrightarrow{\mathrm{id}_{\mathbb{G}}
	\boxtimes\varepsilon}\OO_{\mathbb{G}}\boxtimes\SS$, and let $\mathbb{I}
:=\im(e)\otimes\OO_{\mathbb{G}}\boxtimes\OO_X(-1)$. It is not hard to see that
$\ker(e)\cong\OO_{\mathbb{G}}(-1)\boxtimes\SS(-1)$. Thus, we have the following
exact triple on $X\times\mathbb{G}$:
\begin{equation}\label{univ on XxG}
	0\to\OO_{\mathbb{G}}(-1)\boxtimes\SS(-1)\to\KK\to\mathbb{I}\otimes
	\OO_{\mathbb{G}}\boxtimes\OO_X(1)\to0.
\end{equation}
For an arbitrary point $x\in\mathbb{G}$ the restriction of this triple to
$\{x\}\times X$ gives an exact triple:
\begin{equation}\label{spin tr on X}
	0\to\SS(-1)\to\OO^{\oplus2}_X\to\II(1)\to0,\ \ \ \ \ \ \ \ \ \
	\II:=\II_{\P^1,S}.
\end{equation}
Here $\II_{\P^1,S}$ is the ideal sheaf of an arbitrary line $\P^1\subset S$
from the series of lines defined by the point $x$ on a two-dimensional quadric $S=S_x$.

\vspace{5mm}


\section{Semistable objects of rank two in $D^b(X_2)$ with maximal third Chern class $c_3$}

\vspace{5mm}

In this section we prove one of the main results of the paper --- the Theorem
\ref{Theorem 3.1} on the description of rank 2 semistable objects from $D^b(X_2)$, including
sheaves, on the three-dimensional quadric $$X=X_2,$$ having the maximal third Chern class
$c_3$. Everywhere below, $Q_2$ denotes a two-dimensional quadric — a hyperplane section
of $X$, and $\P^1$ denotes a projective line on the surface $Q_2$.

\begin{theorem}\label{Theorem 3.1} Let $E\in\Coh^\beta(X)$ be a
	tilt-semistable object with $\ch(E)=(2,cH,dH^2,e)$, $d\in\frac{1}{2}\Z$, $e\in\frac{1} {6}\Z$.\\
	(1) If $c=-1$, then $d\le 0$. \\
	(1.1) If $d=0$, then $e\le \frac 16$. In the case $e=\frac16$ we have $E\cong\SS(-1)$.\\
	(1.2) If $d=-\frac 12$, then $e\le\frac 53$. In case of equality $E$ is
	destabilized by an exact triple $0\to\OO_X(-1)^{\oplus
		3}\to E\to\OO_X(-2)[1]\to 0.$\\
	(1.3) If $d\le -1$ is an integer, then $e\le d^2-2d+\frac 16$. In case of equality
	$E$ is destabilized by an exact triple $0\to\OO_X(-1)^{\oplus 2}\to E\to
	\II_{\P^1,Q_2}(d)\to 0.$\\
	(1.4) If $d\le-\frac 32$ is non-integer, then $e\le d^2-2d+\frac {5}{12}$. In
	in case of equality $E$ is destabilized by an exact triple
	$0\to\OO_X(-1)^{\oplus 2}\to E\to\OO_{Q_2}(d-\frac 12)\to 0.$\\
(2) If $c=0$, then $d\le 0$.\\
(2.1) If $d=0$, then $e\le 0$. In case of equality $E\cong\OO_X^{\oplus 2}$.\\
(2.2) If $d=-\frac 12$, then $e\le -\frac 12$. \\
(2.3) If $d=-1$, then $e\le 1$. In the case $e=1$, the sheaf $E$ is destabilized by an exact triple $0\to
\OO_X(-1)^{\oplus 6}\to E\to \SS(-2)^{\oplus 2}[1]\to0.$\\
(2.4) If $d\le-\frac 32$ is non-integer, then $e\le d^2+\frac 14$.
In case of equality, $E$ is destabilized by an exact triple
$0\to \SS(-1)\to E\to\OO_{Q_2}(d+\frac 12)\to 0.$\\
(2.5) If $d=-2$, then $e\le 4$. In the case $e=4$, the sheaf $E$ is destabilized by one of
exact triples $0\to\SS(-1)\to E\to \II_{\P^1,Q_2}(-1)\to 0,\ \ \ 0\to\II_{\P^1,Q_2}(-1)\to E\to\SS(-1)\to 0,\ \ \
0\to\OO_X(-1)^{\oplus 4}\to E\to\OO(-2)^{\oplus 2}[1]\to0$.\\
(2.6) If $d\le -3$ is an integer, then $e\le d^2$. In case of equality $E$
destabilized by an exact triple $0\to \SS(-1)\to E\to\II_{\P^1,Q_2}(d+1)\to 0.$\\
(3) Let $E\in\Coh(X)$ be a semistable sheaf of rank 2 with $c_1(E)=-H$,
$c_2(E)=c_2[l]$, $c_3(E)=c_3[\mathrm{pt}]$. Then $c_2\ge0$ and \\
(3.1) if $c_2$ is even, then
\begin{equation}\label{c1=-1,c2 even}
	c_3\le\frac12c_2^2;
\end{equation}
(3.2) if $c_2$ is odd, then
\begin{equation}\label{c1=-1,c2 odd}
	c_3\le\frac12(c_2^2-1).
\end{equation}
(4) Let $E\in\Coh(X)$ be a semistable sheaf of rank 2 with $c_1(E)=0$, $c_2(E)=
c_2[l]$, $c_3(E)=c_3[\mathrm{pt}]$. Then $c_2\ge0$ and\\
(4.1) if $c_2$ is even, then
\begin{equation}\label{c1=-1,c2 even}
	c_3\le\frac 12c_2^2;
\end{equation}
(4.2) if $c_2$ is odd, then
\begin{equation}\label{c1=-1,c2 odd}
	c_3\le\frac 12(c_2^2+1).
\end{equation}
\end{theorem}   

\begin{remark}\label{Remark 3.2}
	In statements (1.1)-(1.4), (2.1), (2.3)-(2.6) of the Theorem, the maximal 
	value $c_{3\mathrm{max}}(E)$ is non-negative. Statement (2.2) gives
	the only case when $c_{3\mathrm{max}}(E)=-1$ is negative.
	In this case, it is not known whether an object $E$ with Chern classes
	$c_1(E)=0,\ c_2(E)=1,\ c_3(E)=c_{3\mathrm{max}}(E)=-1$ is a Gieseker-(semi)stable
	sheaf.
\end{remark}

We will preface the proof of the Theorem with a series of auxiliary statements
--- Lemmas 3.1-3.12.

\begin{lemma}\label{trivial} Let $E\in\Coh^\beta(X)$ be a
	tilt-semistable object with $\ch(E)=(2,0,0,e)$. Then $e\le 0$, and if
	$e=0$, then $E\cong \OO_X^{\oplus 2}$.
\end{lemma}

\begin{proof} Since $E\in\Coh^\beta(X)$, we get $\beta< 0$. We have \mbox{$
		W_{\alpha,\beta}(E)=6e\beta\ge 0$}, so $e\le 0$. If $e=0$, then,
	following the arguments from the proof of the statement \cite[Prop. 4.1]{Sch19} and
	using Proposition \ref{heart}, we find $E\cong\mathcal O_X^{\oplus
		2}$.\end{proof}

\begin{lemma}\label{spinor}Let $E\in\Coh^\beta(X)$ be a tilt-semistable object with $\ch(E)=(2,-H,0,e)$. Then $e\le\frac 16$, and if $e=\frac 16$, then $E\cong \SS(-1)$.
\end{lemma}

\begin{proof} Since $E\in\Coh^\beta(X)$, then $\beta<-\frac 12$. Hyperbola
	$\nu_{\alpha,\beta}(E)=0$ intersects the $\beta$-axis at the point $\beta=-1$, so
	if there is an actual semicircular wall, it must intersect the line
	$\beta=-1$. In our case, $H^2\cdot\ch_1^{-1}(E)$ has the smallest
	positive value, so if $E$ is $\nu_{\alpha,-1}$-semistable, it
	must be $\nu_{\alpha,-1}$-semistable for any $\alpha$ \cite[Lem.
	3.5]{Sch19}. Therefore, for $E$ there are no other walls except the unique one
	vertical wall.

Note that $\nu_{\frac 14,\frac{2-\sqrt{5}}{4}}(E(1))=0$, therefore,
applying Proposition \ref{heart} to $E(1)$, we obtain that $E(1)$ or $E(1)[1]$
belong to the category $\mathfrak
C$. Solving the equation $\ch(E(1))=a\ch(\OO_X(-1)[3])+b\ch(\SS(-1)[2])+c\ch(\OO_X[1 ])+d\ch(\OO_X(1))$, where
$a,b,c,d\in\mathbb Z$ are simultaneously non-negative or non-positive,
we obtain the required bound for $\ch_3(E)$,
and for its maximum value we find $a=d=0,b=-1,c=-4$. We get an exact triple
$0\to \SS(-2)\to\OO_X(-1)^{\oplus 4}\to E\to 0.$

A calculation shows that the equation $\nu_{\alpha,\beta}(E)=
\nu_{\alpha,\beta}(\OO_X(-1))$ does not hold for any $(\alpha,\beta)$, and
$\nu_{\alpha,\beta}(E)>\nu_{\alpha,\beta}(\OO_X(-1))$. Therefore there is no morphism
$E\onto\OO_X(-1)$. Since $\Hom(\SS(-2),\OO_X(-1))\cong\C^4$, it follows that
any two injective morphisms $\SS(-2)\to\OO_X(-1)^{ \oplus 4}$ with
tilt-semistable cokernel $E$ lie in the same orbit of the $GL(4)$-action. So
there is a unique object $E$ up to isomorphism, and the exact triple \eqref{spin tr} shows that $E\cong \SS(-1)$.
\end{proof}

\begin{lemma}\label{2,0,-1} Let $E\in\Coh^\beta(X)$ be a tilt-semistable
	object with $\ch(E)=(2,0,-\frac 12H^2,e)$. Then $e\le -\frac 12$.
\end{lemma}

\begin{proof} Suppose that there is a semicircular wall intersecting the line
	$\beta=-1$. We have $\ch_{\le 2}^{-1}(E)=(2,2H,\frac{H^2}{2})$, so in this case
	there is a destabilizing
	sheaf $F\into E$ with $\ch_{\le 2}(F)=(R,H,xH^2)$. From the condition $\nu_{\alpha,-1}(F)=
	\nu_{\alpha,-1}(E)$ we obtain $x-\frac 14=\alpha^2(R-1)$. The case $R=1$ is impossible, since
	$x\in\frac 12\mathbb Z$; for $R\ge 2$ we have $x>\frac 14$, and this
	contradicts the condition $\Delta(F)\ge 0$; if $R\le 0$, then instead of $F$ we can work
	with the sheaf $E/F$.
	
Note that from the tilt semistability of $E$ it follows that $H^0(E)=0$. Also
$H^2(E)\cong\Hom(E,\omega_X[1])=0$, since if a nonzero morphism
$E\to\omega_X[1]$ exists we would have a semicircular wall intersecting the line $\beta=-1$. Todd's class
for $X$ has the form $\mathrm{td}(T_X)=(1,\frac{3H}{2},\frac{13H^2}{12},1).$
By the Riemann--Roch Theorem we find that $e+\frac{1}{2}=\chi(E)=-h^1(E)-h^3(E)\leq 0,$ that is
$e\le-\frac 12$.
\end{proof}

\begin{lemma}\label{lemma 3.5} Let $E\in\Coh^\beta(X)$ be a tilt-semistable
	object with $\ch(E)=(2,0,-H^2,e)$. Then $e\le 1$, and if $e=1$ then $E$
	is destabilized by an exact triple $0\to\OO_X(-1)^{\oplus 6}\to E\to\SS(-2)^{\oplus 2}[
	1]\to 0.$
\end{lemma}

\begin{proof} We have $\beta<0$, and any semicircular wall intersects the line
$\beta=-1$. Since $\ch_{\le 2}^{-1}(E)=(2,2H,0)$, then any destabilizing object $F\into E$ has $\ch^{-1}(F)=(R,H,xH^2,e')$. Then the equality $\nu_{\alpha,-1}(F)=\nu_{\alpha,-1}(E)$ takes the form $x=(R-1)\frac{\alpha^2}{2}$.
	If $R=1$, then $x=0$. Since $\ch_1^{-1}(F)$ is minimal, the object $F$ is
	tilt-semistable for all $\alpha>0,\beta<0$.
	
Since $F$ is stable, $H^0(F)=0$. If $H^2(F)\neq 0$, then, by Serre duality,
there is a nontrivial morphism $F\to\mathcal O_X(-3)[1]$. However, such a morphism
would define a semicircular wall, so $H^2(F)=0$.

By the Riemann--Roch Theorem we find that $e'-\frac{1}{2}=\chi(F)=-h^1(F)-h^3(F)\leq 0.$
Thus, $e'\leq\frac 12$. Since $\ch_{\le 2}(E/F)=\ch_{\le 2}(F)$, also
$\ch_3(E/F)\leq\frac 12$. Therefore, in this case $e\le 1$, and the equality
is achieved at $e'=\frac 12$. Applying Proposition \ref{heart}, we get for this
case an exact triple $0\to\OO_X(-1)^{\oplus 3}\to F\to \SS(-2)[1]\to 0,$
and a similar exact triple for $E/F$, from which the statement of the lemma follows for $R=1$.
	
If $R\neq 1$, then, without loss of generality, we can assume that $R\geq 2$,
then $x>0$. However, the equation $\Delta(F)\ge 0$ implies that $x\le\frac 14$, which is
impossible, since $x\in\frac 12\mathbb Z$. Finally, if $E$ has no destabilizing
subobjects defining a semicircular wall, then, carrying out the argument for $E$,
similar to the argument for $F$ in the case $R=1$, we obtain the required statement.
\end{proof}

\begin{lemma}\label{2,0,-3} Let $E\in\Coh^\beta(X)$ and $\ch(E)=(2,0,-\frac
	32H^2,e)$. Then $e\le \frac 52$, and if $e=\frac 52$, then $E$ is
	destabilized by an exact triple $0\to\SS(-1)\to E\to\OO_{Q_2}(-1)\to 0.$
\end{lemma}

\begin{proof}We have $W_{0,-1}(E)=15-6e$, and if $e\ge\frac 52$, then $E$
	destabilizes at $\beta=-1$, or the point $\beta=-1,\alpha=0$ is a
	limit point of a semicircular wall. By calculation we get $\ch^{-1}(E)=(2,2H,
	-\frac 12H,e-\frac 73)$, which means that there is a destabilizing subobject or a
	quotient $F$ with $\ch_1^{-1}(F)\in\{0,H\}$. If $\ch^{-1}(F)=(R,H,xH^2,e')$, then
	the relation $\nu_{\alpha,1}(E)=\nu_{\alpha,1}(F)$ is equivalent to the equality $x+\frac
	14=\alpha^2(R-1)$. We can assume that $R>1$; then $\Delta(F)\ge 0$ implies
	$x=0$. We have $\ch^{-1}(E/F)=(2-R,H,-\frac 12H^2,e'')$, and from the inequality
	$\Delta(E/F)\ge 0$ we obtain $R\in\{2,3\}$.
	
Let $R=2$. From Lemma \ref{spinor} it follows that $e'\le -\frac 16$, and in the case
of equality $F\cong \SS(-1)$. We have $\ch(E/F)=(0,H,-\frac 32H^2,e''+2)$, and any wall
for $E/F$ intersects the line $\beta=-\frac 32$. Since $\ch_1^{-\frac 32}(E/F)=H$,
then any destabilizing subobject $F'\into E/F$ has $\ch_{\le 2}^{-\frac 32}(F'
)=(r,\frac 12H,xH^2)$. From here we find $0=2x-\alpha^2 r$, therefore assuming that
$r>0$, we get the only possibility $r=1,x=\frac 18$. From Lemma \ref{trivial}
it follows that $\ch_3(F')\le -\frac 13$ (in the case of equality $F'\cong\OO_Q(-1)$).
In addition, $\ch_{\le 2}((E/F)/F')=(-1.2H,-2H^2)$, so $\ch_3$ of this object is
less than or equal to $\frac 83$. From these inequalities we find $e\le\frac 52$. When
equality holds, we can assume that both objects $F$ and $E/F$ are $\nu_{\alpha,\beta}
$-semistable, which implies that $E/F\cong\OO_{Q_2}(-1)$. In this case
we have $\Ext^1(F,E/F)=0$, therefore $F$ is indeed a subobject,
and not a quotient of the object $E$. This vanishing is proved using the definition of 
an exceptional collection and the exact triple \eqref{spin tr}.
	
If $R=3$, we denote the destabilizing subobject by $F'\into E$, and $\ch(F')=
(3,-2H,\frac 12H^2,e')$. Any semicircular wall must intersect the line
$\beta=-1$, but $\ch^{-1}_1(F')=H$. Hence $F'(1)[1]\in\mathfrak C$, and a calculation
gives $e'\le-\frac 16$. In addition, $\ch(E/F')=(-1,2H,-2H^2,e'')$, where $e''\le
\frac 83$. From these bounds we obtain $e\le\frac 52$. In case of equality
we obtain exact triples $0\to F'\to E\to \OO_X(-2)[1]\to 0,$ \ \ $0\to \OO_X(-1)^{
	\oplus 5}\to F'\to \SS(-2)[1]\to 0.$ From the last exact triple it follows that $F'\cong\OO_X(-1)\oplus F''$,
and, as in the proof of the lemma \ref{spinor}, $\nu_{\alpha,\beta}(\OO_X(-1))\neq
\nu_{\alpha,\beta}(F'')$, so the object $F'$ is not tilt-semistable.
We find that $e<\frac 52$ in the case of $R=3$.
	
Finally, if $\ch^{-1}_{\le 2}(F)=(r,0,xH^2)$ and $\nu_{\alpha,\beta}(F)=\nu_{\ alpha,\beta}(E)$
on a semicircular wall having a limit point $\beta=-1,\alpha=0$, then the slope
$\nu_{\alpha,\beta}(F)$ must tend to a finite limit when we tend
to this point along the wall, so $x=0$. From the inequality $0\le\Delta(E/F)\le\Delta(E)$
it follows that $0\le r\le 6$. Cases with $r\le 5$ reduce to the case $R=2$ discussed above, 
while for $r=6$ we get $(E/F)(1)[1]\in\mathfrak C$, and
a calculation gives $\ch_3((E/F)(1))\le -\frac 56$, whence $\ch_3(E)\le\frac32$.
\end{proof}

\begin{lemma}\label{2,-1,-1} Let $E\in\Coh^\beta(X)$ be a tilt-semistable
	object with $\ch(E)=(2,-H,-\frac 12H^2,e)$. Then $e\le \frac 53$, and if $e=\frac
	53$, then $E$ is destabilized by an exact triple $0\to\OO_X(-1)^{\oplus 3}\to E\to
	\OO_X(-2)[1]\to 0.$
\end{lemma}

\begin{proof} Assume that $e\ge\frac 53$. Then $W_{0,-\frac 32}(E)<0$, and
	this means that there is a wall intersecting the line $\beta=-\frac 32$. In addition, $\ch_{\le 2}^{-\frac
		32}(E)=(2,2H,\frac 14H^2)$, so there is a destabilizing subobject or
	quotient $F$ of the object $E$ with $\ch_1^{-\frac 32}(F)\in\{\frac 12H,H\}.$ If
	$\ch_{\le 2}^{-\frac 32}(F)=(R,\frac 12H,xH^2)$, then from the equality $\nu_{\alpha,-
		\frac 32}(E)=\nu_{\alpha,-\frac 32}(F)$ we get $2x-\frac 18=\alpha^2(2R-1)$.
	Note that $R$ must be odd since $H^2\cdot \ch^{-\frac 32}_1(F)$ is
	odd. Then the inequalities $\Delta(F)\ge 0,\Delta(E/F)\ge 0$ imply $R\in\{-1,1\}$.
	If $R=-1$, then $x=-\frac 18$ and $\ch_3$ is maximized by the object $F\cong\OO_X(-2)[1]
	$. If $R=1$, then $x=\frac 18$ and $F\cong\OO_X(1)$. If $\ch_{\le 2}^{-\frac
		32}(F)=(R,H,xH^2)$, then $x-\frac 18=\alpha^2(R-1)$, $R$ is even, and we can assume
	that $R>1$. But then $R=2,\ x=\frac 14,\ F\cong\OO_X(-1)^{\oplus 2}$. In any case
	we obtain the required inequality and the exact triple.
\end{proof}

\begin{lemma}\label{G_d} For $d\in\Z$ let $G_d\in\Coh^\beta(X)$ be a
	tilt-semistable object with $\ch(G_d)=(0,H,dH^2,e)$. Then the inequality
	$e\le d^2-\frac 16$ holds, and in case of equality we have $G_d\cong\II_{\P^1,Q_2}(d+1)$ for
	some two-dimensional quadric $Q_2\subset X$ and a line $\P^1\subset Q_2$.
\end{lemma}

\begin{proof} The curve $\nu_{\alpha, \beta}(G_d)=0$ is a straight line $\beta=d$, and
	there are no semicircular walls, so $(G_d)(-d)[1]\in\mathfrak C$. A calculation
	gives the desired inequality and an exact triple
	\begin{equation}\label{Gd}
		0\to \SS(d-1)\to\OO_X(d)^{\oplus 2}\to G_d\to 0.
	\end{equation}

As is known, the scheme of zeros $(s)_0$ of any section $0\ne s\in H^0(\SS)$ is
a reduced projective line $\P^1\subset X$, from which we obtain an exact triple
$0\to\OO_X\xrightarrow{s}\SS\to\II_{\P^1}(1)\to0$. Hence $\coker(\OO_X^{\oplus 2}
\xrightarrow{s,s'}\SS)=\coker(\OO_X\xrightarrow{\phi}\II_{\P^1}(1)$. Consider
the two-dimensional quadric $Q_2=\mathrm{supp}\coker(\phi)$. The morphism $\II_{Q_2}(1)\cong\OO_X
\xrightarrow{\phi}\II_{\P^1}(1)$ induces an exact triple $0\to\II_{Q_2}(1)
\xrightarrow{\phi}\II_{\P^1}(1)\to\II_{\P^1,Q_2}(1)\to0$. This implies the exactness
of the triple $0\to\OO_X^{\oplus 2}\to\SS\to\II_{\P^1,Q_2}(1)\to0$. Applying the functor $\mathbf{R}\HH om(\cdot,\OO_X(d))$ 
to it and taking into account the easily verifiable isomophism $\lExt^1(\II_{
\P^1,Q_2}(1),\OO_X(d))\cong\II_{\P^1,Q_2}(d+1)$, we obtain an exact triple
$0\to\SS(d-1)\to\mathcal O(d)^{\oplus 2}\to\II_{\P^1,Q_2}(d+1)\to 0$. This triple
by construction is isomorphic to the triple \eqref{Gd} with a suitable choice of sections $s,s'$,
hence $G_d\cong\II_{\P^1,Q_2}(d+1)$.
\end{proof}

\begin{lemma}\label{2,0,-4} Let $E\in\Coh^\beta(X)$ be a tilt-semistable object
	with $\ch(E)=(2,0,-2H^2,e)$. Then $e\le 4$, and if $e=4$, then $E$ is destabilized by one
	of the exact triples
	\begin{equation}\label{triple3}
		0\to \OO_X(-1)^{\oplus 4}\to E\to\mathcal O_X(-2)^{\oplus 2}[1]\to 0,
	\end{equation}
	\begin{equation}\label{triple4}
		0\to\SS(-1)\to E\to \II_{\P^1,Q_2}(-1)\to 0,\quad 0\to\II_{\P^1,Q_2}(-1)\to E\to\SS(-1)\to 0.
	\end{equation}
	In the case when the object $E$ is included in the first exact triple (\ref{triple4}), it is
	tilt-stable for $\beta<0,\alpha\gg 0$ and is not included in the exact triple (\ref{triple3}).
\end{lemma}

\begin{proof}
	We have $W_{0,-1}(E)=24-6e$, so the inequality $W_{0,-1}\ge 0$ is equivalent to $e\le 4$.
	Therefore, it is sufficient to check that $e\le 4$ for objects that have a
	semicircular wall at $\beta=-1$. We have $\ch_{\le2}^{-1}(E)=(2,2H,-H^2)$, so
	in this case there is a destabilizing subobject or a quotient $F$ of the object $E$
	with $\ch_1^{-1}(F)=H$. Let $\ch_{\le2}^{-1}(F)=(R,H,xH^2)$, then $\nu_{\alpha,-1}(F)
	=\nu_{\alpha,-1}(E)$ implies $x+\frac 12=\frac{\alpha^2}{2}(R-1)$. We can assume that
	$R\ge2$, and then the inequality $\Delta(F)\ge 0$ implies $x\in\{-\frac 12,0\}$.
	
Assume that $x=-\frac 12$. From the inequality $\Delta(E/F)\ge 0$ we obtain $R\in\{2,3\}$.
If $R=2$, then $\ch_{\le 2}(F)=(2,-H,-\frac 12H^2)$ and $\ch_{\le2}(E/F)=(0 ,H,-\frac 32H^2
)$ (we assume, by abuse of notation, that $F$ is a subobject of $E$).
Therefore, using Lemma \ref{2,-1,-1} and the proof of Lemma \ref{2,0,-3}, we obtain
the desired inequality and one of the exact triples
\begin{equation}\label{triple1}
	0\to F\to E\to\OO_{Q_2}(-1)\to 0,
\end{equation}
\begin{equation}\label{triple2}
	0\to \OO_{Q_2}(-1)\to E\to F\to 0,
\end{equation}
where $F\in M(2,-H,-\frac 12H^2,\frac 53).$ By Lemma \ref{2,-1,-1} we have an exact triple
$0\to\OO_X(-1)^{\oplus 3}\to F\to\OO_X(-2)[1]\to 0$, there is also an exact triple $0\to\OO_X(-1)\to\OO_{Q_2}(-1)\to\OO_X(-2)[1]\to 0$. 
A consideration of long exact sequences of groups $\Ext$ shows that $\Ext^1(\OO_X(-1),F)=\Ext^1(\OO_X(-1),\OO_{Q_2}(-1))=0$, we obtain a morphism $F'=\OO_X(-1)^{\oplus4}\cong
\OO_X(-1)^{\oplus 3}\oplus\OO_X(-1)\to E$. By the Snake Lemma we find that this morphism is injective and $E/F'\cong\OO_X(-2)^{\oplus 2}[1].$
	
If $R=3$, then $\ch_{\le 2}(F)=(3,-2H,0)$ and $\ch_3(E/F)$ is maximized by the object
$\OO_X(-2)[1]$. The curve $\nu_{\alpha,\beta}(F)=0$ intersects the $\beta$-axis at the point
$\beta=-\frac 43$, therefore any semicircular wall for $F$ intersects the line $\beta=
-\frac 43$. We have $\ch_{\le2}^{-\frac 43}(F)=(3,2H,0)$, therefore there is a destabilizing subobject or a quotient $G$ of the object
$F$ with $\ch_1^{-\frac 43}(G)\in\{\frac 13H,\frac 23H,H\}$. In the first case we have $G\cong\OO_X(-1)$, in the second case $G\cong\OO_X(-1)^{
	\oplus2}$ or $G\cong\OO_X(-2)[1]$, and in the third $G\cong\OO_X(-1)^{\oplus 3}$. In any
case, we obtain the required inequality and the exact triple (\ref{triple3}).

Let us now assume that $x=0$. Then $R=2$, and we find $\ch_{\le 2}(F)=(2,-H,0)$
and $\ch_{\le 2}(E/F)=(0,H,-2H^2)$. Lemmas \ref{spinor} and \ref{G_d} give the required
inequality and exact triples (\ref{triple4}).
		
We have proven that $e\le 4$. For $e=4$, note that the exact triple (\ref{triple3}) defines
the wall $W_{\alpha,\beta}(E)=0$, so this wall is the smallest, and repeating the arguments from
the proof of \cite[Theorem 5.1]{Sch19} shows that any object
destabilized by this wall is given by the exact triple (\ref{triple3}). Any other
wall intersects the line $\beta=-1$ at a point with $\alpha>0$, therefore, when $R\ge 2$, the above argument applies for them. If $R=1$, then, due to the fact that
$W_{0,-1}(F)=6-3\ch_3(F)$ and $\ch_1^{-1}(F)$ is minimal, we find that $\ch_3(F)\le 2$ and
the wall $W_{\alpha,\beta}(F)=0$ is unique. This wall is given by the exact triple
$0\to\OO_X(-1)^{\oplus 2}\to F\to\OO_X(-2)[1]\to 0$, from which, together with a similar exact
triple for $E/F$, by the Snake Lemma we obtain a monomorphism $\OO_X(-1)^{\oplus 4}\hookrightarrow E$ and the exact triple (\ref{triple3}).
	
To prove the statement about the tilt stability of an object $E$ from the first exact triple in (\ref{triple4})
it is enough to check its tilt stability for $\beta=-1,\alpha>1$. Because
$\nu_{1,-1}(\SS(-1))=\nu_{1,-1}(\II_{\P^1,Q_2}(-1))$, the object $E$ is
$\nu_{1,-1}$-semistable. If $F$ is a subobject that destabilizes $E$ at $\beta=-1$,
then $\ch_1^{-1}(F)=H$ (the case $\ch_1^{-1}(F)=0$ is impossible, since in this case
$\nu_{1,-1}(F)=+\infty$, which contradicts $\nu_{1,-1}$-semistability of
$E$, and in the case $\ch_1^{-1}(F)=2H$ we have
$\nu_{\alpha,-1}(F)<\nu_{\alpha,-1}(E/F)=+\infty$, so $F$ does not destabilize $E$).
Repeating the argument carried out above in the proof of this Lemma, we obtain that either
$F\cong\SS(-1)$, or $E$ is included in an exact triple of the form (\ref{triple3}).
Direct calculation shows that $\SS(-1)$ does not destabilize $E$ when
$\beta=-1,\alpha>1$, and for any object $E$ from (\ref{triple3}) we have $\Hom(\SS(-1),E)=
\Hom(E,\II_{\P^1,Q_2}(-1))=0$, that is, $E$ cannot be included simultaneously in
(\ref{triple3}) and in the first exact triple (\ref{triple4}).
\end{proof}

\begin{lemma}\label{2,-1,-2} Let $E\in\Coh^\beta(X)$ be a tilt-semistable
	object with $\ch(E)=(2,-H,-H^2,e)$. Then $e\le\frac{19}{6}$, and if $e=\frac{19}{6}
	$, then $E$ is destabilized by an exact triple
	\begin{equation}\label{triple5}
		0\to\mathcal O_X(-1)^{\oplus 2}\to E\to \II_{\P^1, Q_2 }(-1)\to 0.
	\end{equation}
\end{lemma}

\begin{proof} Suppose that $e\ge\frac{19}{6}$. Then $W_{0,-\frac 32}(E)<0$, and
	we have $\ch_{\le 2}^{-\frac 32}(E)=(2,2H,-\frac 14H^2)$, therefore there is
	a destabilizing subobject or quotient $F$ of $E$ with $\ch_1^{-\frac 32}(F)\in
	\{\frac 12H,H\}$. Suppose that $\ch_{\le 2}^{-\frac 32}(F)=(R,H,xH^2)$; Then
	the equality $\nu_{\alpha,-\frac 32}(F)=\nu_{\alpha,-\frac 32}(E)$ implies $x+\frac 18
	=\alpha^2(R-1)$. Suppose that $R>1$, then $-\frac 18<x\le\frac 14$. If
	$x=0$, then $\Delta(E/F)\ge 0$, since Chern classes are integers it follows that $R=4$,
	whence we obtain in the usual way that $e\le\frac 76$. The case $x=\frac 18$ is impossible
	by numerical considerations, and if $x=\frac 14$, then $R=2$, and we get the required
	inequality and the exact triple (\ref{triple5}).

Suppose that $\ch_{\le 2}^{-\frac 32}(F)=(R,\frac 12H,xH^2)$. Then the equality
$\nu_{\alpha,-\frac 32}(F)=\nu_{\alpha,-\frac 32}(E)$ implies $2x+\frac 18=\alpha^2
(2R-1)$. For $R\ge 1$ there is only one possibility $R=1,\ x=\frac 18$, and then
$\ch_3(F)$ is maximized by the sheaf $F\cong\mathcal O_X(-1)$, and we also obtain the exact
triple (\ref{triple5}). For $R<1$ there is only one possibility $R=-1,x=-\frac
18$, from where $F\cong\OO_X(-2)[1]$. In this case we also get $e\le\frac 76$.
\end{proof}

Now we can proceed to the general case.

\begin{lemma}\label{our bound} Let $E\in\Coh^\beta(X)$ be a tilt-semistable
	object with $\ch(E)=(2,cH,dH^2,e)$. Assume that either
	\begin{enumerate}
		\item $c=-1,d\le -\frac 32$, and $e\ge d^2-2d+\frac 16$, or
		\item $c=0,d\le -\frac 52$, and $e\ge d^2.$
	\end{enumerate}
	Then $E$ is destabilized along a semicircular wall by a subobject or
	quotient of rank not more than two.
\end{lemma}

\begin{proof} (1) Suppose that $c=-1,\ d\le -\frac 32$, and $e\ge d^2-2d+\frac 16$. Then a calculation gives
	$$\frac{\rho_W(E)}{\Delta(E)}\ge -\frac{9}{64}d-\frac{19}{256}-\frac{9}{128(1- 4d)}
	-\frac{27}{64(1-4d)^2}+\frac{81}{256(1-4d)^3}.$$
	Bounding each term, we find $\frac{\rho_W(E)}{\Delta(E)}>\frac{1}{12} $, which means we can apply Proposition \ref{radius}.
	
	(2) Suppose that $c=0,\ d\le -\frac 52$, and $e\ge d^2.$ A calculation gives
	$\frac{\rho_W(E)}{\Delta(E)}\ge -\frac{9}{64}d-\frac{1}{4},$
	and applying Proposition \ref{radius} completes the proof.
\end{proof}

\begin{lemma}\label{rank one} Let $E\in\Coh^\beta(X)$ be a tilt-semistable
	object with $\ch(E)=(2,cH,dH^2,e)$.\\
	(1) Let $c=-1$ and $d\le-\tfrac 32$. If $E$ is destabilized by a subobject $F$ of
	rank one, then $e\le d^2-2d+\tfrac{5}{12}$. If $e=d^2-2d+\tfrac{5}{12}$ with
	non-integer $d$, or $e=d^2-2d+\tfrac 16$ with integer $d$, then $F\cong\OO_X(-1)$ and $E$ is
	destabilized by a subobject of rank two. \\
	(2) Let $c=0$ and $d\le-\tfrac 52$. If $E$ is destabilized by a subobject or
	quotient $F$ of rank one, then $e<d^2$.
\end{lemma}

\begin{proof} (1) Suppose that $c=-1,\ d\le-\tfrac 32$ and $e\ge d^2-2d+\tfrac 16$. 
	A calculation gives $W_{0,-\frac 32}(E)=4d^2-6d+\frac 94-6e\le -2d^2+6d+\frac 54<0$, and therefore
	$$0<H^2\cdot\ch_1^{-\frac 32}(F)<H^2\cdot\ch_1^{-\frac 32}(E)=2.$$
	
	Hence $\ch_1(F)\in\{-H,0\}$. Assume that $\ch_1(F)=-H$. Let us show that
	$F$ is a subobject, not a quotient, of the object $E$. We have $\ch(F)=
	(1,0,-yH^2,z)\cdot\ch(\mathcal O_X(-1))$, and Remark \ref{e le d(d+1)}
	shows that $z\le y(y+1)$. A calculation gives
	$$s_W(E)=\frac{d+3e}{4d-1}\le\frac{6d^2-10d+1}{8d-2},\ \ \ \ s(E,F)=d +2y-1.$$
	Since $s(E,F)\le s_W(E)$, we have
	$$y\le\frac{d^2+1}{2-8d}<-\frac d2-\frac 14.$$

Using Remark \ref{e le d(d+1)} for the quotient $E/F$, we find
$$e\le(\frac 12-d-y)(\frac 32-d-y)+y(y+1)+2y-\frac 13.$$

The right-hand side of this inequality is a quadratic function of $y$ with a minimum at the point
$y=-\frac d2-\frac 14$. Therefore, the maximum is achieved at the point $y=0$ with $F\cong\OO_Q
(-1)$, and we get $e\le d^2-2d+\tfrac{5}{12}$. Note also that for $y>0$
we have $e\le d^2-d+\frac{17}{12}<d^2-2d+\frac 16$, so for the integer $d$ we also find that $y=0$. Now from the proof of \cite[Prop. 3.2]{MS18} it follows that
there is a destabilizing morphism $\OO_X(-1)\to G=E/F$, and the quotient has the
Chern character $\ch_{\le 2}(G')=(0,H,(d-1)H^2)$. The argument from the proofs of Lemmas
\ref{2,0,-3} and \ref{2,0,-4} shows that for the non-integer $d$ the character $\ch_3(G')$
is maximized by the sheaf $G'\cong\OO_{Q_2}(d-\frac 12)$, while for the integer $d$
there is an exact triple
$$0\to \SS(d-2)\to\OO_X(d-1)^{\oplus 2}\to G'\to 0.$$
A calculation shows that in any case $\Ext^1(F,G)=0$, so $F$ is indeed
a subobject of $E$ and there is a destabilizing morphism $\OO_X(-1)^{\oplus 2}\to E$.

(2) Suppose that $c=0,\ d\le-\frac 52$ and $e\ge d^2$. By calculation we get
$W_{0,-1}(E)=4d^2-4d-6e\le-2d^2-4d<0,$ and therefore
$$0<H^2\cdot\ch_1^{-1}(F)<H^2\cdot\ch_1^{-1}(E)=2.$$
This means that $\ch_1(F)=0$, but such $F$ does not define a semicircular wall.
\end{proof}

\textit{Proof of Theorem \ref{Theorem 3.1}.} The inequalities on $d$ follow from
the first inequality in Proposition \ref{Prop 2.1}.(ii). Other statements of the Theorem are proved by induction on $\Delta(E)$, and the start of the induction is given by Lemmas
\ref{trivial} -- \ref{2,-1,-2}. Lemma \ref{our bound} shows that $E$
is destabilized by a subobject $F\into E$, where $F$ has rank $R\in\{0,1,2\}$. Lemma
\ref{rank one} deals with the case $R=1$. Let us assume that $R=2$. In what follows we
will show that the case $R=0$ is realized only if $E$ is a direct sum.

(1) Suppose that $c=-1,\ d\le-\frac 32,$ and $e\ge d^2-2d+\frac 16$. As in
the proof of Lemma \ref{rank one}, we have $W_{0,-\frac 32}(E)<0$, and also
$W_{0,-2}(E)<0$. Then the inequalities $H^2\cdot\ch_1^{-3/2}>0$ and
$H^2\cdot\ch_1^{-2}(F)<H^2\cdot\ch_1^{-2}(E)$ together imply that $\ch_1(F)=-2H$.
Therefore, we can assume that $\ch(F)=(2,0,yH^2,z)\cdot\ch(\mathcal O_X(-1))$.
If $y=0$, then $z\le 0$, and the arguments from the proof of Lemma \ref{rank one} give the
required statement. Let us assume that $y\le-\frac 12$. By the induction hypothesis
we have $z\le y^2$. A calculation gives:
$$s_W(E)=\frac{d+3e}{4d-1}\le\frac{6d^2-10d+1}{8d-2},\ \ \ \ s(E,F)=d-y -1.$$
Since $s(E,F)\le s_W(E),$ then $y\ge\frac{2d^2+1}{8d-2}>\frac d2.$ Maximizing
$\ch_3(E/F)$, we get
$$e\le (d-y)^2-2(d-y)+y^2-2y+C,$$
where $C$ is a constant depending on whether $d$ and $y$ are integers or
non-integers. For $y\in\Z$ or $y\in\frac 12+\Z$ the right hand side of the inequality is
a quadratic function of $y$ with the minimum at the point $y=\frac d2$. Therefore
the maximum is achieved at $y=-\frac 12$ or $y=-1$. For $y=-\frac 12$ by calculation
we get $e\le d^2-d+\frac{11}{12}<d^2-2d+\frac 16$, and for $y=-1$ we find $e\le
d^2+\frac{29}{12}<d^2-2d+\frac 16$.
	
(2) Suppose that $c=0,\ d\le-\frac 52$ and $e\ge d^2$. As in the proof
of Lemma \ref{rank one}, we have $W_{0,-1}(E)<0,$ which implies that $\ch_1(F)=-H$.
Suppose $\ch(F)=(2,-H,yH^2,z)$. By the induction hypothesis we have
$z\le y^2-2y+\frac{5}{12}.$ A calculation gives:
$$s_W(E)=\frac{3e}{4d}\le\frac 34d,\ \ \ \ s(E,F)=d-y,$$
hence $y\ge\frac d4>\frac d2+\frac 12.$ Maximizing $\ch_3(E/F)$, we find
$$e\le 2y^2-(2d+2)y+C,$$
where $C$ is a constant depending on whether $d$ and $y$ are integers or
non-integers. For $y\in\Z$ or $y\in\frac 12+\Z$ the right hand side of the inequality is
a quadratic function of $y$ with the minimum at the point $y=\frac d2+\frac 12$. Therefore
the maximum is achieved at $y=0$ or $y=-\frac 12$. For $y=-\frac 12$ we get
$e\le d^2+d+2<d^2$, which implies that $y=0$. It remains to note that
the equality $\Ext^1(F,E/F)=0$ holds, verified by the exact triples defining $F$ and $E/F$.

(3) and (4). A direct consequence of the statements (1) and (2), respectively.
\qed

\vspace{1cm}

\section{Infinite series of components of moduli spaces of semistable sheaves of rank two on varieties $X_1$, $X_2$, $X_4$ and $X_5$}

\subsection{First series.}\label{subsec 4.1}
In this susection we consider extensions of one of the following two
forms. \\
(1) Extensions on $X=X_1,\ X_2,\ X_4$ or $X_5$ of the form
\begin{equation}\label{extension}
	\begin{split}
		& 0\to\OO_X(-n)^{\oplus2}\to E\to\OO_S(m)\to 0,\ \ \ \ \ \ \ \ \ \ \ S\in|\OO_X(k)| ,\\
		& \text{where} \ \ \ \ \ \ k\ge1, \ \ \ \ \ n=\lceil\frac k2\rceil,\ \ \ \ \ \ m+n<0,
	\end{split}
\end{equation}
(2) or an extension on the quadric $X=X_2$ of the form
\begin{equation}\label{extension I}
	\begin{split}
		& 0\to\OO_X(-1)^{\oplus 2}\to E\to\II(m)\to 0,\ \ \ \ \ \ \ \ \ \ \
		S\in|\OO_X(1)|,\\
		& \text{where} \ \ \ \ \ \II=\II_{\P^1,S},\ \ \ \ \ \ m<0,
	\end{split}
\end{equation}
$\P^1$ is a projective line on a two-dimensional quadric $S$, and the sheaf $\II$ is
described in \eqref{spin tr on X}.
This section mainly deals with extensions \eqref{extension}. In particular, for $E$ from \eqref{extension} we have:
\begin{equation}\label{k even}
	\begin{split}
		& \rk(E)=2, \ \ \ c_1(E)=0,\ \ \ c_2(E)=(\frac{k^2}{4}-km)H^2,\\
		& c_3(E)=(\frac{k^3}{4}-k^2m+km^2)H^3,\ \ \ \ \ \ \
		\text{if}\ \ \ k\ \ \ \text{is even},
	\end{split}
\end{equation}
\begin{equation}\label{k odd}
	\begin{split}
		& c_1(E)=-H,\ \ \ c_2(E)=(\frac{(k-1)^2}{4}-km)H^2,\ \ \ c_3(E)=\\
		& (\frac{k^3}{4}-k^2m+km^2)H^3,\ \ \ \ \ \ \text{if}\ \ \ k\ \ \ \text{is odd},
	\end{split}
\end{equation}
where $H$ is the cohomology class of a hyperplane section $X$.

It is well known and verified by standard calculation that
\begin{equation}\label{zero cohom}
	h^i(\OO_X(\ell))=0\; \ \ \ \ \ \ \ \ \ \ell\in\Z,\ \ i=1,2.
\end{equation}
Let us now assume that
\begin{equation}\label{E stable}
	\textit{the sheaf}\ E\ \textit{in}\ \eqref{extension}\ \textit{is stable}.
\end{equation}
In this case, from \eqref{spin tr on X}, \eqref{extension},
\eqref{extension I}, \eqref{zero cohom}, the exact triples
\begin{equation}\label{tr X,S}
	0\to\OO_X(-k)\to\OO_X\to\OO_S\to0,\ \ \ \ \ \ \ \
\end{equation}
\begin{equation}\label{tr I,S}
	\ \ \ \ \ \ \ 0\to\II\to\OO_S\to\OO_{\P^1}\to0,\ \ \ \ \ \ \ \ \ \ \ X=X_2,\ \ \ \ \ \ \ \ ,
\end{equation}
the Serre--Grothendieck duality and the stability of $E$ we get the equalities\\
(a) for the  extension \eqref{extension}:
\begin{equation}\label{vanish ext0}
	\begin{split}
		& \hom(E,E)=\hom(\OO_S(m),\OO_S(m))=\hom(E,\OO_S(m))=1,\\
		& \rmext^1(\OO_S(m),\OO_S(m))=N,
	\end{split}
\end{equation}
\begin{equation}\label{vanish ext}
	\begin{split}
		& \hom(\OO_X(-n)^{\oplus 2},\OO_S(m))=0,\\
		& \rmext^1(\OO_X(-n)^{\oplus 2},E)=\rmext^1(\OO_X(-n)
		^{\oplus 2},E)=\rmext^1(\OO_X(-n)^{\oplus 2},\OO_S(m))=\\
		& \rmext^1(\OO_X(-n)^{\oplus 2},\OO_X(-n)^{\oplus 2})=
		\rmext^i(\OO_S(m),\OO_X(-n)^{\oplus 2})=0,\ \ i=0.2.
	\end{split}
\end{equation}
(b) for the extension \eqref{extension I}:
\begin{equation}\label{vanish ext0I}
	\begin{split}
		& \hom(E,E)=\hom(\OO_S(m),\OO_S(m))=\hom(E,\OO_S(m))=1,\\
		& \rmext^1(\OO_S(m),\OO_S(m))=4,
	\end{split}
\end{equation}
\begin{equation}\label{vanish ext I}
	\begin{split}
		& \hom(\II(m),\II(m))=\hom(E,\II(m))=1,\ \ \hom(\OO_X(-n)^{\oplus2},\II (m))=0,\\
		& \rmext^i(\OO_X(-1)^{\oplus2},\II(m))=0,\ i=0,1,
	\end{split}
\end{equation}
\begin{equation}\label{vanish ext I2}
	\hom(\OO_X(-1)^{\oplus2},\II(1))=4,\ \ \ \hom(\SS(-1),\II(1))=7,\ \ \ \
	\rmext^1(\OO_X^{\oplus2},\II(1))=0,
\end{equation}
\begin{equation}\label{vanish ext3}
\begin{split}
& \hom(\OO_S,\OO_{\P^1})=1,\ \ \ \ \ \rmext^1(\OO_S,\OO_{\P^1})=k+1,\ \ \ \ \
\rmext^1(\OO_S,\OO_S)=N,\\
& \rmext^i(\II(m),\OO_X(-n)^{\oplus 2})=0,\ \ \ \ \ i=0,2.
\end{split}
\end{equation}
Let $\P^N=|\OO_X(k)|$, let $\mathbb{S}\subset\P^N\times X$ be the universal
family of surfaces of degree $k$ in $X$ and let $\P^N\times X\xrightarrow{p}
\P^N$ be the projection. Respectively, in the notation of the triple
\eqref{univ on XxG}, let $\mathbb{G}\times X_2\xrightarrow{\widetilde{p}}\mathbb{G}$ be the projection.
Due to the smoothness of $\P^N$ and $\mathbb{G}$, from
\eqref{vanish ext}-\eqref{vanish ext3}, \cite[Thm. 1.4]{Lan83},
\cite[Satz 3.(ii)]{BPS} there are defined locally free sheaves $\AA$ and $\widetilde{\AA}$ on $\P^N$ and $\mathbb{G}$, respectively:
\begin{equation}\label{sheaf A}
\begin{split}
& \AA=\EE xt_p^1(\OO_{\P^N}\boxtimes\OO_X(m)|_{\mathbb{S}},\OO_{\P^N}
\boxtimes\OO_X(-n)),\\
& \widetilde{\AA}=\EE xt_{\widetilde{p}}^1(\mathbb{I}\otimes\OO_{
\mathbb{G}}\boxtimes\OO_{X_2}(m),\OO_{\mathbb{G}}\boxtimes\OO_{X_2}
(-1)).
\end{split}
\end{equation}
The construction of these sheaves commutes with the base change:
\begin{equation}\label{bch}
	\begin{split}
		& \tau_S:\ \AA\otimes\bk_{S}\xrightarrow{\cong}\Ext^1(\OO_S(m),\OO_X
		(-n)),\ \ \ \ \ \ \ \ \ \ \ S\in|\OO_X(k)|,\\
		& \widetilde{\tau}_x:\ \widetilde{\AA}\otimes\bk_x\xrightarrow{\cong}
		\Ext^1(\II(m),\OO_X(-1)),\ \ \ \ \ \ \ \ \ \ \ \ \ \ \ x\in\mathbb{G},
		\ \ \ \ \ \ S=S_x,
	\end{split}
\end{equation}
and we have
\begin{equation}\label{Ext0=0}
	\begin{split}
		& \EE xt_p^0(\OO_{\P^N}\boxtimes\OO_X(m)|_{\mathbb{S}},\OO_{\P^N}
		\boxtimes\OO_X(-n))=0,\\
		& \EE xt_{\widetilde{p}}^0(\mathbb{I}\otimes\OO_{\mathbb{G}}\boxtimes
		\OO_{X_2}(m),\OO_{\mathbb{G}}\boxtimes\OO_{X_2}(-1))=0.
	\end{split}
\end{equation}
Consider the varieties $P=\P(\AA^{\vee})$, $X_P=P\times X$,
$\widetilde{P}=\P(\widetilde{\AA}^{\vee})$, $X_{\widetilde{P}}=
\widetilde{P}\times X_2$, with the projections $P\xrightarrow{ \rho}\P^N$, $\widetilde{P}\xrightarrow{\widetilde{\rho}}\mathbb{G}$, $P\xleftarrow{p_P}X_P\xrightarrow{\rho_P}\P^N\times X$, $\widetilde{P}\xleftarrow{
	\widetilde{p}_P}X_{\widetilde{P}}\xrightarrow{\widetilde{\rho}_P}
\mathbb{G}\times X_2$. From \eqref{bch}, \eqref{Ext0=0} and
\cite[Cor. 4.5]{Lan83} it follows that on $X_P$ and $X_{\widetilde{P}}$
there are universal families of extensions, respectively:
\begin{equation*}
	\begin{split}
		& 0\to p_P^*\OO_P(1)\otimes(\OO_{\P^N}\boxtimes\OO_X(-n))\to\VV\to
		\rho_P^*(\OO_{\P^N}\boxtimes\OO_X(m)|_{\mathbb{S}})\to0,\\
		& 0\to\widetilde{p}_P^*\OO_{\widetilde{P}}(1)\otimes(\OO_{\mathbb{G}}
		\boxtimes\OO_{X_2}(-1))\to\widetilde{\VV}\to\widetilde{\rho}_P^*(\mathbb{I}\otimes
		\OO_{\mathbb{G}}\boxtimes\OO_{X_2}(m))\to0.
	\end{split}
\end{equation*}
Applying the functors $\rho_{P*}$ and $\widetilde{\rho}_{P*}$ to these triples, respectively, we obtain the universal extensions on $\P^N\times X$ and on $\mathbb{G}\times X$, where $\WW=\rho_{P*}\VV$ and $\widetilde{\WW}=\widetilde {\rho}_{P*}\widetilde{\VV}$:
\begin{equation}\label{univ ext}
	\begin{split}
		& 0\to p^*\AA^{\vee}\otimes(\OO_{\P^N}\boxtimes\OO_X(-n))\to\WW\to\OO_{\P^N}\boxtimes
		\OO_X(m)|_{\mathbb{S}}\to0,\\
		& 0\to\widetilde{p}^*\widetilde{\AA}^{\vee}\otimes(\OO_{\mathbb{G}}\boxtimes\OO_{X_2}(-1))
		\to\widetilde{\WW}\to\mathbb{I}\otimes\OO_{\mathbb{G}}\boxtimes\OO_{X_2}(m)\to0.
	\end{split}
\end{equation}

Note that the extensions \eqref{extension} and \eqref{extension I} are specified by elements $\xi\in\Ext^1(\OO_S(m), \\ \OO_X(-n)^{\oplus 2})$ and $\widetilde{\xi}\in\Ext^1(\II
(m),\OO_{X_2}(-1)^{\oplus 2})$ respectively, considered as \\ homomorphisms
\begin{equation}\label{hom xi}
	\xi:\ \Ext^1(\OO_S(m),\OO_X(-n))^{\vee}\to\bk^2,\ \ \
	\widetilde{\xi}:\ \Ext^1(\II(m),\OO_{X_2}(-1))^{\vee}\to\bk^2.
\end{equation}
We will carry out further argument in this subsection for the upper 
exact triple \eqref{univ ext}.
(For the lower triple \eqref{univ ext} the argument is completely 
similar --- see subsection \ref{subsec 4.2}.) By the universality, 
this triple, restricted to $\{S\}\times X$, in view of the base 
change \eqref{bch} is included, together with the extension 
\eqref{extension}, where $E=E_{\xi}$, into a commutative pushout 
diagram
\begin{equation}\label{diag 83}
\begin{gathered}
\xymatrix{& \Ext^1(\OO_S(m),\OO_X(-n))^{\vee}\otimes\OO_X(-n)\ \ar@{>->}[r] \ar[d]_-{\xi} & \WW|_{\{S\}\times X}\ar@{>>}[r]\ar[d] & \OO_S(m) \ar@{=}[d] \\
& \OO_X(-n)^{\oplus 2}\ \ar@{>->}[r] & E_{\xi}\ar@{>>}[r] & \OO_S(m). }
\end{gathered}
\end{equation}
Note that the homomorphism $\xi$ in \eqref{hom xi} is non-zero, since the extension
\eqref{extension} is non-trivial. Moreover, if $E_{\xi}$ is stable, then $\xi$ is
an epimorphism. Indeed, if $\xi$ has rank 1, then the diagram \eqref{diag 83} goes through a subextension $0\to\OO_X(-n)\to E'\to\OO_S(m)\to0$ of the extension
\eqref{extension}, which implies the existence of a homomorphism $E_{\xi}\twoheadrightarrow
\OO_X(-n)$, contrary to the stability of $E_{\xi}$.

Suppose that $\xi$ in \eqref{diag 83} is an epimorphism. The class $[\xi]$ of the epimorphism $\xi$
modulo automorphisms of the space $\bk^2$ is a point of the grassmannian \\
$\mathrm{Gr}(\Ext^1(\OO_S(m),\OO_X(-n))^{\vee},2)$. Therefore, let us consider the grassmannization
\begin{equation}\label{grasmn}
	\pi:\ {\GG}r(\AA^{\vee},2)\to\P^N
\end{equation}
of two-dimensional quotient spaces of the bundle $\AA^{\vee}$ on $\P^N$ and its open subsets
\begin{equation}\label{def Y}
	\YY_{k,m,n}:=\{[\xi]\in{\GG}r(\AA^{\vee},2)\ |\ E_{\xi}\ \textit{is stable}\},
\end{equation}
\begin{equation}\label{def RR}
\RR_{k,m,n}:=\{[\xi]\in{\GG}r(\AA^{\vee},2)\ |\ E_{\xi}\ \textit{is reflexive}\}.
\end{equation}

\begin{theorem}\label{Theorem 4.9}
	There are embeddings of dense open sets
	\begin{equation}\label{open subsets}
		\RR_{k,m,n}\subset\YY_{k,m,n}\subset{\GG}r(\AA^{\vee},2).
	\end{equation}
\end{theorem}
\begin{proof}
(1) First, let us check the reflexivity of a general sheaf $E$ in the triple \eqref{extension}. 
(In the remaining cases the proofs are completely similar.) Namely, applying to the triple
\eqref{extension} the functor $\mathbf{R}\HH om(-,\OO_X)$ and taking 
into account that $\EE xt^1(\OO_S(m),\OO_X) \\ \cong\OO_S(k-m)$, we 
obtain the exact sequence
	\begin{equation}\label{E, tau}
		0\to E^{\vee}\to\OO_X(n)^{\oplus2}\to\OO_S(m)\to\tau\to0,\ \ \ \ \ \ \ \ \
		\tau:=\EE xt^1(E,\OO_X).
	\end{equation}
The long exact sequence defined by the spectral sequence of
local and global Exts $E^{pq}_2=H^p(\EE xt_{\OO_X}^q(\OO_S(m), \OO_X(-n)^{\oplus2}))\Longrightarrow\Ext^{p+q}(\OO_S(m),\OO_X(-n) ^{\oplus2})$, gives an isomorphism
\begin{equation}\label{isom Ext}
\begin{split}
& \Ext^1(\OO_S(m),\OO_X(-n)^{\oplus2}))\xrightarrow{\cong}H^0(\EE xt^1(\OO_S(m),\OO_X(-n)^{\oplus2}))) \\	
& \cong H^0(\OO_S(k-m-n)^{\oplus2}).
\end{split}
\end{equation}
Under this isomorphism the element $\xi\in\Ext^1(\OO_S(m),\OO_X(-n) ^{\oplus2}))$, defining
the extension \eqref{extension} corresponds to a section $s_{\xi}\in H^0(\OO_S(k-m-n)^{\oplus2})$,
scheme of zeros $Z_{\xi}=(s_{\xi})_0$ of which, due to the ampleness of the sheaf $\OO_S(k-m-n)$ for
$k\ge1,\ m+n<0$, is zero-dimensional for $\xi$ belonging to an open dense subset in
$\Ext^1(\OO_S(m),\OO_X(-n)^{\oplus2}))$.
A standard local calculation shows that $\Sing E=Z_{\xi}$, so $\tau$
is an artinian sheaf and thus $\EE xt^i(\tau,\OO_X)=0$ for $i\le2$.
Therefore, applying the functor $\mathbf{R}\HH om(-,\OO_X)$ to \eqref{E, tau} and taking into account
the isomorphism $\EE xt^1(\OO_S(k-m),\OO_X)\cong\OO_S(m)$, we obtain an exact triple
$0\to\OO_X(-n)^{\oplus 2}\to E^{\vee\vee}\to\OO_S(m)\to 0$ and its isomorphism with the triple
\eqref{extension} via the functor $(-)^{\vee\vee}$. This implies the required
isomorphism $E\cong E^{\vee\vee}$.

(2) Let us prove the stability of an arbitrary sheaf $E\in\RR_{k,m,n}$. Recall that in
\eqref{extension} $k\ge1$, $n=\lceil\frac k2\rceil\ge1$, $m<-n\le1$. This implies,
What
\begin{equation}\label{c1(E)}
	-1\le c_1(E)\le0,
\end{equation}
\begin{equation}\label{h0(E)=0}
	h^0(E)=0.
\end{equation}
Assume that $E$ is unstable. In view of \eqref{c1(E)} the maximal destabilizing
rank 1 subsheaf $F$ of $E$ has the form $F=\II_{Z,\P^3}(q),\ q\ge0$. By virtue of
reflexivity of $E$, the monomorphism $F\to E$ extends to the monomorphism $\OO_{\P^3}(q) \cong
F^{\vee\vee}\overset{s}{\to}E$, that is, a section $0\ne s\in h^0(E(-q)) $. This contradicts
\eqref{h0(E)=0}, since $q\ge0$.
\end{proof}

Consider the modular morphism of the variety $\YY_{k,m,n}$ into the Gieseker-Maruyama moduli scheme
$M_X(v)$:
\begin{equation}\label{def Theta}
	\Theta:\ \YY_{k,m,n}\to M_X(v),\ \ \ [\xi]\mapsto[E_{\xi}],
\end{equation}
where $[E_{\xi}]$ is the isomorphism class of the stable sheaf $E_{\xi}$ obtained from
the point $\xi$ as the lower extension in the diagram \eqref{diag 83}, and $v=\mathrm{ch}(E)$
is determined by formulas \eqref{c vers ch} and
\eqref{k even}-\eqref{k odd}.

Note that by virtue of the first equality \eqref{vanish ext} the sheaf $[E=E_{[\xi]}]\in \Theta(
\YY_{k,m,n})$ has a unique subsheaf $\OO_X(-n)^{\oplus2}$, and hence the extension
\eqref{extension} is uniquely defined, that is, the class $[\xi]$ can be reconstructed from the point $[E]$ uniquely. Therefore,
\begin{equation}\label{Theta bijective}
\Theta:\ \YY_{k,m,n}\to M_{k,m,n}:=\Theta(\YY_{k,m,n})\ - \ \text{bijection},
\end{equation}
\begin{equation}\label{dim=dim}
	\begin{split}
		& \dim M_{k,m,n}=\dim\YY_{k,m,n}=\dim\mathrm{Gr}(\Ext^1(\OO_S(m),
		\OO_X(-n))^{\vee},2)+\\
		& +\dim|\OO_X(k)|=\mathrm{ext}^1(\OO_S(m),\OO_X(-n)^{\oplus2})-4+N.
	\end{split}
\end{equation}
Applying the bifunctor $\mathbf{R}\Hom(-,-)$ to the triple \eqref{extension} and taking into account
the equalities \eqref{vanish ext0} and \eqref{vanish ext}, we obtain for $[E]\in\dim\YY_{k,m,n}$
a commutative diagram
\begin{equation}\label{diag 99}
	\begin{gathered}
		\xymatrix{
			\underset{{}}{\bk}\ar[r]^-{\cong}\ar@{>->}[d] &
			\underset{{}}{\Hom(\OO_S(m),\OO_S(m))}
			\ar@{>->}[d] & \\
			\underset{{}}{\bk^4}\ \ar@{>->}[r]\ar@{>>}[d] &
			\Ext^1(\OO_S(m),\OO_X(-n)^{\oplus2})\
			\ar@{>>}[r]\ar[d] &
			\underset{{}}{\Ext^1(E,\OO_X(-n)^{\oplus2})}\ar@{>->}[d]\\
			\Hom(\OO_X(-n)^{\oplus2},E)/\bk\ \ar@{>->}[r] &
			\Ext^1(\OO_S(m),E)\ \ar@{>>}[r]
			\ar@{>>}[d] & \Ext^1(E,E) \ar@{>>}[d] \\
			& \Ext^1(\OO_S(m),\OO_S(m))\ \ar[r]^-{\cong} &
			\Ext^1(E,\OO_S(m)).}
	\end{gathered}
\end{equation}
From this diagram and the equality $\mathrm{ext}^1(\OO_S(m),\OO_S(m))=N$ (see
\eqref{vanish ext0}) we get
\begin{equation}\label{ext1EE}
	\mathrm{ext}^1(E,E)=\mathrm{ext}^1(\OO_S(m),\OO_X(-n)^{\oplus2})-4+N,
\end{equation}
hence, due to \eqref{dim=dim}, we find $\dim\Theta(\YY_{k,m,n})=\mathrm{ext}^1(E,E)=
\dim_{[E]} M_X(v)$. Taking into account Theorem \ref{Theorem 4.9}, this means that $M_{k,m,n}$ is a
smooth open subset of an irreducible component $\overline{M}_{k,m,n}$ of the
moduli scheme $M_X(v)$. From here and from \eqref{Theta bijective}, according to
\cite[Ch. 2, \S4.4, Thm. 2.16]{Sh}, it follows that
\begin{equation}\label{Theta isom}
	\Theta:\ \YY_{k,m,n}\to M_{k,m,n}\ \text{is an isomorphism},
\end{equation}
that is, $M_{k,m,n}$ is a smooth rational variety, isomorphic in view of Theorem
\ref{Theorem 4.9} to a dense open subset in the smooth variety ${\GG}r(\AA
^{\vee},2)$. At the same time, it directly follows from the diagram \eqref{diag 83} that on
$M_{k,m,n}$ there is a universal family of sheaves. Formulas for dimensions
of varieties $M_{k,m,n}$ as functions of $k,m,n$ are obtained taking into account the Riemann--Roch Theorem
for $X_5$, formulas \eqref{zero cohom}, \eqref{tr X,S}, \eqref{dim=dim} and standard
resolutions for the sheaves $\OO_{X_2}$, $\OO_{X_4}$:
\begin{equation}\label{resolns}
	\begin{split}
		& 0\to\OO_{\P^4}(-2)\to\OO_{\P^4}\to\OO_{X_2}\to0,\\
		& 0\to\OO_{\P^5}(-4)\to\OO_{\P^5}(-2)^{\oplus2}\to\OO_{\P^5}\to\OO_{ X_4}\to0,\\
	\end{split}
\end{equation}
Moreover, in view of Theorem \ref{Theorem 4.9}, the reflexive sheaves in $M_{k,m,n}$ constitute
a dense open set
\begin{equation}\label{R k,m,n}
	R_{k,m,n}:=\Theta(\RR_{k,m,n})\cong\RR_{k,m,n}.
\end{equation}

Thus, the following theorem holds.

\begin{theorem}\label{Theorem 4.1}
	Let $X$ be one of the Fano varieties $X_1$, $X_2$, $X_4$, $X_5$. For an arbitrary
	natural number $k$ and integers $m$ and $n=\lceil\frac k2\rceil$ such that $m+n<0$
	let $M_X(v)$ be the Gieseker-Maruyama moduli scheme of semistable sheaves $E$ on $X$ with 
	Chern character
	$v=\mathrm{ch}(E)$, determined by the formulas \eqref{c vers ch} and
	\eqref{k even}-\eqref{k odd}. Then the following statements are true.\\
	(i) The family
	$$
	M_{k,m,n}=\{[E]\in M_X(v)\ |\ E\ \textit{is a stable extension \eqref{extension}, where}\
	S\in|\OO_X(k)|\}
	$$
is a smooth dense open subset of an irreducible rational component of
scheme $M_X(v)$, and there is an isomorphism \eqref{Theta isom}, in which $\YY_{k,m,n}$ is
the open subset of the grassmannization defined in \eqref{def Y} $\pi:\ {\GG}r(\AA
^{\vee},2)\to\P^N$, where $\AA$ is the locally free sheaf on $\P^N=|\OO_X(k)|$,
given by the first formula \eqref{sheaf A}; in this case $M_{k,m,n}$ is a fine
moduli space, and reflexive sheaves constitute a dense open
subset $R_{k,m,n}$ defined in \eqref{R k,m,n}.\\
(ii) The dimension of the variety $M_{k,m,n}$ is given by the formulas:\\
(ii.1) $\dim M_{k,m,n}=2\binom{k-m-n+3}{3}-2\binom{-m-n+3}{3}+\binom{k +3}{k}-5$
for $X=X_1$,\\
(ii.2) $\dim M_{k,m,n}=2\binom{k-m-n+4}{4}-2\binom{k-m-n+2}{4}-2\binom{- m-n+4}{4}+
2\binom{-m-n+2}{4}+\binom{k+4}{4}-\binom{k+2}{4}-5$ for $X=X_2$,\\
(ii.3) $\dim M_{k,m,n}=2\binom{k-m-n+5}{5}-4\binom{k-m-n+3}{5}+2\binom{k-m -n+1}{5}-
2\binom{-m-n+5}{5}+4\binom{-m-n+3}{5}-2\binom{-m-n+1}{5}+\binom{k +5}{5}-2\binom{k+3
}{5}+\binom{k+1}{5}-5$ for $X=X_4$,\\
(ii.4) $\dim M_{k,m,n}=\frac53(k-m-n)^3+5(k-m-n)^2+\frac{16}3
(k-m-n)-\frac53(-m-n)^3-5(-m-n)^2-\frac{16}3(-m-n)
+\frac56k^3+\frac52k^2+\frac{8}3k-4$ for $X=X_5$.\\
(iii) The dimension of the variety $M_{k,m,n}$ coincides with the virtual dimension of the scheme
$M_X(v)$, that is, $\rmext^2(E,E)=0$ for $[E]\in M_{k,m,n}$, if and only if
$k<i$ and $m>k-n-i$, where $i$ is the index of the Fano variety $X$.
\end{theorem}
\begin{proof}
Statements (i) and (ii) of the theorem were proven above. Statement (iii)
is proven by a direct calculation by analogy with the formulas
\eqref{vanish ext}.
\end{proof}
\begin{remark}\label{Remark 4.A}
Direct checking using formulas \eqref{k even} and \eqref{k odd}
for the Chern classes of the sheaf $[E]\in M_{k,m,n}$ shows that
$\frac{c_3(E)}{H^3}< \frac 4k(\frac{c_2(E)}{H^2})^2$ in case of
even $k$, and, respectively, $\frac{c_3(E)}{H^3}<\frac{4k}{
	(k-1)^2}(\frac{c_2(E)}{H^2})^2$ in case of odd $k$.
\end{remark}

\begin{remark}\label{Remark 4.B}
	On the quadric $X=X_2$ we consider a nontrivial extension of the form
	\begin{equation}\label{extension Spin}
		\begin{split}
			& 0\to\SS(-n)\to E\to\OO_S(m)\to 0,\ \ \ \ \ \ \ \ \ S\in|\OO_X(k)|,\\
			& k\ge1,\ \ \ \ \ n=\lfloor\frac k2\rfloor+1, \ \ \ \ \ m\le-n,
		\end{split}
	\end{equation}
	where $\SS$ is the spinor bundle on $X$. Let $M_X(v)$ be the Gieseker-Maruyama moduli scheme
    of semistable sheaves $E$ with Chern character $v=\mathrm{ch}(E)$,
	determined by the formulas \eqref{c vers ch}, \eqref{v(S)} and \eqref{extension Spin}.
	From \eqref{extension Spin} it follows that $E$ satisfies the formulas \eqref{vanish ext0},
	\eqref{vanish ext}, \eqref{ext1EE} and the diagram \eqref{diag 99}, in which
	$\OO_X(-n)^{\oplus2}$ is everywhere replaced by $\SS(-n)$, and the left vertical triple in
	\eqref{diag 99} is replaced by the triple $\bk\rightarrowtail\bk\twoheadrightarrow0$.
	Respectively, in the notation of the first formula \eqref{sheaf A}, on $\P^N$ the following
	locally free sheaf is defined:
	\begin{equation}\label{sheaf AS}
		\AA_{\SS}=\EE xt_p^1(\OO_{\P^N}\boxtimes\OO_X(m)|_{\mathbb{S}},\OO_{\P^N}\boxtimes
		\SS(-n)).\\
	\end{equation}
In addition, replacing the sheaf $\OO_X(-n)^{\oplus2}$ in the formulas \eqref{grasmn}-\eqref{def RR}
by $\SS(-n)$, we obtain the projectivization $\pi:\ \P(\AA_{\SS}^{\vee})\to\P^N$ and its dense
open subsets $\YY_{\SS,k,m,n}$ and $\RR_{\SS,k,m,n}$ of stable and reflexive
sheaves respectively. For them, by analogy with \eqref{open subsets}, we have embeddings of dense
open sets
\begin{equation}\label{subsets S}
	\RR_{\SS,k,m,n}\subset\YY_{\SS,k,m,n}\subset{\GG}r(\AA_{\SS}^{\vee},2)
\end{equation}
and, as in \eqref{Theta isom}, an isomorphism
\begin{equation}\label{Theta isom S}
	\Theta:\ \YY_{\SS,k,m,n}\xrightarrow{\cong} M_{k,m,n}:=\Theta(\YY_{\SS,k,m,n}),\ \ \
	[\xi]\mapsto[E_{\xi}].
\end{equation}
Repeating for the indicated varieties the corresponding stages of the proof of Theorem
\ref{Theorem 4.1}, we obtain the following theorem.
\end{remark}

\begin{theorem}\label{Theorem 4.1S}
	Let $X=X_2$ be the three-dimensional quadric. Then, in the notation of Remark \ref{Remark 4.B},
	the following statements are true.\\
	(i) The family
	$$
	M_{k,m,n}=\{[E]\in M_X(v)\ |\ E\ \textit{is a stable extension \eqref{extension Spin},
		where}\S\in|\OO_X(k)|\}
	$$
	is a smooth dense open subset of an irreducible rational component of
	scheme $M_X(v)$, and there is an isomorphism $\Theta$ in \eqref{Theta isom S}, in which
	$\YY_{\SS,k,m,n}$ is an open subset (defined in Remark \ref{Remark 4.B}) of the
	projectivization $\pi:\ \P(\AA_{\SS}^{\vee})\to\P^N$, where $\AA_{\SS}$ is the locally
	free sheaf on $\P^N=|\OO_X(k)|$, given by the formula \eqref{sheaf AS}; wherein
	$M_{k,m,n}$ is a fine moduli space, and reflexive sheaves constitute in
	$M_{k,m,n}$ a dense open set $R_{\SS,k,m,n}=\Theta(\RR_{\SS,k,m,n})$, where
	$\RR_{\SS,k,m,n}$ is defined in Remark \ref{Remark 4.B}.\\
	(ii) The dimension of the variety $M_{k,m,n}$ is given by the formula $\dim M_{k,m,n}=\\
	4\binom{k-m-n+3}{3}-4\binom{-m-n+3}{3}+\binom{k+4}{4}-\binom{k+2}{4}-2$.
\end{theorem}

\subsection{Second series.}\label{subsec 4.2}
In this subsection we treat in more detail the extensions \eqref{extension I} on the quadric
$X=X_2$.

Let us first find the dimension of the space $\Ext^1(\II(m),\II(m))$, which we will need
below. To do this, apply the functor $\mathbf{R}\Hom(-,\II(1))$ to the triple \eqref{spin tr on X}:
\begin{equation*}
	\begin{split}
		& 0\to\Hom(\II(1),\II(1))\to\Hom(\OO_X^{\oplus2},\II(1))\to\Hom(\SS(-1), \II(1))\to \\
		& \to\Ext^1(\II(1),\II(1))\to\Ext^1(\OO_X^{\oplus2},\II(1)).
	\end{split}
\end{equation*}
From here and from \eqref{vanish ext I} and \eqref{vanish ext I2} we find
\begin{equation}\label{ext(I,I)}
	\rmext^1(\II(m),\II(m))=\rmext^1(\II(1),\II(1))=4.
\end{equation}

Let us now consider the lower triple \eqref{univ ext}. By universality, this triple
restricted to $\{x\}\times X$, $x\in\mathbb{G}$, in view of the base change \eqref{bch}
can be included with the extension \eqref{extension I}, where $\widetilde{\xi}:\
\Ext^1(\II(m),\OO_X(-1))^{\vee} \to \bk^2$ is the second homomorphism in \eqref{hom xi} and
$E=E_{\widetilde{\xi}}$, into a commutative diagram similar to \eqref{diag 83}:
\begin{equation}\label{diag 83 I}
	\begin{gathered}
		\xymatrix{
			& \Ext^1(\II(m),\OO_X(-1))^{\vee}\otimes\OO_X(-n)\ \ar@{>->}[r]\ar[d]_ -{\widetilde{
					\xi}} & \widetilde{\WW}|_{\{x\}\times X}\ar@{>>}[r]\ar[d] & \II(m) \ar@{=} [d]\\
			& \OO_X(-1)^{\oplus 2}\ \ar@{>->}[r] & E_{\widetilde{\xi}}\ar@{>>}[r] & \II(m ).}
	\end{gathered}
\end{equation}
Moreover, as in \eqref{diag 83}, if $E_{\widetilde{\xi}}$ is stable, then
$\widetilde{\xi}$ is an epimorphism. Therefore, by analogy with
\eqref{def Y}-\eqref{def RR}, consider the grassmannization
$$
\widetilde{\pi}:\ {\GG}r(\widetilde{\AA}^{\vee},2)\to\mathbb{G}
$$
of two-dimensional quotient spaces of the bundle $\widetilde{\AA}^{\vee}$ on $\mathbb{G}$,
defined in \eqref{sheaf A}, and its open subsets
\begin{equation}\label{def tilde Y}
\widetilde{\YY}_m:=\{[\widetilde{\xi}]\in{\GG}r(\widetilde{\AA}^{\vee},2)\ |\
E_{\widetilde{\xi} }\ \textit{is stable}\},
\end{equation}
\begin{equation}\label{def tilde RR}
\widetilde{\RR}_m:=\{[\widetilde{\xi}]\in{\GG}r(\widetilde{\AA}^{\vee},2)\ |\
E_{\widetilde{\xi}}\ \textit{is reflexive}\}.
\end{equation}
Repeating the proof of Theorem 4.9 for $E$ from the triple \eqref{extension I},
we obtain embeddings of dense open subsets
\begin{equation}\label{tilde subsets}
	\widetilde{\RR}_m\subset\widetilde{\YY}_m\subset{\GG}r(\widetilde{\AA}^{\vee},2).
\end{equation}
As in \eqref{def Theta} we have the modular morphism $\Theta:\ \widetilde{\YY}_m\to M_X(v),
\ [\xi]\mapsto[E_{\widetilde{\xi}}]$, where $[E_{\widetilde{\xi}}]$ is the isomorphism class
of the bundle $E_{\widetilde{\xi}}$ obtained from the point $\widetilde{\xi}$ as the lower
extension in the diagram \eqref{diag 83 I}. Assuming $\widetilde{M}_m=\Theta(
\widetilde{\YY}_m)$, as in \eqref{Theta bijective} and \eqref{dim=dim}, we have a bijection
\begin{equation}\label{Theta bijective 2}
	\Theta:\ \widetilde{\YY}_m\xrightarrow{\simeq}\widetilde{M}_m
\end{equation}
and an equality
\begin{equation}\label{dim=dim2}
	\begin{split}
		& \dim\widetilde{M}_m=\dim\widetilde{\YY}_m=\dim\mathrm{Gr}(\Ext^1(\II(m),
		\OO_X(-1))^{\vee},2)+\dim\mathbb{G}=\\
		& 2\mathrm{ext}^1(\II(m),\OO_X(-1))-4+4=\mathrm{ext}^1(\II(m),\OO_X(-1)^{ \oplus2}).
	\end{split}
\end{equation}
Applying the bifunctor $\mathbf{R}\Hom(-,-)$ to the triple \eqref{extension I} and taking into account
the equalities \eqref{vanish ext0}-\eqref{vanish ext I}, we obtain a commutative diagram
\begin{equation}\label{diag 113}
	\begin{gathered}
		\xymatrix{
			\underset{{}}{\bk}\ar[r]^-{\cong}\ar@{>->}[d] & \underset{{}}{\Hom(\II(m),\ II(m))}
			\ar@{>->}[d] & \\
			\underset{{}}{\bk^4}\ \ar@{>->}[r]\ar@{>>}[d] & \Ext^1(\II(m),\OO_X(- 1)^{\oplus2})\
			\ar@{>>}[r]\ar[d] & \underset{{}}{\Ext^1(E,\OO_X(-1)^{\oplus2})}\ar@{>->}[d]\\
			\Hom(\OO_X(-1)^{\oplus2},E)/\bk\ \ar@{>->}[r] &
			\Ext^1(\II(m),E)\ \ar@{>>}[r]\ar@{>>}[d] & \Ext^1(E,E) \ar@{>>}[d]\\
			& \Ext^1(\II(m),\II(m))\ \ar[r]^-{\cong} & \Ext^1(E,\II(m)).}
	\end{gathered}
\end{equation}
From this diagram and the equalities \eqref{ext(I,I)} and \eqref{dim=dim2} we obtain
\begin{equation}\label{ext1EE I}
	\rmext^1(E,E)=\rmext^1(\II(m),\OO_X(-n)^{\oplus2})=\dim\widetilde{M}_m.
\end{equation}
From here, according to the deformation theory, in view of the stability of $E$, it follows that $\widetilde{M}_{k,m,n}$
is a smooth open subset of an irreducible component of the module scheme $M_X(v)$. Therefore
from \eqref{Theta bijective 2}, as in \eqref{Theta isom}, it follows that
\begin{equation}\label{Theta isom I}
	\Theta:\ \widetilde{\YY}_m\to\widetilde{M}_m \text{is an isomorphism},
\end{equation}
that is, $\widetilde{M}_m$ is a smooth rational variety isomorphic,
due to \eqref{tilde subsets}, to a dense open subset of a smooth variety
${\GG}r(\widetilde{\AA}^{\vee},2)$. Moreover, $\widetilde{M}_m$ is also
a fine moduli space. Formulas for the dimensions of the varieties $\widetilde{M}_m$ as
functions of $m$ are obtained taking into account \eqref{zero cohom}-\eqref{tr I,S} and
\eqref{dim=dim2} by analogy with formulas for dimensions of $M_m$.\\
Moreover, due to \eqref{tilde subsets}, reflexive sheaves in $\widetilde{M}_m$
constitute a dense open subset
\begin{equation}\label{tilde R k,m,n}
	\widetilde{R}_m:=\Theta(\widetilde{\RR}_m)\cong\widetilde{\RR}_m.
\end{equation}

So, the following theorem holds.
\begin{theorem}\label{Theorem 4.2}
	Let $X=X_2$ be the quadric, and let $M_X(v)$ be the Gieseker-Maruyama moduli scheme
	of semistable sheaves $E$ on $X$ with Chern character $v=\mathrm{ch}(E)$ defined by
	the formulas \eqref{c vers ch 3} and \eqref{extension I}. Then the following statements are true.\\
	(i) The family
	\begin{equation}
		\begin{split}
			& \widetilde{M}_m=\{[E]\in M_X(v)\ |\ E\ \textit{is a stable extension
				\eqref{extension I}}\},
		\end{split}
	\end{equation}
is a smooth dense open subset of an irreducible rational component of the
scheme $M_X(v)$, and there is an isomorphism \eqref{Theta isom I}, in which $\widetilde{\YY
}_m$ is the defined in \eqref{def tilde Y} open subset of the grassmannization 
$\widetilde{\pi}:\ {\GG}r(\widetilde{\AA}^{\vee},2)\to\mathbb{G}$, where $\widetilde{\AA}$
is the locally free sheaf on $\mathbb{G}$, given by the second formula \eqref{sheaf A};
moreover, $\widetilde{M}_m$ is a fine moduli space, and reflexive sheaves
constitute in $\widetilde{M}_m$ the dense open set $\widetilde{R}_m$,
defined in \eqref{tilde R k,m,n}.\\
(ii) The dimension of the variety $\widetilde{M}_m$ is given by the formula\\
$\dim\widetilde{M}_m=2\binom{-m+4}{4}-2\binom{-m+2}{4}-2\binom{-m+3}{4}+2 \binom{-m+1}{4}
-2m+2=2m^2-6m+4$.\\
\end{theorem}

\begin{remark}\label{Remark 4.B I}
	On the quadric $X=X_2$ we consider a nontrivial extension of the form
	\begin{equation}\label{extension Spin I}
		\begin{split}
			& 0\to\SS(-1)\to E\to\II(m)\to 0,\ \ \ \ \ \ \ \ \ \ m<-1,
		\end{split}
	\end{equation}
	where $\SS$ is the spinor bundle on $X$, and the sheaf $\II$ is the same as in
	\eqref{extension I}. As in the case of sheaves $E$ from \eqref{extension I}, we assume that
	the sheaf $E$ is stable. Let $M_X(v)$ be the Gieseker-Maruyama moduli scheme of semistable
	sheaves $E$ with Chern character $v=\mathrm{ch}(E)$ defined by the formulas
	\eqref{c vers ch}, \eqref{v(S)} and \eqref{extension Spin I}. From \eqref{extension Spin I}
	it follows that $E$ satisfies the formulas \eqref{vanish ext0},
	\eqref{vanish ext I}, \eqref{vanish ext3}, \eqref{ext1EE I} and the diagram
	\eqref{diag 113}, in which the sheaf $\OO_X(-n)^{\oplus2}$ is replaced everywhere
	by $\SS(-1)$, and the left vertical triple in \eqref{diag 113} is replaced by a triple
	$\bk\rightarrowtail\bk\twoheadrightarrow0$. Respectively, in the notation of the second
	formula \eqref{sheaf A} on $\mathbb{G}$ the following locally free sheaf is defined:
\begin{equation}\label{sheaf AS I}
\widetilde{\AA}_{\SS}=\EE xt_{\widetilde{p}}^1(\mathbb{I}\otimes\OO_{\mathbb{G}}\boxtimes
\OO_X(m),\OO_{\mathbb{G}}\boxtimes\SS(-1)).
\end{equation}	
Next, as in Remark \ref{Remark 4.B}, we consider the projectivization $\widetilde{\pi}:\
\P(\widetilde{ \AA}_{\SS}^{\vee})\to\mathbb{G}$ and its open subsets $\widetilde{
	\YY}_{\SS,m}$ and $\widetilde{\RR}_{\SS,m}$ of stable and reflective sheaves
respectively. For them, by analogy with \eqref{subsets S} and \eqref{Theta isom S} we have
embeddings of dense open sets
\begin{equation}\label{subsets SI}
	\widetilde{\RR}_{\SS,m}\subset\widetilde{\YY}_{\SS,m}\subset{\GG}r(\widetilde{\AA}_{\SS}
	^{\vee},2)
\end{equation}
and an isomorphism
\begin{equation}\label{Theta isom SI}
	\Theta:\ \widetilde{\YY}_{\SS,m}\xrightarrow{\cong} M_m:=\Theta(\widetilde{\YY}_{\SS,m}),
	\ \ \ [\widetilde{\xi}]\mapsto[E_{\widetilde{\xi}}].
\end{equation}
Repeating the proof of Theorem \ref{Theorem 4.2} for the indicated varieties, we obtain
the following theorem.
\end{remark}

\begin{theorem}\label{Theorem 4.2S}
Let $X=X_2$ be the quadric. Then, in the notation of Remark \ref{Remark 4.B I}, the following statements are true.\\
(i) The family
$$
M_m=\{[E]\in M_X(v)\ |\ E\ \textit{is a stable extension \eqref{extension Spin I}} \}
$$
is a smooth dense open subset of an irreducible rational component of the
scheme $M_X(v)$, and there is an isomorphism $\Theta$ in \eqref{Theta isom SI}, in which
$\widetilde{\YY}_{\SS,m}$ is an open
subset of the projectivization $\widetilde{\pi}:\ \P(\widetilde{\AA}_{\SS}^{\vee})\to
\mathbb{G}$, defined in Remark \ref{Remark 4.B I}, where $\widetilde{\AA}_{\SS}$ is the locally free sheaf of rank
$\rk\widetilde{\AA}_{\SS}=2m^2-8m+7$
on $\mathbb{G}$ defined by the formula \eqref{sheaf AS I}; in this case $M_m$ is
a fine moduli space, and reflexive sheaves constitute a dense open
set $\widetilde{R}_{\SS,m}=\Theta(\widetilde{\RR}_{\SS,m})$ in $M_m$, where $\widetilde{\RR}
_{\SS,m}$ is defined in Remark \ref{Remark 4.B I}.\\
(ii) The dimension of the variety $M_m$ is given by the formula
$\dim M_m=2m^2-8m+10$.
\end{theorem}

\subsection{Third series.}\label{subsec 4.3}

Let $k\ge1$, $n=\lfloor\frac k2\rfloor+1$ and $m+n<0$. Consider on $X$ a sheaf $F$ of rank 2 and a variety $W$,
defined as follows.\\
(I) In the case $X=X_1$ the sheaf $F$ is determined from the exact triple
\begin{equation}\label{def of F 1}
	0\to\OO_{\P^3}(-1)\to\OO_{\P^3}^{\oplus 3}\to F\to 0.
\end{equation}
As is known (see, for example, \cite[Example 4.2.1]{H}), or
\cite[Remark 2]{AJT}), the moduli space $W$ of sheaves $F$ is
isomorphic to $\p3$, and the isomorphism $W\xrightarrow{\simeq}\p3$
is given by the formula $[F]\mapsto\Sing(F)$.\\
(II) In the case when $X=X_4$ is a complete intersection of a general pencil of hyperquadrics in $\p5$,
let $\p1\subset|\OO_{\p5}(2)|$ be the base of this pencil of quadrics, and let $\Gamma$ be
the hyperelliptic curve of genus 2, defined as the double covering $\rho:\ \Gamma\to
\P^1$, branched at points corresponding to degenerate quadrics of the pencil. Let
$\Gamma^*=\rho^{-1}({\p1}^*)$, where ${\p1}^*\subset\p1$ is an open subset of
nondegenerate quadrics of the pencil, and let $\Delta=\rho^{-1}(\p1\smallsetminus{\p1}^*)$. Any
point $y\in\Gamma^*$ corresponds to one of the two series of generating planes on a non-degenerate
4-dimensional quadric $Q(y):=\rho(y)$, and this series corresponds to a spinor bundle $\SS(y)$
of rank 2 on $Q(y)$ with $\det\SS(y)=\OO_{Q(y)}(1)$. In this case we set $F_y=\SS(y)|_X$.
It is known that $\SS(y)$ is an arithmetically Cohen--Macaulay sheaf (briefly: an ACM sheaf), that is, it satisfies the equalities:
\begin{equation}\label{ACM}
	h^i(\SS(y)(\ell))=0\ \ \ \ \ell\in\Z,\ \ \ \ i=1,2.
\end{equation}
From this and from the cohomology of the triple $0\to\SS(y)(-2)\to\SS(y) \to F_y\to0$
it follows that $F_y$ is also an ACM sheaf. Let now $y\in\Delta$,
that is, the degenerate quadric $Q(y)$ is a cone with its vertex at the point
say $z(y)$, so the projection $\mu:Q(y)\smallsetminus \{z(y)\}\to Q_y$ is defined,
where $Q_y$ is a smooth three-dimensional quadric. On $Q_y$
the spinor bundle $\SS_{Q_y}$ with $\det\SS_{Q_y}=\OO_{Q_y}(1)$ is defined, and 
we set $F_y=\mu^*\SS_{Q_y}|_X$. The sheaf $F_y$ is also
an ACM sheaf. In this case we set $W=\{[F_y]\ |\ y\in\Gamma\}\simeq
\Gamma$.\\
(III) In the case $X=X_5$, the sheaf $F$ is defined as the restriction to $X$ of the tautological bundle on the grassmannian $\Gr(2,5)$, twisted by
$\OO_X(1)$. The isomorphism class
$[F]$ of $F$ is uniquely defined, and we assume that $W$ is a point $\{[F]\}$.
Note that $F$ is also an ACM sheaf.

Note that in all cases (I)-(III) the universal family $\mathbb{F}$ of sheaves $F$ on $W\times X$ is defined. 

Consider the sheaf $E$ on $X$ obtained as a nontrivial extension
\begin{equation}\label{extension F}
	0\to F(-n)\to E\to\OO_S(m)\to0,\ \ \ \ \ \ S\in|\OO_X(k)|,\ \ \ \ \ k\ge1.
\end{equation}
Further, in case when $E$ is stable, by analogy with \eqref{vanish ext},
using the definition and the above properties of the sheaf $F$, the exact triple
$0\to\OO_X(-k)\to\OO_X\to \OO_S\to0$ and the condition $m+n<0$, we have the equalities
\begin{equation}\label{vanish ext 2}
	\begin{split}
		& \hom(E,E)=\hom(F,F)=\hom(\OO_S(m),\OO_S(m))=\hom(E,\OO_S(
		m))=\\
		& \hom(F(-n),E)=1,\ \ \ \ \hom(F(-n),\OO_S(m))=\rmext^2(\OO
		_S(m), F(-n))=0.
	\end{split}
\end{equation}

Let, as above, $\P^N=|\OO_X(k)|$, let $\mathbb{S}\subset\P^N \times X$ be a universal family of surfaces of degree $k$ in $X$, and let $W\times X\xleftarrow{q}W\times\P^N
\times X\xrightarrow{p}W\times\P^N$ be the projections. On $W\times\P^N$, by analogy with
\eqref{sheaf A}-\eqref{Ext0=0}, we have a locally free sheaf
\begin{equation}\label{sheaf A 2}
	\AA=\EE xt_p^1(\OO_{W\times\P^N}\boxtimes\OO_X(m)|_{W\times\mathbb{S}},\OO_{\P^N}
	\boxtimes \OO_X(-n)\otimes q^*\mathbb{F}),
\end{equation}
the construction of which commutes with the base change:
\begin{equation}\label{bch 2}
	\tau_{F,S}:\ \AA\otimes\bk_{\{F,S\}}\xrightarrow{\cong} \Ext^1(\OO_S(m),F(-n)),
\end{equation}
and we have an equality
\begin{equation}\label{Ext0=0 2}
	\EE xt_p^0(\OO_{W\times\P^N}\boxtimes\OO_X(m)|_{W\times\mathbb{S}},\OO_{W\times\P^N}
	\boxtimes\OO_X(-n)\otimes q^*\mathbb{F})=0.
\end{equation}
Consider the varieties $Y=\P(\AA^{\vee})$ and $X_Y=Y\times X$ with projections $Y\xrightarrow{
	\rho}\P^N$ and $Y\xleftarrow{p_Y}X_Y\xrightarrow{\rho_Y}W\times\P^N\times X$. From
\eqref{bch 2}, \eqref{Ext0=0 2} and \cite[Cor. 4.5]{Lan83} it follows that on $X_Y$
there is the universal family of extensions
\begin{equation*}
	0\to p_Y^*\OO_P(1)\otimes(\OO_{W\times\P^N}\boxtimes\OO_X(-n)\otimes q^*\mathbb{F})\to\VV
	\to\rho_Y^*(\OO_{W\times\P^N}\boxtimes\OO_X(m)|_{W\times\mathbb{S}})\to0.
\end{equation*}
Applying the functor $\rho_{Y*}$ to this triple, we obtain on $W\times\P^N\times X$
the universal extension, where $\UU=\rho_{Y*}\VV$:
\begin{equation}\label{univ ext 2}
	0\to p^*\AA^{\vee}\otimes(\OO_{W\times\P^N}\boxtimes\OO_X(-n))
	\to\UU\to\OO_{\P^N}\boxtimes\OO_X(m)|_{\mathbb{S}}\to0.
\end{equation}

Note that the extension \eqref{extension F} is specified by an element $\xi\in\Ext^1(\OO_S(m),
F(-n))$, considered as a homomorphism
\begin{equation}\label{hom xi 2}
	\xi:\ \Ext^1(\OO_S(m),\OO_X(-n))^{\vee}\to\bk.
\end{equation}
By the universality, the exact triple \eqref{univ ext 2}, restricted 
to $\{S\}\times X$, in view of the base change \eqref{bch 2} can be 
included with the extension \eqref{extension F} in a commutative 
diagram
\begin{equation}\label{diag 98}
	\begin{gathered}
		\xymatrix{
			& \Ext^1(\OO_S(m),F(-n))^{\vee}\otimes F(-n)\ \ar@{>->}[r] \ar[d]_-{ \xi} &
			\UU|_{\{S\}\times X}\ar@{>>}[r]\ar[d] & \OO_S(m) \ar@{=}[d] \\
			& F(-n)\ \ar@{>->}[r] & E_{[\xi]}\ar@{>>}[r] & \OO_S(m). }
	\end{gathered}
\end{equation}
The homomorphism $\xi$ in \eqref{hom xi 2} is non-zero, since the extension \eqref{extension F} is non-trivial. Thus, $\xi$ is an epimorphism. The class $[\xi]$ of the epimorphism $\xi$
modulo automorphisms of the space $\bk$ is a point of projectivization $\mathbb{P}(
\Ext^1(\OO_S(m),$\\ $\OO_X(-n))^{\vee})$. We have the subset $\YY_{k,m,n}$ in $\mathbb{P}
(\AA ^{\vee})$ defined as $\YY_{k,m,n}:=\{[\xi]\in \mathbb{P}(\AA ^{\vee})\ |\ \\ \textit{the extension}\ E_{[\xi]}\ \textit{in}\ \eqref{diag 98}\ \textit{is stable}\}$, with the modular morphism
$\Theta:\ \YY_{k,m,n}\to M_X(v),\ \ \ [\xi]\mapsto[E_{[\xi]}],$
defined verbatim as in \eqref{def Theta}. At the same time, as in the case of extension
\eqref{extension}, due to \eqref{vanish ext 2}, the extension \eqref{extension F} defined
by the sheaf $E$ is unique, that is, the class $[\xi]$ can be reconstructed from the point $[E]$. Thus, $\Theta$ is a bijection of the variety $\YY_{k, m,n}$ onto its
image $M_{k,m,n}$ in $M_X(v)$, and by analogy with \eqref{dim=dim} we have the equalities
\begin{equation}\label{dim=dim 2}
	\begin{split}
		& \dim M_{k,m,n}=\dim\YY_{k,m,n}=\dim\mathbb{P}(\Ext^1(\OO_S
		(m),F(-n))^{\vee})+\dim W+\\
		& +\dim|\OO_X(k)|=\mathrm{ext}^1(\OO_S(m),F(-n))+\dim W+N-1.
	\end{split}
\end{equation}
Applying the bifunctor $\mathbf{R}\Hom(-,-)$ to the triple \eqref{extension F}
and taking into account the equalities \eqref{vanish ext 2},
we get a commutative diagram:
\begin{equation*}\label{diag 93}
	\begin{gathered}
		\xymatrix{
			\underset{{}}{\Ext^1(\OO_S(m),F(-n)}/\bk\ \ar@{>->}[r]\ar@{>->}[d] &
			\underset{{}}{\Ext^1(E,F(-n))}/\bk\
			\ar@{>->}[r]\ar@{>->}[d] &
			\underset{{}}{\Ext^1(F(-n),F(-n))}\ar[d]^-{=}\\
			\Ext^1(\OO_S(m),E)\ \ar@{>->}[r] \ar@{>>}[d] & \Ext^1(E,E)\
			\ar@{>>}[r] \ar@{>>}[d] & \Ext^1(F(-n),E) \\
			\Ext^1(\OO_S(m),\OO_S(m)) \ar[r]^-{=} & \Ext^1(E,\OO_S(m)). & }
	\end{gathered}
\end{equation*}
From this diagram and the equalities $\mathrm{ext}^1(\OO_S(m),\OO_S(m))=N$ and $\rmext^1(F(-n), \\
F(-n))=\dim W$ it follows that $\mathrm{ext}^1(E,E)=\mathrm{ext}^1(\OO_S(m),F(-n))-1 +\dim W
+N$, from which, due to \eqref{dim=dim 2}, we find $\dim\Theta(\YY_{k,m,n})=\mathrm{ext}^1(E,E)
=\dim_{[E]} M_X(v).$ This means that $M_{k,m,n }$ is a smooth component of the moduli scheme
$M_X(v)$. From here, as in \eqref{Theta isom} and \eqref{Theta isom I}, we obtain an isomorphism
$\Theta:\ \YY_{k,m,n}\to M_{k,m,n}$, that is, $M_{k,m,n}$ is a smooth variety.
Moreover, from the diagram \eqref{diag 83} we obtain that $M_{k,m,n}$ is a fine
moduli space. Exact formulas for the dimensions of the varieties $M_{k,m,n}$
are obtained using \eqref{dim=dim 2} by analogy with the formulas \eqref{vanish ext}.
Thus, the following theorem holds.
\begin{theorem}\label{Theorem 4.3}
	Let $X$ be one of varieties $X_1$, $X_4$, $X_5$, let $k\ge1$, $n=\lfloor
	\frac k2\rfloor+1$, $m<-n$, and $M_X(v)$ be the Gieseker-Maruyama moduli scheme of semistable
	sheaves $E$ on $X$ with Chern character $v=\mathrm{ch}(E)$ determined from the triple
	\eqref{extension F}. Then\\
	(i) the family
\begin{equation*}
\begin{split}
& M_{k,m,n}=\{[E]\in M_X(v)\ |\ E\ \textit{is specified by the extension \eqref{extension F} in which}\ S\in|\OO_X(k)|,\\
& \textit{and the sheaf}\ F\ \textit{is determined by the above conditions (I)-(III)}\}
\end{split}
\end{equation*}
is a smooth irreducible component of the scheme $M_X(v)$ and is described as
the projectivization $\pi:\ \mathbb{P}(\AA^{\vee})\to W\times\P^N$ of the bundle $\AA^{\vee}$, where
the variety $W$ is defined by conditions (I)-(III) above, and $\AA$ 
is the locally free sheaf on $W\times\P^N$ defined by the formula 
\eqref{sheaf A 2}, where $\P^N=| \OO_X(k)|$;
this component is a rational variety for $X=X_1,\ X_2,\ X_5$, and 
non-rational for $X=X_4$; Moreover, $M_{k,m,n}$ is a fine moduli 
space, and all sheaves from $M_{k,m,n}$ are stable;\\
(ii) the dimension of the variety $M_{k,m,n}$ is given by the 
formulas:\\
(ii.1) $\dim M_{k,m,n}=\binom{k+3}{3}+3\binom{-m-n+k+3}{3}-\binom{-m -n+k+2}{3}-
3\binom{-m-n+3}{3}+\binom{-m-n+2}{3}+1$ for $X=X_1$;\\
(ii.2) $\dim M_{k,m,n}=4[\binom{k-m-n+4}{4}-\binom{k-m-n+2}{4}-\binom{-m -n+4}{4}
+\binom{-m-n+2}{4}]+\frac 23k^3+2k^2+\frac 73k$
for $X=X_4$;\\
(ii.3) $\dim M_{k,m,n}=\frac 53(k-m-n+1)^3+\frac{15}2(k-m-n+1)^2+\frac{65}6(k-m-n +1)-
\frac 53(-m-n+1)^3-\frac{15}2(-m-n+1)^2-\frac{65}6(-m-n+1)+\frac 56k^3 +\frac 52k^2+
\frac{8}{3}k-1$ for $X=X_5$.
\end{theorem}

\subsection{Reflexive stable sheaves of general type with $c_1=0$ on $X_4$ and $X_5$.} \label{subsec 4.4}

To conclude this section, we prove a theorem concerning reflexive sheaves in
the components of the Gieseker-Maruyama scheme constructed in Theorems
\ref{Theorem 4.1}-\ref{Theorem 4.3}.

For $X=X_4$ or $X_5$ we denote by $B(X)$ the base of the family of lines on $X$. As is known,
$B(X_4)$ is a smooth abelian surface, and $B(X_5)\simeq\P^2$.
Let us give the following definition, which will be key for use in section 6.
\begin{definition}\label{def 4.10}
	A reflexive sheaf $E$ of rank 2 with first Chern class $c_1(E)=0$ on $X=X_4$ or $X=X_5$
	is called a sheaf of general type if for any line $l\in B(X)$ not passing
	through points of  $\Sing E$ we have either $E|_l\cong\OO_{\P^1}^{\oplus2}$, and such lines
	constitute a dense open set in $B(X)$, or $E|_l\cong\OO_{\P^1}(m)
	\oplus\OO_{ \P^1}(-m)$, where $m>0$, and the set $B_2(X):=\{l\in B(X)\ |\ E|_l
	\cong\OO_{\P^1}(m)\oplus\OO_{\P^1}(-m),\ m\ge2\}$ has dimension $\le0$.
\end{definition}
The following theorem gives examples of infinite series of irreducible components of 
Gieseker-Maruyama moduli schemes for varieties $X_4$ and $X_5$ such that general points of these
components are reflexive sheaves of general type.
\begin{theorem}\label{Theorem 4.4}
	Consider the following infinite series of components of the Gieseker--Maruyama moduli schemes of
	rank two semistable sheaves on varieties $X_4$ and $X_5$:\\
	(I) series of components $M_{2,m,1}$ from Theorem \ref{Theorem 4.1} for the case $k=2$, $m\le-2$, $n=1$;\\
	(II) series of components $M_{1,m,1}$ from Theorem
	\ref{Theorem 4.3} for the case $k=n=1$, $m\le-1$.\\
	Then the general sheaf in each of these components is a reflexive stable sheaf
	of general type.
\end{theorem}
\begin{proof}
	(I) For $[E]\in M_{2,m,1}$ we have the triple
	\eqref{extension} for $k=2,\ m\le-2,\ n=1$ and $S\in\P^N=|\OO_X(2)|$:
	\begin{equation}\label{extension k=2,n=1}
		0\to\OO_X(-1)^{\oplus 2}\to E\to\OO_S(m)\to 0,\ \ \ \ \ \ \ \ \ \ \ S\in|\OO_X(2)| .
	\end{equation}
	Consider the variety $\Pi:=\{(\P^1,S)\in B(X)\times\P^N\ |\ \P^1\subset S\}$ with
	projections $B(X)\xleftarrow{pr_1}\Pi\xrightarrow{pr_2}\P^N$. An easy calculation
	shows that the dimension of a fiber of the projection $pr_1$ is equal to $N-3$. On the other hand,
	it is known \cite{Isk} that $\dim B(X)=2$ for $X=X_4$. Hence $\dim\Pi=N-1<N=\dim\P^N$.
	Therefore, let us take in the triple \eqref{extension k=2,n=1}
	a surface $S$ such that $S\in\P^N\smallsetminus(pr_2(\Pi))$. By the  definition of $\Pi$ it
	means that there are no projective lines on the surface $S$. Therefore, taking an arbitrary projective 
	line $\P^1\subset X$ disjoint from $\Sing E$, we have a scheme $Z_2:=S\cap\P^1$ of length $\ell(Z_2)
	=2$. Restricting the triple \eqref{extension k=2,n=1} to $\P^1$ we obtain an exact triple
	\begin{equation}\label{ext restr}
		0\to\OO_{\P^1}(-1)^{\oplus 2}\to E|_{\P^1}\to\OO_{Z_2}\to0.
	\end{equation}
This restriction can be considered as a homomorphism of spaces of extensions
$r:\ \Ext^1(\OO_S(m),\OO_X(-1)^{\oplus2}))\to\Ext^1(\OO_{Z_2},\OO_{\P^1}( -1)^{\oplus2}))$,
for which $r(\xi)=\xi_{\P^1}$, where the elements $\xi$ and $\xi_{\P^1}$ define extensions
\eqref{extension k=2,n=1} and \eqref{ext restr} respectively. The homomorphism $r$ decomposes
into a composition
$r:\ \Ext^1(\OO_S(m),\OO_X(-n)^{\oplus2}))\underset{\cong}{\xrightarrow{\varphi}}H^0(
\OO_S(1-m)^{\oplus2})\xrightarrow{\otimes\OO_{\P^1}}H^0(\OO_{Z_2}^{\oplus2})\underset{
	\cong}{\xrightarrow{\varphi_{\P^1}^{-1}}}\Ext^1(\OO_{Z_2},\OO_{\P^1}(-1)^{\oplus2} ))$,
where $\varphi$ is an isomorphism \eqref{isom Ext} for $k=2,\ n=1$, and the isomorphism $\varphi_{
	\P^1}$ is constructed by analogy with $\varphi$. Let us check that $r$ is an epimorphism. For this
it is enough to verify that the homomorphism $\otimes\OO_{\P^1}$ is epimorphic. Consider the embedding
$X=X_q\hookrightarrow\P^{q+1}$, $q\in\{4,5\}$, and let $L\cong\P^{q-1}$ be a
subspace of
codimension 2 in $\P^{q+1}$, containing the subscheme $Z_2$ and intersecting $S$ along some
0-dimensional scheme $Z$. (A simple calculation of parameters shows that for a general $L$ passing through $Z_2$
the condition $\dim Z=0$ is satisfied.) Then we have the $\OO_S$-Koszul resolution of the sheaf
$\OO_Z^{\oplus2}$:\ \ $0\to\OO_S(-m-1)^{\oplus2}\to\OO_S(-m)^{\oplus4}\to\OO_S(-m+1 )^{
	\oplus2}\xrightarrow{\otimes\OO_Z}\OO_Z^{\oplus2}\to0$. Using the triple \eqref{tr X,S} and
the condition $m\le-2$, we find $H^2(\OO_S(-m-1))=H^1(\OO_S(-m))=0$, and the resolution gives
an epimorphism $h^0(\varepsilon):\ H^0(\OO_S(-m+1)^{\oplus2})\to H^0(\OO_Z^{\oplus2})$. Thus,
$\otimes\OO_{\P^1}:\ H^0(\OO_S(-m+1)^{\oplus2})\xrightarrow{h^0(\otimes\OO_Z)}H^0(
\OO_Z^{\oplus2})\overset{h^0(\otimes\OO_{Z_2})}{\twoheadrightarrow}H^0(\OO_{Z_2}^{
\oplus2})$ is an epimorphism.

Since $\P^1\cap\Sing E=\varnothing$, then the sheaf $E|_{\P^1}$ is locally free, so
the triple \eqref{ext restr} shows that $E|_{\P^1}\cong\OO_{\P^1}(a)\oplus\OO_{\P^1}(-a
),\ 0\le a\le1$. On the other hand, for general $\xi_{\P^1}\in\Ext^1(\OO_{Z_2},\OO_{\P^1}
(-1)^{\oplus2}))$ in \eqref{ext restr} we have $E|_{\P^1}\cong\OO_{\P^1}^{\oplus2}$. From here
in view of surjectivity of $r$ it follows that for general $\xi\in\Ext^1(\OO_S(m),\OO_X(-1)
^{\oplus2}))$ the sheaf $E$ in \eqref{extension k=2,n=1} is a sheaf of general type.

(II) Consider the case $[E]\in M_{1,m,1}$ from Theorem \ref{Theorem 4.3}.
We have a triple \eqref{extension} for $k=n=1$, $m\le-1$, and $S\in\P^N=|\OO_X(1)|$:
\begin{equation}\label{extension k=n=1}
	0\to F(-1)\to E\to\OO_S(m)\to 0,\ \ \ \ \ \ \ \ \ \ \in|\OO_X(1)|.
\end{equation}
In this case, a general surface $S\in|\OO_X(1)|$ for both varieties $X=X_4$ and $X=X_5$
is a smooth del Pezzo surface containing a finite number of lines and satisfying the
condition $S\cap\Sing E=\varnothing$. Then the sheaf $E|_{\P^1}$ is locally free, and the triple
\eqref{extension k=n=1} restricted to any line $\P^1\subset S$ gives an epimorphism
$E|_{\P^1}\twoheadrightarrow\OO_{\P^1}(m)$, hence $E|_{\P^1}\cong\OO_{\P^1}(a)\oplus\OO
_{\P^1}(-a), \ a\le m$. Let now $\P^1$ be any line on $X$ such that $\P^1\not
\subset S$ and $\P^1\cap\Sing E=\varnothing$. Then, as before, the sheaf $E|_{\P^1}$ is locally
free, and restricting the triple \eqref{extension k=n=1} to $\P^1$ we obtain an exact triple
\begin{equation*}\label{ext restr S}
	0\to\OO_{\P^1}\oplus\OO_{\P^1}(-1)\to E|_{\P^1}\to\bk_x\to0,\ \ \ \ \ \ x =\P^1\cap S.
\end{equation*}
By analogy with the case (I), it suffices for us to check the surjectivity of the homomorphism
$r:\ \Ext^1(\OO_S(m),\SS(-1))\to\Ext^1(\bk_x,\OO_{\P^1}\oplus\OO_{\P^1} (-1))$, which is in
turn equivalent to the surjectivity of the homomorphism $h^0(\otimes\OO_Z):\ H^0(\SS(-m)|_S
)\to H^0(\OO_Z^{\oplus2})$ for a 0-dimensional scheme $Z=L\cap S$ containing the point $x$,
where, as above, $L$ is a subspace of codimension 2 in $\P^{q+1}$. The surjectivity of
$h^0(\otimes\OO_Z)$ for $m\le-1$ is checked using the resolution $0\to F(-m-2)|_S\to
F(-m-1)^{\oplus2}|_S\to F(-m)|_S\xrightarrow{\otimes\OO_Z}\OO_Z^{\oplus2}\to0$.

Note that the reflexivity and the stability of general $E$ in \eqref{extension k=2,n=1} and
\eqref{extension k=n=1} follow from Theorem \ref{Theorem 4.9}.
\end{proof}

\vspace{1cm}

\section{Moduli spaces of rank two semistable sheaves with maximal third Chern class on the three-dimensional quadric $X_2$}

\subsection{General case of rank two semistable sheaves with $-1\le c_1\le0$ and maximal class $c_3$ on the quadric $X_2$.}\label{subsec 5.1}

Stable sheaves $E$ of rank 2 with maximal third Chern class, described in Theorem
\ref{Theorem 3.1}, are special cases of sheaves included in the exact triples
\eqref{extension}, \eqref{extension I}, \eqref{extension Spin}, \eqref{extension Spin I}.
Namely, for a sheaf $E$ with $v=\ch(E)=(2,cH,dH^2,e[\mathrm{pt}])$ we have:\\
(1) For $c=-1$, non-integer $d\le-\frac32$ and $e=d^2-2d+\frac5{12}$, according to the statement (1.4)
of Theorem 3.1, the sheaf $E$ is included in the exact triple \eqref{extension} with $k=n=1,\ m=d-\frac12$,
$S=Q_2$.\\
(2) For $c=-1$, integer $d\le-1$ and $e=d^2-2d+\frac16$, according to the statement (1.3) of Theorem 3.1,
the sheaf $E$ is included in the exact triple \eqref{extension I} with $m=d$, $S=Q_2$.\\
(3) For $c=0$, non-integer $d\le-\frac32$ and $e=d^2+\frac14$, according to the statement (2.4)
of Theorem 3.1, the sheaf $E$ is included in the exact triple \eqref{extension Spin} with $k=n=1,\
m=d+\frac12$, $S=Q_2$.\\
(4) For $c=0$, integer $d\le-3$ and $e=d^2$, according to the statement (2.6) of Theorem 3.1, the sheaf $E$
is in the exact triple \eqref{extension Spin I} with $m=d$, $S=Q_2$.

In this section we prove the following theorem, which gives a complete description of moduli spaces
of rank 2 semistable sheaves with maximal third Chern class on the quadric $X_2$ for
of all possible values of the second Chern class $c_2$ in cases (1)-(4) above.

\begin{theorem}\label{moduli with c3 max}
	Let $X=X_2$ be the quadric and $M_X(v)$ be the Gieseker-Maruyama moduli scheme
	of rank 2 semistable sheaves $E$ on $X$ with Chern classes
	$(c_1,c_2,c_3)$, where $c_1\in\{-1,0\}$, $c_2\ge0$,
	$c_3=c_{3\mathrm{max}}$ is maximal for every $c_1$ and $c_2$, and
	\begin{equation*}
		v=\ch(E)=(2,c_1H,\frac12(c_1^2-c_2)H^2,\frac12(c_{3\mathrm{max}}+\frac23c_1^3-c_1c_2)
		[\mathrm{pt}])
	\end{equation*}
according to \eqref{c vers ch 3}. Then the scheme $M_X(v)$ is an irreducible and smooth
rational projective variety, all sheaves from $M_X(v)$ are stable, a general sheaf from
$M_X(v)$ is reflexive, and $M_X(v)$ is a fine moduli space. Moreover, we have
the following statements.\\
(i) For $c_1=-1$, an even $c_2=2p$, $p\ge2$, and $c_{3\mathrm{max}}=\frac12c_2^2$
the variety $M_X(v)$ is the grassmannization of 2-dimensional quotient spaces of the vector
bundle $\AA^{\vee}$ of rank $\rk\AA=\frac14(c_2+2)^2$ on the space $\P^4$,
defined by the first formula \eqref{sheaf A} for $n=1$ and $m=-p$. In this case $\dim M_X(v)=
\frac12(c_2+2)^2$.\\
(ii) For $c_1=-1$, an odd $c_2=2p+1$, $p\ge1$, and $c_{3\mathrm{max}}=\frac12(c_2^2-1)$
the variety $M_X(v)$ is the grassmannization of 2-dimensional quotient spaces of the vector
bundle $\widetilde{\AA}^{\vee}$ of rank $\rk\widetilde{\AA}^{\vee}=\frac14(c_2+1)(c_2+3)$ on the grassmannian $\mathbb{G}$ defined by the second formula \eqref{sheaf A} for
$m=-p$. In this case $\dim M_X(v)=\frac12(c_2+1)(c_2+3)$.\\
(iii) For $c_1=0$, an odd $c_2=2p+1$, $p\ge1$, and $c_{3\mathrm{max}}=\frac12(c_2^2+1)$
the variety $M_X(v)$ is the projectivization of the vector bundle $\AA_{\SS}^{\vee}$ of
rank $\rk\AA_{\SS}^{\vee}=\frac12(c_2+1)(c_2+3)$ on the space $\P^4$,
defined by the formula \eqref{sheaf AS} for $n=1$ and $m=-p$. In this case $\dim M_X(v)=
\frac12c_2^2+2c_2+\frac92$.\\
(iv) For $c_1=0$, an even $c_2=2p$, $p\ge3$, and $c_{3\mathrm{max}}=\frac12c_2^2$
the variety $M_X(v)$ is the projectivization of the vector bundle $\widetilde{\AA}
_{\SS}^{\vee}$ of rank $\rk\widetilde{\AA}_{\SS}^{\vee}= \frac12c_2^2+2c_2+1$ on
the grassmannian $\mathbb{G}$, defined by the formula \eqref{sheaf AS I} for $n=1$ and $m=1-p$.
In this case $\dim M_X(v)=\frac12c_2^2+2c_2+4$ .\\
\end{theorem}
\begin{proof} 
(i)-(ii) Let us show that any sheaf $E$ defined by the extension \eqref{extension} for
the case $k=n=1,\ m=-p$, $S\in|\OO_X(1)|$:
\begin{equation}\label{ext 12}
	0\to\OO_X(-1)^{\oplus 2}\to E\to\OO_{S}(-p)\to0,\ \ \ \ \ p\ge2,
\end{equation}
or extension \eqref{extension I} for $m=-p$:
\begin{equation}\label{ext 12 J}
	0\to\OO_X(-1)^{\oplus 2}\to E\to\JJ_{\P^1,S}(-p)\to0,\ \ \ \ \ p\ge1,
\end{equation}
is stable, that is, according to the definitions \eqref{def Y} and \eqref{def tilde Y}, that
\begin{equation}\label{Y=Gr}
	\YY_{1,m,1}={\GG}r(\AA^{\vee},2),\ \ \ m\le-2;\ \ \ \ \ \ \ \ \ \ \ \
	\widetilde{\YY}_m={\GG}r(\widetilde{\AA}^{\vee},2),\ \ \ m\le-1.
\end{equation}
Indeed, specifying an extension \eqref{ext 12} is equivalent to specifying an element of the space
$\Ext^1(\OO_{S}(-p),\OO_X(-1)^{\oplus2})$.
For a semistable sheaf $E$, the triple \eqref{ext 12} gives its tilt
Harder--Narasimhan filtration in the sense of \cite[Lem. 3.2.4]{BMT} if $E$ is tilt-unstable.
The Harder--Narasimhan tilt filtration factors are unique, that is, $E$ determines
a surface $S$ and a subsheaf $\OO_X(-1)^{\oplus 2}$. In this case, the group $\GL(2)$ acts
by automorphisms on the sheaf $\OO_X(-1)^{\oplus 2}$ without changing the isomorphism class of the sheaf $E$. This
means that there is a subspace in 
$\Ext^1(\OO_{S}(-p),\OO_X(-1))$. If this subspace is zero-dimensional, then $E$ is a
direct sum and, therefore, is unstable. Let us assume that this subspace is one-dimensional.
Then the morphism $\OO_{S}(-p)\to\OO_X(-1)^{\oplus2}[1]$ factorizes through $\OO_X(-1)[1]$.
Then the octahedron axiom implies the existence of a mapping $E\onto\OO_X(-1)$, contrary to
stability of $E$.
	
Conversely, let a 2-dimensional subspace in $\Ext^1(\OO_{S}(-p),\OO_X(-1))$ be given. 
A choice of two basis vectors in this subspace
gives an extension \eqref{ext 12}. Here $E$ is strictly semistable along the induced wall
$W$. Its Jordan-H\"{o}lder filtration factors along the wall are two copies of the sheaf $\OO_X(-1
)$ and one copy of the sheaf $\OO_{S}(-p)$. Then the Jordan--H\"{o}lder factors of any
destabilizing subobject for $E$ are contained among the specified sheaves. By construction
we have $\Hom(E,\OO_X(-1))=0$. Therefore, neither the sheaf $\OO_{S}(-p)$ nor the extension of the sheaf
$\OO_{S}(-p)$ by $\OO_X(-1)$ cannot be a subobject of $E$.

In the case of the extension \eqref{ext 12 J} the reasoning is similar to the case of the extension
\eqref{ext 12}. Now applying statements (i) and (ii.2) of Theorems \ref{Theorem 4.1} and
\ref{Theorem 4.2} to the varieties $\YY_{1,m,1}={\GG}r(\AA^{\vee},2)$ and $\widetilde{\YY}
_m={\GG}r(\widetilde{\AA}^{\vee},2)$ from \eqref{Y=Gr}, respectively, we obtain the statements
(i) and (ii) of the Theorem.

(iii)-(iv) Let us show that any sheaf $E$ defined by extensions \eqref{extension Spin}
and \eqref{extension Spin I} for the case $k=n=1,\ m=-p$, $S\in|\OO_X(1)|$:
\begin{equation}\label{triple 7}
	0\to\SS(-1)\to E\to\OO_{S}(-p)\to0,\ \ \ \ \ p\ge1,
\end{equation}
\begin{equation}\label{triple 7 J}
	0\to\SS(-1)\to E\to\JJ_{\P^1,S}(-p)\to0,\ \ \ \ \ p\ge2.
\end{equation}
is stable, that is, according to the definitions of $\YY_{\SS,1,m,1}$ and $\widetilde{\YY}_{\SS,1,m,1}$
(see Remarks \ref{Remark 4.B} and \ref{Remark 4.B I} respectively) that
\begin{equation}\label{Y=Gr S}
	\YY_{\SS,1,m,1}=\P(\AA_{\SS}^{\vee}),\ \ \ \ \ m\le-1;\ \ \ \ \ \ \ \ \
	\widetilde{\YY}_{\SS,m}=\P(\widetilde{\AA}_{\SS}^{\vee}),\ \ \ \ \ m\le-2.
\end{equation}
Indeed, an extension \eqref{triple 7} corresponds to an element from $\Ext^1(\OO_{S}
(-p),\SS(-1))$. For semistable objects this extension gives tilt
Harder--Narasimhan filtration for $E$ provided that $E$ is tilt-unstable. Harder--Narasimhan
factors in this case are unique. This means that $E$ determines a two-dimensional quadric
$S$ and the sheaf $\SS(-1)$ as a subobject of $E$. Multiplying the morphism $\SS(-1)\to E$ by a scalar
does not change the isomorphism class of the sheaf $E$. In addition, if $E$ is a direct sum,
it is not stable. 
	
Conversely, let a one-dimensional subspace in $\Ext^1(\OO_{S}(-p),\SS(-1))$ be given. Choice
of any nonzero vector in this subspace gives a nontrivial extension
\eqref{triple 7}. The object $E$ in it is strictly semistable along the induced wall $W$.
The Jordan--H?lder factors of the object $E$ are uniquely defined, and the only subobjects that
can destabilize $E$ above the wall are either $\SS(-1)$ or $\OO_{S}(-p)$.
However, $\SS(-1)$ does not destabilize $E$ for numerical reasons.
On the other hand, the non-splitability of the previous exact triple means that $\OO_{S}(-p)$
cannot be a subobject of $E$.

In the case of the extension \eqref{triple 7 J} the argument is similar to the case of the extension
\eqref{triple 7}. Now applying the statements (i) and 
(ii.2) of Theorems \ref{Theorem 4.1S} and
\ref{Theorem 4.2S} to the varieties $\YY_{\SS,1,m,1}={\GG}r(\AA_{\SS}^{\vee},2)$ and
$\widetilde{\YY}_{\SS,1,m,1}={\GG}r(\widetilde{\AA}_{\SS}
^{\vee},2)$ from \eqref{Y=Gr S} respectively, we obtain the statements (iii) and (iv) of the Theorem.
\end{proof}

\subsection{Special cases of rank two semistable sheaves with small values of class $c_2$ and maximal class $c_3\ge0$ on $X_2$.}\label{subsec 5.2}
Theorem \ref{moduli with c3 max} covers most of the possible values of Chern classes of rank 2 semistable sheaves $E$ on the quadric $X_2$. The remaining special cases for $c_3\ge0$ and small values of class $c_2$ are indicated in statements (1.1), (1.2), (2.1),
(2.3) and (2.5) of Theorem \ref{Theorem 3.1}. Below in Theorem \ref{moduli with c3 max 2} we
describe the moduli of sheaves $E$ for the first three of the indicated cases, and in Theorems
\ref{moduli with c3 max 3} and \ref{moduli with c3 max 4} --- for the last two cases
respectively.
\begin{theorem}\label{moduli with c3 max 2}
Under the conditions and notation of Theorem \ref{moduli with c3 max} the following statements are true:\\
(1) For $c_1=-1$, $c_2=1$ and $c_{3\mathrm{max}}=0$ the variety $M_X(v)$ is a point $[\SS(-1)]
$.\\
(2) For $c_1=c_2=c_{3\mathrm{max}}=0$ the variety $M_X(v)$ is a point $[\OO_X^{\oplus2}]
$.\\
(3) For $c_1=-1$, $c_2=2$ and $c_{3\mathrm{max}}=2$ we have $M_X(v)\simeq \mathrm{Gr}(2,5)$. \\
\end{theorem}
\begin{proof}
Statements (1) and (2) follow directly from statements (1.1) and (2.1) of Theorem
\ref{Theorem 3.1}. Let us consider the last case.

(3) In this case, according to statement (1.2) of Theorem \ref{Theorem 3.1} the sheaf $E$ is
determined by an exact triple
\begin{equation}\label{E=coker f}
	0\to\OO_X(-2)\xrightarrow{f}\OO_X(-1)\otimes U\to E\to0.
\end{equation}
The morphism $f$ can be considered as a 3-dimensional quotient $U$ of the space\\
$H^0(\OO_X(1))^{\vee }$, that is, as a point of the grassmannian $\mathrm{Gr}(2,5)$. Let us consider
the mapping $\varphi:\mathrm{Gr}(2,5)\to M: (H^0(\OO_X(1))^{\vee}\onto U)\mapsto E$, where
$E$ is defined in \eqref{E=coker f}.

Let us check that the function $\varphi$ is correctly 
defined. To do this, it is enough to make sure that
$E$ is $\mu$-stable. By Theorem \ref{Theorem 3.1} there is 
a unique wall $W$ for the above objects $E$, given by the 
exact triple
\begin{equation}\label{triple (1.2)}
0\to\OO_X(-1)^{\oplus 3}\to E\to\OO_X(-2)[1]\to0.
\end{equation}
Therefore, checking that an object $E=\varphi(U)$ is $\mu$-stable is equivalent to checking
$\nu_{\alpha, \beta}$-stability of $E$ in the neighborhood of this wall $W$ above it. If $E$
is not semistable over $W$, then there is a destabilizing semistable quotient $E\onto G$.
Since $E$ is strictly semistable along $W$, the quotient satisfies the equality $\nu_{
\alpha,\beta}(E)=\nu_{\alpha, \beta}(G)$ for $(\alpha,\beta)$ on $W$. As is known,
any Jordan--H\"older filtration for $E$ has three stable factors $\OO_X(-1)$ and
one stable factor $\OO_X(-2)[1]$. Thus, the stable factors of the object $G$ are
contained among these sheaves. The factor $\OO_X(-2)[1]$ does not destabilize $E$ over $W$ by
numerical reasons. The vector space $\Hom(E,\OO_X(-1))$ is the kernel of
a homomorphism
$\Hom(\OO_X(-1)\otimes U,\OO_X(-1))\to\Hom(\OO_X(-2), \OO_X(-1))$, which is injective.
Therefore, $G\not\cong\OO_X(-1)^{\oplus a}$ for $a\in\{1,2,3\}$. If $G$ is an
extension of $\OO_X(-1)$ by $\OO_X(-2)[1]$, then the corresponding subobject is
$\OO_X(-1)^{\oplus 2} $ and it does not destabilize $E$ above the wall. Likewise, if $G$ is an
extension of $\OO_X(-1)^{\oplus 2}$ by $\OO(-2)[1]$, then the corresponding subobject
is given by the sheaf $\OO_X(-1)$ that does not destabilize $E$ above the wall. Thus $E$ is
$\mu$-stable.

Let us check that $\varphi$ is a bijection. According to Theorem \ref{Theorem 3.1}.(1.2) any
semistable object $E$ with $c_1=-1$, $c_2=2$ and $c_{3\mathrm{max}}=0$ satisfies
the exact triple \eqref{triple (1.2)}. Specifying such an extension is equivalent to specifying
an element $\Ext^1(\OO_X(-2)[1], \\ \OO_X(-1)^{\oplus 3})=H^0(\OO_X(1))^{\oplus 3}$ . In
Theorem \ref{Theorem 3.1} we showed that this extension is the tilt
Harder--Narasimhan filtration for $E$ under the wall. The factors of this filtration are unique. This
means that $E$ determines the subobject $\OO_X(-1)^{\oplus 3}$. However, the group $GL(3)$
acts by automorphisms on $\OO_X(-1)^{\oplus 3}$ without changing the isomorphism class of $E$. This
means that we get a unique subspace in $H^0(\OO_X(1))$.
However, if this subspace is not three-dimensional, then there is a destabilizing morphism
$E\onto\OO_X(-1)$. This simultaneously proves surjectivity and injectivity of
$\varphi$.

Let us show that $\varphi$ is a morphism of schemes. To do this, it is enough to construct a universal
family $\EE$ of sheaves $E$ on $\mathrm{Gr}(2,5)\times \P^3$. Universal property of
the scheme $M_X(v)$ then shows that $\varphi$ is a morphism. We have two projections
$\mathrm{Gr}(2,5)\xleftarrow{p_1}\mathrm{Gr}(2,5)\times X\xrightarrow{p_2}X$ and
the rank 3 tautological quotient bundle of $\OO_{\mathrm{Gr}(2,5)}\otimes
H^0(\OO_X(1))^{\vee}\onto\QQ$ on $\mathrm{Gr}(2,5)$. The required universal family
$\EE$ is obtained as the cokernel of the composition of morphisms $p_2^*\OO_X(-2)\to H^0(\OO_X(1))^{\vee}
\otimes p_2^*\OO_X(-1)\to\QQ\boxtimes\OO_X(-1)$.

To prove the smoothness of $M_X(v)$, it is enough to show that $\Ext^2(E,E)=0$. This
is obtained by applying the three functors ${\mathbf R}\Hom(-,\OO_X(-2)[1])$, ${\mathbf R}
(-,\\ \OO_X(-1))$, and ${\mathbf R}\Hom(E,-)$ to the exact triple \eqref{triple (1.2)}. From the
smoothness of $M_X(v)$ and bijectivity of $\varphi$ and \cite[Ch. 2, \S4.4, Thm. 2.16]{Sh} it follows
that $\varphi$ is an isomorphism.
\end{proof}

\begin{theorem}\label{moduli with c3 max 3}
	For $c_1=0$, $c_2=2$ and $c_{3\mathrm{max}}=2$ the scheme $M_X(v)$ is irreducible and has
	dimension 9 and is not smooth.
\end{theorem}

\begin{proof}
	Let $E$ be a Gieseker-semistable sheaf with $c_1=0$, $c_2=2$ and $c_{3\mathrm{max}}=
	2$, that is, $[E]\in M_X(v)$, where $v=(2,0,-H^2,\{\mathrm{pt}\})$. By Theorem
	\ref{Theorem 3.1}.(2.3) $E$ is destabilized by an exact triple
	\begin{equation}\label{non-smooth}
		0\to\SS(-2)^{\oplus2}\to\OO_X(-1)^{\oplus 6}\to E\to 0.
	\end{equation}

Let us prove that there exist stable sheaves $E$ included in \eqref{non-smooth}.
Let $E$ be properly semistable, then there exists a subsheaf $F\subset E$ of rank 1 such that
that $\chi(F(m))=\frac 12\chi(E(m))$ for $m\gg 0$. From get thathere we obtain $F\in M(1,0,-\frac12 H^2,\frac 12)$. Since $F$ is stable, by the proof of Lemma \ref{lemma 3.5} we find
an exact triple $0\to\SS(-2)\to\OO_X(-1)^{\oplus3}\to E\to 0.$ Since $\mathrm{Ext}^1 (\OO_X(-1)^ {\oplus 6},\SS(-2)^{\oplus 2})$, the embedding $F\to E$ induces a morphism $\OO_X (-1)^{\oplus 3}\to \OO_X(-1 )^{\oplus 6}$. Let us denote $W=H^0(\SS(1))\cong\bk^{16}$ and $A=H^0(\alpha^\vee(-1)):\bk^6\to\bk ^2\otimes W$, then we get a commutative diagram with exact rows
$$
\xymatrix{ 0\ar[r] &\bk^3 \ar[r]\ar[d] & \bk^6 \ar[d]^A\ar[r] & \bk^3 \ar[d ]\ar[r]&0\\
	0\ar[r] & \bk\otimes W\ar[r] & \bk^2\otimes W\ar[r] & \bk\otimes W\ar[r] & 0}.
$$
In other words, the matrix of the mapping $A$ in some bases of the spaces $\bk^6$ and
$\bk^2$ (consistent with the choice of subspaces $\bk^3\subset\bk^6,\bk\subset\bk^2$)
has the form
$$\begin{pmatrix}
	* & * & * & * & * & * \\
	0 & 0 & 0 & * & * & *
\end{pmatrix}$$
Consider the incidence variety
$$
X=\{(A_i)_{i=1}^{16}\in\mathrm{Hom}(\bk^6,\bk^2)^{\oplus16},V_1\in\mathrm{Gr}(3, 6),
V_2\in\mathrm{Gr}(1,2)|A_i(V_1)\subset V_2\}
$$
with the projection of $X\to Y=\mathrm{Hom} (\bk^6,\bk^2)^{\oplus 16}$. Because
$\dim\mathrm{Gr}(3,6)=9,\linebreak\dim\mathrm{Gr}(1,2)=1$ and on each of the 16 matrices $A_i$
three independent linear conditions are imposed, $\dim 
X=16\cdot12+9+1-16\cdot3<16\cdot12=\dim Y$. It follows 
that there is a dense open subset of the variety
$Y$ that does not intersect the image of $X$ under the projection $X\to Y$. This proves the existence
of Gieseker-stable sheaves $E=\coker\alpha$.

Since $E$ is tilt-semistable over the semicircular wall, then $\Hom(E,\OO_X(-1))=0$. In addition, $\Ext^1(\SS(-2),\SS(-2))=0$ (since $\SS(-1)$ is an
exceptional object of the category $D^b(X)$) and $\Ext^2(\OO_X(-1),\SS(-2))=0$ (since $\SS(-1)$ is an ACM sheaf). Now from (\ref{non-smooth}) the equality $\Ext^2(E, \SS(-2)^{\oplus2})=0$ follows, so we have an exact sequence
\begin{equation}\label{ext 1}
	\begin{split}
		& 0\to\Hom(E,E)\to\Ext^1(E,\SS(-2)^{\oplus 2})\to \\
		& \Ext^1(E,\OO_X(-1)^{\oplus 6})\to\Ext^1(E,E)\to0.
	\end{split}
\end{equation}
If now the sheaf $E$ is stable, then $\hom(E,E)=1$, and a calculation by \eqref{ext 1} and
\eqref{non-smooth} gives $\rmext^1(E,E)=9$. On the other hand, from the triple \eqref{non-smooth}
it follows that $M_X(v)$ is irreducible and $\dim M_X(v)=12h^0(\SS)-\dim(GL(2,\bk)\times GL(6,\bk)
/\bk^*)=12\cdot4-4-36+1=9$, so $M_X(v)$ is smooth at the point $[E]$.

In the properly semistable case we use the following result
\cite[p. 294, Lemma]{MRT}:
\begin{lemma}\label{tangent}
	Let $E$ be a semistable sheaf of rank 2, $E\cong E_1\oplus E_2$, where $\chi(E_1(m))=
	\chi(E_2(m))=\frac 12\chi(E(m))$ and $E_1\not\cong E_2$. If $\Ext^2(E,E)=0$, then
	the tangent space to the moduli scheme of semistable sheaves at the point $[E]$ is isomorphic to
	$\Ext^1(E_1,E_1)\oplus(\Ext^1(E_1,E_2)\oplus\Ext^1(E_2,E_1))\oplus\Ext^1(E_2,E_2).$
\end{lemma}
In our case, we apply this lemma to the sheaf $E\cong E_1\oplus E_2$, where $E_1\not\cong E_2$,
the sheaves $E_1$ and $E_2$ are stable and included in exact triples of the form
\begin{equation}\label{non-smooth1}
	0\to\SS(-2)\to\OO_X(-1)^{\oplus 3}\to E_i\to 0.
\end{equation}
Since $\Ext^1(\SS(-2),\OO_X(-1)))\cong H^1(\SS)=0$ and $\Ext^2(\OO_X(-1),\ OO_X(-1))=0$,
then, taking into account \eqref{non-smooth1} we have $\Ext^2(E,\OO_X(-1)^{\oplus 6})=0$. On the other side, since $\Ext^2(\SS(-2),\SS(-2))=\Ext^3(\OO_X(-1),\SS(-2))=0$ due to the
exceptionality of the pair $(\SS(-1),\OO_X)$ in $D^b(X)$, then $\Ext^3(E,\SS(-2))=0$. From here
and from \eqref{non-smooth} we get $\Ext^2(E,E)=0$. Now by Lemma \ref{tangent}
the dimension of the tangent space to $M_X(v)$ at the point $[E]=[E_1\oplus E_2]$ is equal to
$3+2+2+3=10\neq 9$, so $M_X(v)$ is not smooth.
\end{proof}

\begin{theorem}\label{moduli with c3 max 4}
	For $c_1=0$, $c_2=4$ and $c_{3\mathrm{max}}=8$, the scheme $M_X(v)$ is irreducible and equal to a union of two irreducible subsets $M_1$ and $M_2$. These subsets are described as follows.\\
	(i) $M_1$ is a smooth rational variety of dimension 20, which is
	the projectivization of a locally free sheaf of rank 17 on $\mathbb{G}$. $M_1$ is a fine
	moduli space, and all sheaves from $M_1$ are stable. The scheme $M_X(v)$ is
	nonsingular along $M_1$.\\
	(ii) the scheme $M_2$ is irreducible, has dimension 21, and polystable sheaves from $M_2$
	constitute a closed subset of dimension 12 in $M_2$ in which the scheme $M_X(v)$ is not
	smooth.
\end{theorem}
\begin{proof}
Let $E$ be a Gieseker-semistable sheaf with $c_1=0$, $c_2=2$ and $c_{3\mathrm{max}}
=2$, that is, $[E]\in M_X(v)$, where $v=(2,0,-2H^2,4\{\mathrm{pt}\})$. By Theorem
\ref{Theorem 3.1}.(2.5) the sheaf $E$ is destabilized by an exact triple
\begin{equation}\label{tr 2.1}
	0\to\SS(-1)\to E\to\II_{\P^1,S}(-1)\to0,
\end{equation}
or an exact triple
\begin{equation}\label{tr 2.2}
	0\to\OO_X(-2)^{\oplus2}\to\OO_X(-1)^{\oplus 4}\to E\to0.
\end{equation}
(i) According to Lemma \ref{2,0,-4} any sheaf $E$ in the triple \eqref{tr 2.1} is tilt-stable
for $\beta<0,\alpha\gg0$, and therefore is stable according to Proposition \ref{Prop 2.1}.
In other words, in the notation of Remark \ref{Remark 4.B I} we have an equality $\widetilde{\YY}
_{\SS,1,m,1}=\P(\widetilde{\AA}_{\SS}^{\vee})$, similar to the second equality
\eqref{Y=Gr S} (in our case $m=-1$). Therefore, arguing as in the proof of Theorem \ref{Theorem 4.2S}, we obtain then the bundles $E$ included in the triple \eqref{tr 2.1} are stable and form a smooth irreducible projective subvariety $M_1$ of the scheme $M_X(v)$, described as the projectivization of $\widetilde{\pi}:\ \P(\widetilde{\AA}_{\SS}^{\vee})\to\mathbb G$, where $\AA$ is a locally free sheaf of rank 17 on $\mathbb G$. Hence $\mathrm{dim}\ M_1=\mathrm{dim}\ \mathbb G+\mathrm{rk}\ \widetilde{\AA}_{\SS}-1=4+17-1=20$. On the other hand the virtual dimension of $M_X(v)$
is $21$ (direct calculations by the Riemann–Roch–Grothendieck theorem for the given Chern character $v$), so that any irreducible component of $M_X(v)$ with stable general sheaf has dimension $\ge 21$. Consequently, the 20-dimensional projective variety $M_1$ is not an irreducible component of the scheme $M_X(v)$.

(ii) Let us denote $M_2=\{[E]\in M_X(v)\ |\ E\ \text{is included in the triple}\ \eqref{tr 2.2}\}$.
From \eqref{tr 2.2} we obtain that $M_2$ is irreducible as a quotient of an open subset
of the space $\Hom(\OO_X(-2)^{\oplus2},\OO_X(-1)^{\oplus4})$ by the group $\dim(GL(2,\bk)
\times GL(4,\bk)/\bk^*$, and $\dim M_2=\hom(\OO_X(-2)^{\oplus2},\OO_X(-1)^{\oplus4
})-\dim(GL(2,\bk)\times GL(4,\bk)/\bk^*)=2\cdot4\cdot5-4-16+1=21$. Next we will carry out for the
sheaves $E$ from $M_2$ an argument similar to the argument with the triple \eqref{non-smooth},
carried out in the proof of Theorem \ref{moduli with c3 max 3}. In particular,
it is verified that $M_2$ contains stable sheaves $E$. For stable sheaves $E$
we have $\Hom(E,\mathcal O_X(-1)^{\oplus 4})=0$, and also $\Ext^2(E,\OO_X(-2)^{\oplus2})=0$
due to the exceptionality of line bundles on $X$ and the ACM property. Thus, in
this case there is a similar to (\ref{ext 1}) exact sequence
$$
0\to\Hom(E,E)\to\Ext^1(E,\OO_X(-2)^{\oplus2})\to\Ext^1(E,\OO_X(-1)^{\oplus4 })\to\Ext^1
(E,E)\to0.
$$
A standard calculation with this sequence using 
\eqref{tr 2.2} shows
that the dimension $\rmext^1 (E,E)$ of the tangent space to $M_X(v)$ at the point $[E]$ for
stable $E$ is equal to 21, which coincides with $\dim M_2$. Thus, $M_2$ is an irreducible open 21-dimensional subset of the irreducible component $\overline{M_2}$ of the scheme $M_X(v)$, where $\overline{M_2}$ is the closure of $M_2$ in $M_X(v)$. Since by (\ref{tr 2.1}) and (\ref{tr 2.2}) we have $M_X(v)=M_1\cup M_2=M_1\cup\overline{M_2}$, it follows that if $M_1\not\subset\overline{M_2}$, then the variety $M_1$
is an irreducible component of $M_X(v)$, in contradiction to the above. Hence $M_1\subset\overline{M_2}$, and therefore the scheme $M_X(v)=\overline{M_2}$ is irreducible.

Let us now consider a point $[E]=[E_1\oplus E_2]$, where $E_1\not\cong E_2$ and $E_i$ are included
into exact triples of the form
\begin{equation}\label{Ei}
	0\to\OO_X(-2)\to\OO_X(-1)^{\oplus 2}\to E_i\to0,\ \ \ \ \ i=1,2,
\end{equation}
A simple calculation using these triples shows that $\Ext^2(E_i,E_j)=0,\ 1\le i,j\le2,$
$\rmext^1(E_i,E_i)=6,\ i=1,2$, $\rmext^1(E_i,E_j)=5,\ i\ne j$. Hence we have $\Ext^2(E,E)
=0$, and by Lemma \ref{tangent} $\dim T_{[E]}M_X(v)=6+5+5+6=22\neq 21$, therefore
$M_X(v)$ is not smooth at the point $[E]$.
\end{proof}

\vspace{5mm}

\section{Boundedness of the third Chern class of rank two stable reflexive sheaves of general type with $c_1=0$ on the varieties $X_4$ and $X_5$}

\vspace{5mm}

In this section we consider rank two stable reflexive sheaves
of general type with $c_1=0$ (see Definition \ref{def 4.10}) on varieties
$X_4$ and $X_5$. For any such sheaf $E$ we prove boundedness
from above of the third Chern class $c_3$ of the sheaf $E$ by a quadratic polynomial in
the second Chern class $c_2$ -- see Theorems \ref{Thm 6.1} and \ref{Thm 6.4}
below.

\subsection{Rank two stable reflexive sheaves of general type with $c_1=0$ on the variety $X_5$}
In this subsection everywhere $X=X_5$.

Let us recall some well-known facts about the variety $X$ (see, for example,
\cite[Th. 4.2(iii) and Cor. 6.6(ii)]{Isk}).\\
1) The base $B=B(X)$ of the family of lines on $X$ is isomorphic to $\p2$.\\
2) For an arbitrary line $l\in B$, the set $B_l=\{l'\in B\ |\
l'\cap l\ne\varnothing\}$ is a line in $\P^2$.\\
3) For an arbitrary line $l\in B$, either $N_{l/X}\cong\OO_{\P^1}^{\oplus2}$ (general case), or $N_{l/X}\cong\OO_{\P^1}(1)\oplus\OO_{\P^1}(-1)$.\\
4) Let $\varphi:\ X\dasharrow \p4$ be the linear projection of the variety
$X\subset\p6$ into the space $\p4$ from a line $l\subset X$, for which
$N_{l/X}\cong\OO_{\p1}^{\oplus2}$. Then $\varphi(X)=Q\cong X_2$ is a smooth quadric in
$\P^4$, the morphism $\varphi: X\dasharrow Q$ is birational and decomposes into a Hironaka's roof
\begin{equation}\label{eqn 21}
	\xymatrix{
		& \ar[dl]_{\delta} \ar[dr]^{\sigma}\widetilde{X} & \\
		X \ar@{-->}[rr]^-{\varphi} & & Q.
	}
\end{equation}
Here $\delta^{-1}:\ X\dasharrow\widetilde{X}$ is the blow-up of $X$ centered at $l$, $S
=\delta^{-1}(l)\simeq\p1\times\p1$, $\sigma|_S:S\xrightarrow{\simeq}\sigma(S) =Q_2$ is
an isomorphism of $S$ onto a smooth two-dimensional quadric $Q_2\subset Q$, and $\sigma^{-1}:\
Q\dasharrow\widetilde{X}$ is the blow-up of $Q$ centered in a smooth rational cubic
curve $C$ of type (2,1) on $Q_2$, so $\widetilde{X}\simeq\p(\II_{C,Q})$.
In particular, a triple $0\to N_{C,Q_2}\to N_{C/Q}\to N_{Q_2/Q}|_{C}\to0$ is exact, where due to
isomorphisms $Q_2\simeq\p1 \times\p1$ and $C\simeq\p1$ we obtain $N_{C/Q_2}\cong\OO_{\p1}
(3)$, $N_{Q_2/Q}|_{C}\cong\OO_Q(H)|_{C}\cong\OO_{\p1}(3)$, hence $\det N_{C/ Q}\cong
\OO_{\p1}(6)$. \\
5) Let $\Gamma=\{(x,l)\in X\times B\ |\ x\in l\}$ be an incidence graph with projections
$X\xleftarrow{p_1}\Gamma\xrightarrow{p_2}B$. As we know \cite[Prev. 5.2]{Isk},
$p_1:\Gamma\to X$ is a morphism, which is finite at a general point, and we have $\dim\{x\in X\ |\ \dim p_1^{-1}
(x)\ge1\}\le0$.

Consider an extension
\begin{equation}\label{xi on C 0}
	\xi:\ 0\to\OO_Q(nH)\to F\xrightarrow{\varepsilon}\II_{C,Q}\to0.
\end{equation}
The sheaf $\lHom(\II_{C,Q},\OO_Q(nH))\cong\OO_Q(nH)$ on $Q$ and the sheaf $\MM=\lExt^1(\II_{C,Q}, \\
\OO_Q(nH))\cong\lExt^2(\OO_C,\OO_Q(nH))\cong\det N_{C/Q}\otimes\OO_Q(nH)\cong\OO_{\p1}(6+3n)$ on $C$ for $n\gg0$ are very ample. Therefore, we can assume that:\\
1) $H^i(\lHom(\II_{C,Q},\OO_Q(nH)))=0,\ i=1,2$, which implies that the long exact
sequence of $\Ext$-groups for a pair of sheaves $\II_{C,Q},\OO_Q(nH)$ gives
isomorphisms $\Ext^1(\lHom(\II_{C,Q},\OO_Q(nH))\cong H^0(\lExt^1(\II_{C,Q},\OO_Q(nH))
\cong H^0(\lExt^2(\OO_C,\OO_Q(nH))=H^0(\MM)$;\\
2) The triple \eqref{xi on C 0} as an extension $\xi\in\Ext^1(\lHom(\II_{C,Q},\OO_Q(nH)))=H^0(\MM)$ 
is a section of the sheaf $\MM$, the scheme of zeros $(\xi)_0$ of which is reduced,
i.e. it is a simple divisor $D_{\xi}=x_1+...+x_{3n+6}$. This means by
Serre's construction (see \cite[Section 4]{H}), that the sheaf $F$ in \eqref{xi on C 0}
is a reflexive sheaf of rank two with $\Sing F=D_{\xi}$ and the simplest
singularities at the points $x_i$, i.e. $\lExt^1(F,\OO_Q)\cong\oplus_{i=1}^{3n+6}\mathbf{k}_{x_i}$. 
Thus, according to \cite[Cor. 2.8]{BW}, the projective spectrum $Y=\P(F)$
of the sheaf $F$ is a smooth variety. The epimorphism $\varepsilon$ in
\eqref{xi on C 0} corresponds to an embedding $i$ of the divisor $\widetilde{X}=\P(\II_{C,Q})$ into
$Y$, and let $p:Y\to Q$ be the natural projection, so that $\sigma=p\circ i$. Note
also that the sheaf
\begin{equation}\label{sheaf L}
	L=\OO_{Y/Q}(1)\otimes p^*\OO(mH)
\end{equation}
is ample for $m\gg0$, where $\OO_{Y/Q}(1)$ is the Grothendieck sheaf on $Y=\P(F)$.

Let
\begin{equation}\label{E stable}
\begin{split}
& E\ \textit{be a rank 2 stable reflexive sheaf of general type with}\ c_1(E)=0\\
& \textit{on}\ X;\ \textit{in particular,}\ H^0(E)=0.
\end{split}
\end{equation}

Since $\dim\Sing E\le0$, it follows from properties 1) -- 5) that we can choose
for $l$ a general line on $X$ for which $N_{l/X}\cong\OO_{\p1}^{\oplus2}$ 
such that the following conditions are satisfied.\\
a) $B_l\cap p_2(p_1^{-1}(\Sing E))=\varnothing$, so the surface $S_l=p_1(p_2^{-1}(B_l))$
does not intersect the set $\Sing E$. Since by construction $S_l=\delta(S)$, then
this means that the sheaf
\begin{equation}\label{def tilde E}
	\widetilde{E}=\delta^*E
\end{equation}
satisfies the condition
\begin{equation}\label{tE cap S}
	\Sing\widetilde{E}\cap S=\varnothing.
\end{equation}
Note that from the projection formula for the blow-up $\delta$ it follows that
$E=\delta_*\widetilde{E}$, $R^i\delta_*\widetilde{E}=0,\ i>0,$
Therefore
\begin{equation}\label{chi=chi}
	\chi(E)=\chi(\widetilde{E})
\end{equation}
and, in addition, \eqref{E stable} implies
\begin{equation}\label{tilde E stable}
	H^0(\widetilde{E})=0.
\end{equation}

b) Since $E$ is a sheaf of general type, the set $G_C=\{x\in C\ |\ \widetilde{E}|
_{\sigma^{-1}(x)}\cong\OO_{\p1}(1)\oplus\OO_{\p1}(-1)\}$ is a proper
subset of the curve $C$, and $\widetilde{E}|_{\sigma^{-1}(x)}\cong\OO_{\p1}^{\oplus2}$
for $x\in C\smallsetminus G_C$.

In addition, we will assume that the section $\xi$ of the sheaf $\MM$ is sufficiently general, so that
\begin{equation}\label{Dxi cap empty}
	D_{\xi}\cap G_C=\varnothing.
\end{equation}
Below we will need another sufficiently general section
$\xi'\in H^0(\LL)$ such that
\begin{equation}\label{Dxi' cap empty}
	D_{\xi'}\cap G_C=\varnothing,\ \ \ \ \ \ D_{\xi}\cap D_{\xi'}
	=\varnothing.
\end{equation}

Put
\begin{equation}\label{defn of Q0 etc}
\begin{split}
& Q^0:=Q\smallsetminus(p(\Sing\widetilde{E})\cup D_{\xi}),\ \ \ \ \ \
C^0:=C\smallsetminus D_{\xi}\hookrightarrow Q^0,\\
& X^0:=\sigma^{-1}(Q^0)=\widetilde{X}\smallsetminus(\Sing
\widetilde{E}\cup\sigma^{-1}D_{\xi}),\\
& Y^0:=Y\smallsetminus p^{-1}(Q^0),\ \ \ \ \ \
\widetilde{E}^0:=\widetilde{E}|_{X^0}, \ \ \ \ \ \
\overline{E}^0:=\sigma_*\widetilde{E}^0,
\end{split}
\end{equation}
where $X^0$ becomes a divisor in $Y^0$ via the embedding $i$ defined above, so
that the following diagram is commutative:
\begin{equation}\label{eqn 26}
	\xymatrix{
		X^0\ \ar[dr]_{\sigma}\ar@{^{(}->}[rr]^-{i} & & Y^0\ar[dl]^{p} \\
		&Q^0. &
	}
\end{equation}
Note that, by \eqref{defn of Q0 etc}, $\Sing F=D_{\xi}$ 
and $Q_0$ are disjoint, so that the $\OO_{Q^0}$-sheaf $F|_{Q^0}$ is locally free and, therefore, $p:Y^0=
\P(F|_{Q^0})\to Q^0$ is a locally trivial $\p1$-bundle.
Therefore, by Serre’s theorem, for sufficiently large natural numbers $m,s,r$
there is an epimorphism $e:\ \LL_0:=(L^{\otimes-s}|_{Y^0})^{\oplus r}\twoheadrightarrow
\widetilde{E}^0$. Since $\overline{E}^0$ is a locally free $\OO_{X^0}$-sheaf
on a smooth divisor $X^0$ in a smooth variety $Y^0$, then, considered as
$\OO_{Y^0}$-sheaf, $\widetilde{E}^0$ has homological dimension 1, so
$\LL_1:=\ker e$ is a locally free $\OO_{Y^0} $-sheaf. Thus, there is an exact triple
\begin{equation}\label{triple 27}
	0\to\LL_1\to\LL_0\xrightarrow{e}\widetilde{E}^0\to0.
\end{equation}
From the definition \eqref{defn of Q0 etc} of the curve $C^0$ it follows, that
\begin{equation}\label{tilde E fiberwise}
	\widetilde{E}^0|_{\sigma^{-1}(x)}\cong\left\{
	\begin{array}{cc}
		\OO_{\p1}^{\oplus2}, \ \ \ &\ \ \ \ \ \ \ x\in C^0\smallsetminus G_C,\\
		\OO_{\p1}(1)\oplus\OO_{\p1}(-1), & x\in G_C,\\
		\mathbf{k}^2_{\sigma^{-1}(x)}, &\ \ \ \ \ \ \ x\in Q^0\smallsetminus C^0,
	\end{array}\right.
	\chi(\widetilde{E}^0|_{\sigma^{-1}(x)})=2.
\end{equation}
Therefore, applying \cite[Prop. 1.13]{ESt} to the projection $\sigma:X^0\to Q^0$, we obtain that
\begin{equation}\label{E0 loc free}
	\overline{E}^0|_{Q^0\smallsetminus G_C} \ \textit{is a locally free}\ {\OO_{Q^0\smallsetminus
			G_C}-}\textit{sheaf}.
\end{equation}
and
\begin{equation}\label{eqn 29}
	\widetilde{E}^0|_{X^0\smallsetminus \sigma^{-1}(G_C)}=\sigma^*(\overline{E}^0|_{Q^0
		\smallsetminus G_C}),\ \ \ \ \ \ \textrm{Supp}(R^1\sigma_*\widetilde{E}^0)\subset G_C.
\end{equation}
Since $x\in Q^0$, the fiber $p^{-1}(x)$ is isomorphic to $\p1$, and
\begin{equation}\label{L0 fiberwise}
	\LL_0|_{p^{-1}(x)}\cong\OO_{\p1}(-s)^{\oplus r},
\end{equation}
so $p_*\LL_0=0$, and taking into account the relative 
Serre duality for the locally trivial $\p1$-bundle 
$p:Y^0\to Q^0$ we obtain that
\begin{equation}\label{rk L0}
	\begin{split}
		& \MM_0=R^1p_*\LL_0 \ \textit{is a locally free sheaf of rank}\\
		& \rk\MM_0=h^1(\LL_0|_{p^{-1}(x)})=-\chi(\LL_0|_{p^{-1}(x)})=r( s-1), \ \ \ \ \
		\ \ \ \ x\in Q^0.
	\end{split}
\end{equation}
Restricting the exact triple \eqref{triple 27} to the fiber $p^{-1}(x)$, we obtain an exact triple
\begin{equation}\label{triple fiberwise}
	0\to\LL_1|_{p^{-1}(x)}\to\OO_{\p1}(-s)^{\oplus r}\to\widetilde{E}^0|_{\sigma ^{-1}(x)}
	\to0.
\end{equation}
From \eqref{tilde E fiberwise} and \eqref{triple fiberwise} it follows that
$\chi(\LL_1|_{p^{-1}(x)})=\chi(\LL_0|_{p^{-1}(x)})-\chi(\widetilde{E}^ 0|_{\sigma^{-1}
	(x)})=-r(s-1)-2$, and by analogy with \eqref{rk L0} we find that
\begin{equation}\label{rk L1}
	\begin{split}
		& \MM_1=R^1p_*\LL_1 \ \textit{is a locally free sheaf of rank}\\
		& \rk\MM_1=h^1(\LL_1|_{p^{-1}(x)})=-\chi(\LL_1|_{p^{-1}(x)})=r( s-1)+2, \ \ \ \
		\ \ \ \ \ x\in Q^0.
	\end{split}
\end{equation}
As a consequence, applying the functor $R^ip_*$ to \eqref{triple 27}, we obtain an exact
subsequence
\begin{equation}\label{exact 4ple}
	0\to\overline{E}^0\to\MM_1\to\MM_0\to\kappa\to0,\ \ \ \ \
	\kappa:=R^1\sigma_*\widetilde{E}^0,
\end{equation}
and in view of \eqref{eqn 29} we have the inclusion
\begin{equation}\label{Supp kappa}
	\textrm{Supp}(\kappa)\subset G_C.
\end{equation}
Since $\MM_0$ and $\MM_1$ are locally free sheaves on a smooth three-dimensional
variety $Q^0$, it follows (see, for example, \cite[Prop. 1.1]{H}) that
\begin{equation}\label{E0 refl}
	\overline{E}^0 \ \textit{is a reflexive}\ \OO_{Q^0}-\textit{sheaf}.
\end{equation}
On the other hand, again due to the local freeness of $\MM_0$ and $\MM_1$, applying to
\eqref{exact 4ple} the functor $\lExt_{\OO_{Q^0}}^1(-, \OO_{Q^0})$ and using the fact that
$\kappa$ is a sheaf of dimension $\le0$, we obtain an isomorphism of sheaves
\begin{equation}\label{Ext1=Ext3}
	\lExt_{\OO_{Q^0}}^1(\overline{E}^0,\OO_{Q^0})\cong\lExt_{\OO_{Q^0}}^3(
	\kappa,\OO_{Q^0}).
\end{equation}
Since by virtue of \eqref{Supp kappa} the sheaf $\kappa$ is either zero or an artinian sheaf, that is $\kappa$ has a finite filtration with factors - residue fields
$\mathbf{k}_x$ of points $x\in\textrm{Supp}(\kappa)$. Since $\lExt_{\OO_{Q^0}}^3
(\mathbf{k}_x,\OO_{Q^0})\cong\mathbf{k}_x,\ x\in\textrm{Supp}(\kappa),$ therefore
$\chi(\lExt_{\OO_{Q^0}}^3(\kappa,\OO_{Q^0}))=\chi(\kappa)$.
From this and from \eqref{Ext1=Ext3} it follows, that
\begin{equation}\label{Supp kappa=...}
	\textrm{Supp}(\lExt_{\OO_{Q^0}}^1(\overline{E}^0,\OO_{Q^0}))=\textrm{Supp}(\kappa),
\end{equation}
\begin{equation}\label{d=}
	d:=\chi(\lExt_{\OO_{Q^0}}^1(\overline{E}^0,\OO_{Q^0}))=\chi(\kappa).
\end{equation}

Since $\sigma:\widetilde{X}\smallsetminus S\xrightarrow{\cong}Q\smallsetminus C$ is an
isomorphism and $\textrm{Sing}\widetilde{E}\subset\widetilde{X}\smallsetminus S$, then,
setting
$$
\overline{E}:=\sigma_*\widetilde{E},
$$
we have
\begin{equation*}
	\sigma(\textrm{Sing}\widetilde{E})=\textrm{Sing}(\overline{E}|_{Q\smallsetminus C})
	\subset\textrm{Sing}\overline{E}.
\end{equation*}
We also obtain an isomorphism of artinian sheaves
\begin{equation}\label{Ext=Ext}
	\sigma_*:\ \lExt_{\OO_{\widetilde{X}}}^1(\widetilde{E},\OO_{\widetilde{X}})
	\xrightarrow{\cong}\lExt_{\OO_{Q}}^1(\overline{E},\OO_{Q})|_{Q\smallsetminus C},
\end{equation}
where
\begin{equation}\label{Supp=Sing}
	\textrm{Supp}(\lExt_{\OO_{Q}}^1(\overline{E},\OO_{Q})|_{Q\smallsetminus C})=
	\textrm{Sing}(\overline{E}|_{Q\smallsetminus C}).
\end{equation}
Moreover, setting $c_3:=c_3(E)$, we have
\begin{equation}\label{c3=...}
	c_3=c_3(\widetilde{E})=\chi(\lExt_{\OO_{\widetilde{X}}}^1(\widetilde{E},\OO_{
		\widetilde{X}}))=\chi(\lExt_{\OO_{Q}}^1(\overline{E},\OO_{Q})|_{Q\smallsetminus C}).
\end{equation}
Here in \eqref{c3=...} the first equality follows from \eqref{tE cap S}, the second equality
is proven in \cite[Prop. 2.6]{H}, and the third equality follows from \eqref{Ext=Ext}.

Put
\begin{equation}\label{Qxi etc}
	Q_{\xi}:=Q\smallsetminus D_{\xi}=Q^0\cup\Sing\overline{E},
	\ \ \ X_{\xi}:=\sigma^{-1}(Q_{\xi}),\ \ \ \ \ \
	\widetilde{E}_{\xi}:=\widetilde{E}|_{X_{\xi}}, \ \ \
	\overline{E}_{\xi}:=\sigma_*\widetilde{E}_{\xi}.
\end{equation}
By construction, $\overline{E}^0=\overline{E}_{\xi}|_{Q^0}$, therefore from
\eqref{Dxi cap empty}, \eqref{E0 loc free}, \eqref{E0 refl} and \eqref{Qxi etc} follows,
that $\overline{E}_{\xi}$ is a reflexive sheaf, locally free at the points of the curve
$C_{\xi}= C\smallsetminus D_{\xi}$, in particular, at the points of the divisor $D_{\xi'}$ on $C$:
\begin{equation}\label{Exi refl etc}
	\overline{E}_{\xi}\  \textit{is a reflexive sheaf, locally free at
		the points of}\ D_{\xi'}.
\end{equation}
Taking now $\xi'$ instead of $\xi$ and assuming by analogy with \eqref{Qxi etc}
\begin{equation}\label{Qxi' etc}
	Q_{\xi'}:=Q\smallsetminus D_{\xi'},
	\ X_{\xi'}:=\sigma^{-1}(Q_{\xi'}),\ \ \ \ \ \
	\widetilde{E}_{\xi'}:=\widetilde{E}|_{X_{\xi'}}, \ \ \
	\overline{E}_{\xi'}:=\sigma_*\widetilde{E}_{\xi'},
\end{equation}
we get similarly to \eqref{Exi refl etc}:
\begin{equation}\label{Exi' refl etc}
	\overline{E}_{\xi'}\  \textit{is a reflexive sheaf, locally free at the points of}\
	D_{\xi}.
\end{equation}
Consider the sheaf
$$
\overline{E}=\sigma_*\widetilde{E}.
$$
From \eqref{tilde E stable} we have:
\begin{equation}\label{bar E stable}
	H^0(\overline{E})=0.
\end{equation}
Note that due to \eqref{Qxi etc} and \eqref{Qxi' etc}
$E_{\xi}=\overline{E}|_{Q_{\xi}}$, $E_{\xi'}=\overline{E}|_{Q_{\xi'}}$. 
From this and from \eqref{Exi refl etc} and \eqref{Exi' refl etc}
we find that $\overline{E}$ is a reflexive sheaf. Moreover, for
any line $l'$ on $X$ that does not intersect a line $l$ we have
an isomorphism $\beta:=(\sigma\circ\delta^{-1}|_l):l'\xrightarrow{\sim}m:=\beta(l')$ 
such that $E|_{l'}\cong\overline{E}|_m$,
hence $c_1(\overline{E})=c_1(\overline{E}|_{m})=c_1(E|_{l'})=0$.
Thus, by virtue of \eqref{bar E stable} $E$ is stable. So,
\begin{equation}\label{bar E refl etc}
	\overline{E}\ \textit{is a stable reflexive sheaf with}\
	c_1(\overline{E})=0.
\end{equation}
Note also that from \eqref{Supp kappa}, \eqref{Qxi etc},
\eqref{Qxi' etc} and the fact that $\overline{E}^0=\overline{E}|_{Q^0}$, it follows that
$\kappa=R^1\sigma_*\widetilde{E}^0=R^1\sigma_*\widetilde{E}$, hence
\begin{equation}\label{Ext1=}
	\lExt_{\OO_{Q}}^1(\overline{E},\OO_{Q})|_{Q_0}=\lExt_{\OO_{Q^0}}^1(\overline{E}^0,
	\OO_{Q^0})
\end{equation}
and due to \eqref{Supp kappa=...} $\textrm{Supp}(\lExt_{\OO_{Q^0}}^1(\overline{E}^0,
\OO_{Q^0}))\subset G_C$. Since $G_C\cap(Q\smallsetminus C)=\varnothing$, then from
\eqref{Ext1=} the equality follows:
\begin{equation}\label{dir sum}
	\lExt_{\OO_{Q}}^1(\overline{E},\OO_{Q})=\lExt_{\OO_{Q^0}}^1(\overline{E}^0,\OO_{
		Q^0})\oplus\lExt_{\OO_{Q}}^1(\overline{E},\OO_{Q})|_{Q\smallsetminus C}.
\end{equation}
Note that due to \cite[Prop. 2.6]{H} and \eqref{bar E refl etc} we have
$\overline{c}_3:=c_3(\overline{E})=\chi(\lExt_{\OO_{Q}}^1(\overline{E},\OO_{Q}))$,
therefore \eqref{d=}, \eqref{c3=...}, \eqref{dir sum} and the equality $\kappa=R^1\sigma_*
\widetilde{E}$ give:
\begin{equation}\label{bar c3=}
	\overline{c}_3=d+c_3,\ \ \ \ \ \ d=\chi(R^1\sigma_*\widetilde{E}).
\end{equation}
Since, obviously, $R^i\sigma_*\widetilde{E}=0,\ i\ge2,$ then the Leray spectral sequence for the morphism $\sigma:
\widetilde{X}\to Q$ gives $\chi(\widetilde{E})= \chi( \overline{E})-\chi(R^1\sigma_*\widetilde{E})$. Hence, taking into account \eqref{bar c3=} and \eqref{chi=chi} we have:
\begin{equation}\label{chi=chi+}
	\chi(\overline{E}))=\chi(E)+\chi(R^1\sigma_*\widetilde{E})=\chi(E)+d.
\end{equation}
Denote
\begin{equation}
	c_2:=c_2(E),\ \ \ \ \ \overline{c}_2:=c_2(\overline{E}).
\end{equation}

Recall the general Riemann--Roch formula for an arbitrary coherent sheaf $\EE$ of rank
$r$ with Chern classes $\tilde{c}_1,\ \tilde{c}_2,\ \tilde{c}_3$ on a smooth projective
three-dimensional Fano variety $X$ with canonical class $\omega_X$:
\begin{equation}\label{RR}
	\chi(\EE)=r\chi(\OO_X)+\frac 16\tilde{c}_1^3+\frac 12(\tilde{c}_3-\tilde{c}_1\tilde{
		c}_2)-\frac 14\omega_X(\tilde{c}_1^2-2\tilde{c}_2)+\frac{\tilde{c}_1}{12}(c_2(\Omega
	_X)+\omega_X^2).
\end{equation}
Applying this formula to the sheaf $E$ on $X=X_5$ (see \eqref{E stable} and, respectively,
to the sheaf $\overline{E}$ on $X=Q$, we get:
\begin{equation}\label{chi E}
	\chi(E)=2+\frac{1}{2}c_3-c_2,
\end{equation}
\begin{equation}\label{chi bar E}
	\chi(\overline{E})=2+\frac{1}{2}\overline{c}_3-\frac{3}{2}\overline{c}_2.
\end{equation}
In addition, according to Theorem \ref{Theorem 3.1}.(4) we have
\begin{equation}\label{bar c3 le}
	\overline{c}_3\le\left\{
	\begin{array}{cc}
		\frac 12\overline{c}_2^2\ \ \ , &\text{if}\ c_2\ \text{is even},\\
		\frac 12(\overline{c}_2^2+1), & \ \ \ \text{if}\ c_2\ \text{is odd}.
	\end{array}\right.
\end{equation}
From \eqref{chi=chi+}, \eqref{chi E} and \eqref{chi bar E} we find: $\overline{c}_2=
\frac 23c_2-\frac 13d$. Substituting this relation and the equality $c_3=\overline{c}_3-d$, following from \eqref{bar c3=}, in \eqref{bar c3 le}, we obtain the inequality
\begin{equation}\label{bar c3 le}
c_3=\overline{c}_3-d\le\left\{
\begin{array}{cc}
\frac 12(\frac 23c_2-\frac 13d)^2-d, &\text{if}\ c_2\ \text{is even},\\
\frac 12((\frac 23c_2-\frac 13d)^2+1)-d, & \ \ \ \text{if}\ c_2\ \text{is odd}.
\end{array}\right.
\end{equation}
Since $\overline{c}_2=\frac 23c_2-\frac 13d\ge0$ and $d\ge0$, then, assuming in
\eqref{bar c3 le} $d=0$, we obtain the main result of this subsection --- the following theorem.

\begin{theorem}\label{Thm 6.1}
	Let $E$ be a rank 2 stable reflexive sheaf of general type with
	$c_1(E)=0$, $c_2=c_2(E)>0$, $c_3=c_3(E)$ on the variety $X=X_5$. Then
	the following inequalities are valid:
	\begin{equation}\label{c3 le}
		c_3\le\left\{
		\begin{array}{cc}
			\frac 29c_2^2,\ \ &\text{if}\ c_2\ \text{is even},\\
			\frac 29c_2^2+\frac 12, & \ \ \ \text{if}\ c_2\ \text{is odd}.
		\end{array}\right.
	\end{equation}
\end{theorem}
\subsection{Rank two stable reflexive sheaves of general type on $X_4$}
In this subsection everywhere $X=X_4$.

Let us list some known facts about the variety $X$ (see, for example,
\cite[Theor. 4.2(iii)]{Isk}).\\
1) The base $B=B(X)$ of the family of lines on $X$ is isomorphic to the Jacobian $J(C)$ of a smooth curve $C$
of genus 2.\\
2) For an arbitrary line $l\in B$ the set $B_l=\{l'\in B\ |\ l'\cap
l\ne\varnothing\}$ is a curve in $B$ isomorphic to the curve $C$.\\
3) For an arbitrary line $l\in B$ or $N_{l/X}\cong\OO_{\p1}^{\oplus2}$ (general
case), or $N_{l/X}\cong\OO_{\p1}(1)\oplus\OO_{\p1}(-1)$.\\
4) Let $\varphi:\ X\dasharrow \p4$ be the linear projection of the variety $X\subset\p5$
to the space $\p3$ from a line $l\subset X$ for which $N_{l/X}\cong\OO_{\p1}
^{\oplus2}$. Then $\varphi(X)=\p3$, the morphism $\varphi: X\dasharrow\p3$ is birational and
decomposes into a Hironaka's roof
\begin{equation}\label{eqn 67}
\xymatrix{
& \ar[dl]_{\delta} \ar[dr]^{\sigma}\widetilde{X} & \\
X \ar@{-->}[rr]^-{\varphi} & & Q.}
\end{equation}
Here $\delta^{-1}:\ X\dasharrow\widetilde{X}$ is the 
blow-up of $X$ centered at $l$, $S=\delta^{-1}(l)\simeq 
\p1 \times\p1$, $\sigma|_S:S\xrightarrow{\simeq}\sigma(S) 
=Q_2$ is an isomorphism of $S$ onto a smooth 
two-dimensional quadric $Q_2\subset\p3$, and $\sigma^{-1} :\ Q\dasharrow\widetilde{X}$ is the blow-up of $Q$ centered in a smooth curve $C$ of type (2,3) on $Q_2$, so that $\widetilde{X}\simeq\p(\II_{C,\p3})$, which is isomorphic to the curve $C$ from the statement 1) above. In particular, the triple $0\to N_{C,Q_2}\to N_{C/\p3}\to N_{Q_2/\p3}|_{C}\to0$ is exact,
where, due to the isomorphism $Q_2\simeq\p1\times\p1$, we obtain $N_{C/Q_2}\cong \OO_C(D_I)$,
$\deg D_I=12$, $N_{Q_2/\p3}|_{C}\cong\OO_{\p3}(2H)|_{C}\cong\OO_C(D_{II})$, $ \deg
D_{II}=10$, hence $\det N_{C/\p3}\cong\OO_C(D_I+D_{II})$. Note that the divisor
$D_I+D_{II}$ on $C$ is very ample, since $C$ is a curve of genus 2.\\
5) Let $\Gamma=\{(x,l)\in X\times B\ |\ x\in l\}$ be an incidence graph with projections
$X\xleftarrow{p_1}\Gamma\xrightarrow{p_2}B$. It is known \cite[Prop. 5.2]{Isk} that
$p_1:\Gamma\to X$ is a finite at a general point morphism and $\dim\{x\in X\ |\ \dim p_1^{-1}(x)\ge1\}\le0$.

Consider an extension
\begin{equation}\label{xi on C}
	\xi:\ 0\to\OO_{\p3}(nH)\to F\xrightarrow{\varepsilon}\II_{C,\p3}\to0.
\end{equation}
The sheaf $\lHom(\II_{C,\p3},\OO_{\p3}(nH))\cong\OO_{\p3}(nH)$ on $\p3$ and the sheaf $\MM=\lExt^1(\II_{C,\p3}, \\ \OO_{\p3}(nH))\cong\lExt ^2(\OO_C,\OO_{\p3}(nH))\cong\det N_{C/ \p3}
\otimes\OO_{\p3}(nH)\cong\OO_C(D_I+D_{II}+nH)$ on $C$ with $n\gg0$ are very ample.
Therefore, we can assume that:\\
1) $H^i(\lHom(\II_{C,\p3},\OO_{\p3}(nH)))=0,\ i=1,2$, which means that the long exact
sequence of $\Ext$-groups for the pair of sheaves $\II_{C,\p3},\OO_{\p3}(nH)$ gives
isomorphisms $\Ext^1(\lHom(\II_{C,\p3},\OO_{\p3}(nH))\cong H^0(\lExt^1(\II_{C,\p3},
\OO_{\p3}(nH))\cong H^0(\lExt^2(\OO_C,\OO_{\p3}(nH))=H^0(\MM)$;\\
2) the triple \eqref{xi on C} as an extension $\xi\in\Ext^1(\lHom(\II_{C,\p3},\OO_Q(nH))=H
^0(\MM)$ is a section of the sheaf $\MM$, the scheme of zeros $(\xi)_0$ of which is reduced, i.e.
it is a simple divisor $D_{\xi}=x_1+...+x_{5n+22}$. This means by
Serre's construction (see \cite[Section 4]{H}) that in \eqref{xi on C} the sheaf $F$ is
a reflexive sheaf of rank two with $\Sing F=D_{\xi}$ and has the simplest singularities in
the points $x_i$, i.e. $\lExt^1(F,\OO_Q)\cong\oplus_{i=1}^{3n+6}\mathbf{k}_{x_i}$. Therefore, according to \cite[Cor. 2.8]{BW}, the projective spectrum $Y=\P(F)$ of the sheaf $F$ is
a smooth variety. The epimorphism $\varepsilon$ in \eqref{xi on C} corresponds to
an embedding $i$ of the divisor $\widetilde{X}=\P(\II_{C,\p3})$ into $Y$, and let $p:Y\to\p3$ be the
natural projection, so $\sigma=p\circ i$. Note also that the sheaf
\begin{equation}\label{sheaf L}
	L=\OO_{Y/\p3}(1)\otimes p^*\OO(mH)
\end{equation}
is ample for $m\gg0$, where $\OO_{Y/\p3}(1)$ is the Grothendieck sheaf on $Y=\P(F)$.

Let
\begin{equation*}\label{E stable on X4}
	\begin{split}
		& E\ \textit{be a rank 2 stable reflexive sheaf of general type with}\ c_1(E)=0\\
		& \textit{on}\ X_4;\ \textit{in particular}\ H^0(E)=0.
	\end{split}
\end{equation*}

Repeating for the sheaf $E$ verbatim all the arguments from
\eqref{def tilde E} -- \eqref{bar c3 le} with $Q$ and $Q^0$ replaced by $\p3$
and $(\p3)^0$ respectively, and using Hartshorne’s result for
stable reflexive sheaves on $\p3$ (see \cite[Thm. 8.2(b)]{H}),
we obtain the following analogue of theorem \ref{Thm 6.1} for $X=X_4$.
\begin{theorem}\label{Thm 6.4}
	Let $E$ be a rank 2 stable reflexive sheaf of general type with
	$c_1(E)=0$, $c_2=c_2(E)>0$, $c_3=c_3(E)$ on the variety $X=X_4$. Then
	the following inequality is true:
	\begin{equation}\label{c3 le for X4}
		c_3\le c_2^2-c_2+2.
	\end{equation}
	
\end{theorem}

\vspace{5mm}

\noindent
Alexander~S.~Tikhomirov\\
Faculty of Mathematics\\
National Research University\\
Higher School of Economics\\
Usacheva St., 6\\ 
119048 Moscow, Russia\\
\textit{E-mail}:{\ astikhomirov@mail.ru}

\vspace{3mm}

\noindent
Danil A.~Vassiliev\\
Faculty of Mathematics\\
National Research University\\
Higher School of Economics\\
Usacheva St., 6\\ 
119048 Moscow, Russia\\
\textit{E-mail}:{\ danneks@yandex.ru}

\end{document}